\documentclass[11pt]{amsart}

\usepackage[top=1in, bottom=1in, left=1in, right=1in]{geometry}

\usepackage{graphicx}
\usepackage[dvipsnames]{xcolor}
\usepackage{amsfonts,amsmath,amssymb,amsthm,bbm,centernot}
\usepackage{mathrsfs}
\usepackage{enumerate}
\usepackage[normalem]{ulem}
\usepackage{comment,constants}
\usepackage{hyperref}
\usepackage{mathtools}
\usepackage{multirow}

\newtheorem{theorem}{Theorem}[section]
\newtheorem{proposition}[theorem]{Proposition}
\newtheorem{lemma}[theorem]{Lemma}
\newtheorem{corollary}[theorem]{Corollary}

\theoremstyle{remark}
\newtheorem{remark}[theorem]{Remark}

\theoremstyle{definition}
\newtheorem{definition}{Definition}[section]

\newcommand{\N}{\mathbb N}     
\newcommand{\R}{\mathbb R}     
\newcommand{\Z}{\mathbb Z}     

\renewcommand{\epsilon}{\varepsilon}
\newcommand{\del}{\partial}

\renewcommand{\kappa}{\varkappa}

\newcommand{\B}{\mathfrak{B}}
\newcommand{\D}{\mathfrak{D}}

\newcommand{\fl}[1]{\lfloor #1 \rfloor}  
\newcommand{\ind}[1]{ \mathbbm{1}_{\{ #1 \}} }
\newcommand{\indf}[1]{ \mathbbm{1}_{#1} }

\definecolor{dgr}{rgb}{0.0, 0.50, 0.15}

\definecolor{dblue}{rgb}{0.0, 0.0, 0.55}

\DeclareMathOperator{\Var}{Var}
\DeclareMathOperator{\Cov}{Cov}
\DeclareMathOperator{\sgn}{sign}
\let\emptyset\varnothing

\title[Self-interacting Random Walks]{Convergence and non-convergence of scaled self-interacting random walks to Brownian motion perturbed at extrema}

\author{Elena Kosygina}
\address{Elena Kosygina\\One Bernard Baruch Way \\ Department of Mathematics, Box B6-230 \\ Baruch College \\ New York, NY 10010 \\ USA}
\email{elena.kosygina@baruch.cuny.edu}
\urladdr{http://www.baruch.cuny.edu/math/elenak/}
\thanks{E.\,Kosygina is partially supported by the Simons Foundation through a Collaboration Grant for Mathematicians \#523625.}

\author{Thomas Mountford}
\address{Thomas Mountford\\\'Ecole Polytechnique F\'ed\'eral de Lausanne\\Department of Mathematics\\EPFL SB MATH PRST\\
MA B1 517 (Bâtiment MA)
Station 8
CH-1015 Lausanne\\
Switzerland}
\email{thomas.mountford@epfl.ch}
\urladdr{http://people.epfl.ch/thomas.mountford}
\thanks{T.\,Mountford was partially supported by the Swiss National Science Foundation, grant FNS 200021L 169691.}

\author{Jonathon Peterson}
\address{Jonathon Peterson\\Purdue University\\Department of Mathematics\\150 N University Street\\West Lafayette, IN  47907\\USA}
\email{peterson@purdue.edu}
\urladdr{http://www.math.purdue.edu/$\sim$peterson}
\thanks{J.\,Peterson was partially supported by the Simons Foundation through a Collaboration Grant for Mathematicians \#635064.}



\begin{document}

\begin{abstract}
  We use generalized Ray-Knight theorems introduced by B.\,T\'oth in
  \cite{tGRK} together with techniques developed for excited random
  walks as main tools for establishing positive and negative results
  concerning convergence of some classes of diffusively scaled
  self-interacting random walks (SIRWs) to Brownian motions perturbed
  at extrema (BMPE). The work \cite{tGRK} studied two classes of
  SIRWs: asymptotically free and polynomially self-repelling walks.
  For both classes B.\,T\'oth has shown, in particular, that the
  distribution function of a scaled SIRW observed at independent
  geometric times converges to that of a BMPE indicated by the
  generalized Ray-Knight theorem for this SIRW. The question of weak
  convergence of one-dimensional distributions of scaled SIRW remained
  open. In this paper, on the one hand, we prove a full functional
  limit theorem for a large class of {\em asymptotically free} SIRWs which
  includes asymptotically free walks considered in \cite{tGRK}. On the
  other hand, we show that rescaled {\em polynomially self-repelling}
  SIRWs do not converge to the BMPE predicted by the
  corresponding generalized Ray-Knight theorems and, hence, do not
  converge to any BMPE.
\end{abstract}

\maketitle

\section{Introduction and main results}

This work has as its starting point the paper of B.\,T\'oth
  \cite{tGRK}, which along with \cite{tTSAWGBR} and \cite{tTSAW}
  greatly expanded the domain of processes to which the classical
  Ray--Knight approach could be applied. 
  It introduced two families of nearest neighbor self-interacting
  random walks: the asymptotically free random walks and the
  polynomially self-repelling random walks.  This approach, relying on
  tree structures and so inapplicable to higher dimensions,
  established correct scaling and then convergence in distribution of
  appropriately scaled hitting times of the random walks.  It was also
  very close to establishing the convergence in law of one point
  distributions in that it identified the limiting law were such a
  limit to exist.

  B.\,T\'oth also noted that the Ray--Knight theorems for the
  asymptotically free and polynomially self repelling random walks
  were the same as analogous results for a previously studied
  continuous model, in this case the Ray--Knight theorems for Brownian
  motions perturbed at extrema (BMPE)\footnote{See
    \cite[Section3]{cpyBetaPBM}.} (or multiples thereof) with
    parameters determined by the particularities of the original
    model.
    This raised the possibility of finding a general result
  whereby a functional limit theorem for the rescaled walk could
  simply be proven via establishing the Ray--Knight theorems and
  minimal technical conditions.      
  A number of results for excited random walks,
  \cite{dkSLRERW,kpERWPCS,kmpCRW}, weighed positively on this
  possibility except that the ``technical conditions'' remained
  elusive and the treatment varied significantly from model to model.

  In this article we show both the
  usefulness and the limitations of the Ray--Knight approach to establishing
  scaling limits.  On the one hand, we prove that the rescaled
  asymptotically free random walks do indeed converge to a BMPE.  On the
  other hand, we show that the polynomially self-repelling random
  walks do not converge to the natural BMPE limit suggested by the
  Ray--Knight theorems.  This nonconvergence
  result is important as it shows that there is, in fact, no
  general theorem that permits easy passage from Ray--Knight theorems
  to functional convergence, and opens up the question of what natural
  conditions on the original processes ensure functional convergence.
  It also motivates the study of families of processes having the same
  Ray--Knight behavior.  

\subsection{Model description} We consider a discrete time nearest
neighbor self-interacting random walk $(X_i)_{i\ge 0}$ on $\Z$ which
starts at $X_0=0$ and at times $i\in\N$ jumps to one of the two
nearest neighbors: if $X_{i-1}=x$ then $X_i$ is equal to $x\pm 1$ with
probabilities dependent on the numbers of crossings of undirected
bonds $\{x,x\pm1\}$ prior to $i$. More precisely, let
$\Omega=\{\omega=(\omega_i)_{i\ge 0}:\ \omega_0=0,\
|\omega_i-\omega_{i-1}|=1\ \forall i\in\N\}$ be a set of all nearest
neighbor paths originating at 0, $\mathcal{F}_i$, $i\ge 0$, be the
$\sigma$-algebra generated by all subsets of $\Omega$ of the form
$\{\omega_j=x_j,\,\forall j\in\{0,1,\dots,i\}\}$, and
$\mathcal{F}=\sigma\left(\cup_{i=0}^\infty \mathcal{F}_i\right)$.  We set
$\ell_x^0=r_x^0=0$ for all $x\in\Z$ and for each $\omega\in\Omega$
define
\begin{equation}\label{RL}
  r_x^i(\omega):=\sum_{j=1}^i\ind{x,x+1}(\{\omega_{j-1},\omega_j\});\quad \ell_x^i:=\sum_{j=1}^i\ind{x-1,x}(\{\omega_{j-1},\omega_j\}),\quad i\in\N,
\end{equation}
where for two sets $A$ and $S$, $\mathbbm{1}_A(S)=1$ if $S\subseteq A$
and $0$ otherwise. We assume that under the probability measure $P$ on
$(\Omega, \mathcal{F})$ the self-interacting random walk (SIRW)
$X=(X_i(\omega))_{i\ge 0}$ has the following dynamics: $X_0=0$,
\begin{equation}
  \label{sirw}
  P(X_{i+1}=X_i+1\mid {\mathcal F}_i)=1-P(X_{i+1}=X_i-1\mid {\mathcal F}_i)=\frac{w(r_{X_i}^i)}{w(r_{X_i}^i)+w(\ell_{X_i}^i)},\quad i\ge 0,
\end{equation}
where $w:\N_0\to(0,\infty)$ is a given weight function.  Properties
of $w$ (in particular, its asymptotics at infinity) determine the type
of SIRW and its long term behavior.  If the weight function $w$ is
non-increasing then the SIRW is said to be {\em self-repelling}. If
$w$ is non-decreasing then the SIRW is said to be {\em
  self-attracting}.

Fix $\alpha\ge 0$, $p\in(0,1]$, $B\in\R$, $\kappa>0$ and let
$w:\N_0\to(0,\infty)$ satisfy
\begin{equation}
  \label{w}
  \frac{1}{w(n)}=n^\alpha\left(1+\frac{2^pB}{n^p}+O\left(\frac{1}{n^{1+\kappa}}\right)\right)\quad\text{as $n\to\infty$}.
\end{equation}
This model was introduced and studied by B.\,T\'oth in
\cite{tGRK}. According to the terminology of \cite{tGRK}, setting
$\alpha=0$ in \eqref{w} places the model in the {\em asymptotically
  free} regime: $w(n)\sim 1$, while $\alpha>0$ corresponds to the {\em
  polynomially self-repelling} case: $w(n)\sim n^{-\alpha}$.  In
\cite{tGRK}, B.\,T\'oth proved functional limit theorems for the local
time processes of SIRWs (a.k.a.\ generalized Ray--Knight Theorems) and
local limit theorems for the position of a random walker observed at
independent geometric times with linearly growing means
(\cite[Theorems 2A, 2B]{tGRK}). These local limit theorems imply that
if one-dimensional distributions of rescaled SIRWs converge weakly,
then the limiting distribution must be that of (a multiple of) a
Brownian motion perturbed at its extrema, see
Definition~\ref{bmpe}. Even though the original paper \cite{tGRK}
considered only $p=\kappa=1$, the relevant results of that work can be
extended to $p\in(0,1]$ and $\kappa>0$. The question of weak
convergence (not just along an independent geometric sequence of
times) of one-dimensional distributions of the rescaled position of
the walk to the conjectured limit remained open and has motivated our
work.

\subsection{Main results}
\begin{definition}\label{bmpe}
  Given $\theta^+,\theta^-<1$, a BMPE
  $W^{\theta^+,\theta^-}=(W^{\theta^+,\theta^-}(t),t\ge 0)$ with
  parameters $(\theta^+,\theta^-)$ is the pathwise unique solution of
  the equation
\begin{equation*}
  W(t)=B(t)+\theta^+\sup_{s\le t}W(s)+\theta^-\inf_{s\le t}W(s),\quad t\ge 0,\quad W(0)=0,
\end{equation*}
where $(B(t))_{t\ge 0}$ is a standard Brownian motion. 
\end{definition} 
It is known that this solution is adapted to the filtration of $(B(t))_{t\ge 0}$,
has Brownian scaling\footnote{$(cW^{\theta^+,\theta^-}(c^{-2}t))_{t\ge 0}$ has the
  same distribution as $(W^{\theta^+,\theta^-}(t))_{t\ge 0}$ for all
  $c>0$.}, and the triple
$(\inf_{s\le t}W^{\theta^+,\theta^-}(s),\,W^{\theta^+,\theta^-}(t),\,\sup_{s\le t}W^{\theta^+,\theta^-}(s)),\ t\ge 0$, is a strong
Markov process. (See \cite{pwPBM,cpyBetaPBM, cdPUPBM, Dav99} and
references therein.)

B.\,Davis (\cite{Dav96,Dav99}) has shown that BMPEs arise naturally as
diffusive limits of perturbed random walks, so-called $pq$-walks. More
specifically, consider a nearest neighbor random walk which starts
with an unbiased jump from the origin and continues to make unbiased
jumps to one of its neighbors except when it visits an extremal point
of its current range. If at time $n\in\N$ the walk is at its maximum
up to time $n$ then the probability that it jumps to the right in the
next step is $p:=(2-\theta^+)^{-1}$. Similarly, if at time $n\in\N$
the walk is at its current minimum then the probability that it jumps
to the right in the next step is
$q:=1-(2-\theta^-)^{-1}$. \cite[Theorem 1.2]{Dav96} states that
linearly interpolated and rescaled perturbed random walks converge
weakly to the process $W^{\theta^+,\theta^-}$.\footnote{This result
  was shown in \cite{Dav96} under the additional assumption
  $|\theta^+\theta^-|<(1-\theta^+)(1-\theta^-)$ which was
  removed in \cite[Theorem 4.8]{Dav99}.} We note that
when $\theta^+=\theta^-=\theta$ the perturbed walk is a special case
of SIRW with $w(0)=(2-\theta)^{-1}$ and $w(n)=1$ for $n\in\N$. Hence,
the idea that a class of asymptotically free SIRWs with much more
general weight functions $w$ might have BMPEs as scaling limits seems
plausible but not very intuitive: while it is not difficult to see
that the range of a SIRW after $n$ steps is of order $\sqrt{n}$, it is
not at all clear why after the diffusive scaling and away from the
extrema the SIRW should behave essentially as a Brownian motion.

BMPEs also appear as scaling limits of other non-markovian random
walks on $\Z$ such as once reinforced random walks
(\cite{Dav90,Dav96}), excited random walks
(\cite{dCLTERW,dkSLRERW,kpERWPCS, kmpCRW}), and rotor walks with
defects (\cite{HLSH18}).  A few of the above models are of the type
where the (recurrent\footnote{In the sense that the walk visits each $x\in\Z$ infinitely many times with probability 1.}) random walk after visiting a site ``in the
bulk'' (i.e., away from the boundary of its current range) a certain
fixed number of times makes only unbiased steps from that
site\footnote{E.g., this is the case for once
  reinforced random walks, $pq$-walks,
  and excited random walks with
  bounded cookie stacks (\cite{dkSLRERW}).}, and for these models
the convergence to a BMPE seems very intuitive. On the other hand, for
excited random walks with Markovian cookie stacks the work
\cite{kmpCRW}) established the convergence to a multiple of BMPE, even
though the Brownian behavior ``in the bulk'' is not seen at the random
walk level.  Establishing this behavior is non-trivial and requires
intermediate ``mesoscopic'' coarse graining of space and time.

Our first main result is the functional limit theorem for the
asymptotically free case $\alpha=0$.
Following the notation in \cite{tGRK}, define
\begin{equation}\label{UVr}
  U_1(n):=\sum_{j=0}^{n-1}(w(2j))^{-1}\quad\text{and}\quad V_1(n):=\sum_{j=0}^{n-1}(w(2j+1))^{-1}.
\end{equation}
Set
\begin{equation}
  \label{v-u}
  \gamma:=\lim_{n\to\infty}(V_1(n)-U_1(n)).
\end{equation}
Note that if $w$ satisfies \eqref{w} with $\alpha=0$ and $p,\kappa>0$
then the above limit always exists. This $\gamma$ is the main
parameter which identifies the limiting process in
Theorem~\ref{thmAFC}. It is related to the
parameter $\delta$ in \cite{tGRK} by the equation $\gamma=1-\delta/2$.

It is easy to show that if $w$ is monotone, then $\gamma<1$. More
precisely, if $w$ is non-increasing (self-repelling case), then
$\gamma\in[0,1)$, and if $w$ is non-decreasing (self-attracting case),
then $\gamma\le 0$.  For reader's convenience we reproduce the
argument given in \cite[p.\,1340]{tGRK}. If $w$ is non-increasing, then
$1/w$ is non-decreasing and
\[\gamma=\lim_{n\to\infty}(V_1(n)-U_1(n))=\sum_{j=0}^\infty\left(\frac{1}{w(2j+1)}-\frac{1}{w(2j)}\right)\ge
  0.\] On the other hand, $U_1(n+1)-U_1(n)=(w(2n))^{-1}\to 1$ as
$n\to\infty$, and we also get
\begin{align*}
 \gamma=1+\lim_{n\to\infty}(V_1(n)-U_1(n+1))&=1-\frac{1}{w(0)}-\lim_{n\to\infty}\sum_{j=0}^{n-1}\left(\frac{1}{w(2j+2)}-\frac{1}{w(2j+1)}\right)\\ &=1-\frac{1}{w(0)}-\sum_{j=1}^\infty\left(\frac{1}{w(2j)}-\frac{1}{w(2j-1)}\right)<1. 
\end{align*}
If $w$ is non-decreasing, then $V_1(n)-U_1(n)\le 0$ for all $n$ and
$\gamma\le 0$.

\begin{theorem}\label{thmAFC}
  Let $w: \N_0\to (0,\infty)$ be monotone and satisfy \eqref{w} with
  $\alpha=0$, $p\in(1/2,1]$, and $\kappa>0$. Consider a SIRW
  $(X_i)_{i\ge 0}$ defined in \eqref{sirw} with $X_0=0$. Then
\[\left(\frac{X_{\fl{nt}}}{\sqrt{n}}\right)_{t\ge 0}\ \ \Longrightarrow \ \ \left(W^{\gamma,\gamma}(t)\right)_{t\ge 0}\quad \text{as $n\to\infty$}\] in the standard Skorokhod topology on $D([0,\infty))$.
\end{theorem}

\begin{remark}
  Note that the case $p>1$ is covered by \eqref{w} with $B=0$.  The
  restriction to $p>\frac12$ guarantees that the series representing
  the ``total drift'' at a single site (see \eqref{dxdef} and Lemma
  \ref{lem:bdxfin}) converges absolutely. We think that the case
  $p\in(0,1/2]$ can be considered as well but will certainly require a
  different and technically more involved treatment such as
    coarse graining in the spirit of \cite{kmpCRW}.
\end{remark}

\begin{remark}
  Following \cite{tGRK} we have assumed that $w$ is
  monotone. If $w$ is not monotone (and $\alpha=0$), then $\gamma$ can
  be any real number.  Monotonicity assumption on $w$ naturally
  restricts the parameter $\gamma $ to ``recurrent'' values.  Our
  arguments, provided that $\gamma<1$, do not require monotonicity
  except for the results which we quote from \cite{tGRK}. However, we
  believe that one could treat non-monotonic $w$ with a slightly
  different method and that for $\gamma<1$ the results will remain the
  same.  The removal of the monotonicity assumption would permit
  arbitrary values of $\gamma$ including ``transient'' values (i.e.,
  $\gamma > 1$).  In this case the approach of \cite{kmLLCRW} would
  likely give results of the character of those found in
  \cite{kmLLCRW} and \cite[Section 6]{kzERWsurvey}. We leave these
  questions to an interested reader.
\end{remark}

Our second result concerns the polynomially self-repelling  case
$\alpha>0$. This negative result is rather unexpected and, for this
reason, is, in the authors' opinion, particularly interesting. Let
\begin{equation}
  \label{wlim}
  W_\alpha(t):=(2 \alpha+1)^{1/2}W^{1/2,1/2}(t),\quad t\ge 0,
\end{equation}
where $W^{1/2,1/2}$ is a BMPE with $\theta^+=\theta^-=1/2$ (see
Definition~\ref{bmpe}), and recall two facts which make the process
$W_\alpha$ the natural candidate for a weak limit of
$X^{(n)}:=(X^{(n)}(t))_{t\ge
  0}:=\left(n^{-1/2}X_{\fl{nt}}\right)_{t\ge 0}$ as $n\to\infty$.
\begin{enumerate}[(i)]
\item The generalized Ray--Knight theorems for the sequence of
  processes $(X^{(n)})_{n\in \N}$, \cite[Theorem 1B, Corollary
  1B]{tGRK}, correspond to Ray--Knight theorems for $W_\alpha$ (see
  \cite[Theorems 3.1, 3.4]{cpyBetaPBM}).
\item If for a fixed $t>0$ the sequence $(X^{(n)}(t))_{n\in\N}$
  converges in distribution, then the limiting distribution must be
  that of $W_\alpha(t)$ (see \cite[Theorem 2B, (3.2.9), (3.2.11), and
  Remark on p.\,1334]{tGRK}).
\end{enumerate}
These facts identify the process $W_\alpha$ as the only possible weak
limit of $(X^{(n)})_{n\in\N}$ in the set
$\{cW^{\theta^+,\theta^-},\ \theta^+\vee \theta^-<1,\ c\in\R\}$ of all
scalar multiples of BMPEs.  Moreover, our first main result,
Theorem~\ref{thmAFC}, shows that the BMPEs ``predicted'' by exactly
the same information as in (i), (ii) are bona fide weak limits for a
class of asymptotically free SIRWs.

\begin{theorem} \label{thmPSR} Let $\alpha>0$,
  $(W_\alpha(t))_{t\ge 0}$ be given by \eqref{wlim}, and
  $w(n)=(n+1)^{-\alpha}$ for all $n\in \N_0$. Consider a SIRW
  $(X_i)_{i\ge 0}$ defined in \eqref{sirw} with $X_0=0$. Then
  \[\left(\frac{X_{\fl{nt}}}{\sqrt{n}}\right)_{t\ge 0}\ \
    \centernot\Longrightarrow\ \ \left(W_\alpha(t)\right)_{t\ge 0}\quad \text{as $n\to\infty$}\] in the standard Skorokhod
  topology on $D([0,\infty))$.
\end{theorem}

\begin{remark}
  The result presented in Theorem~\ref{thmPSR} and its proof remain
  valid for a more general weight function $w$ given by \eqref{w} with
  $p=\kappa=1$ as in \cite{tGRK}.  We have chosen to write out the
  arguments for $w(n)=(n+1)^{-\alpha},\,n\ge 0$, simply for clarity's
  sake and also as our result is a counterexample to the functional
  convergence of rescaled polynomially self-repelling SIRWs to BMPEs.
\end{remark}

Already at this point we would like to say a
few words as to why in the polynomially self-repelling case we must
reject the only natural candidate $W_\alpha$ as a possible weak limit
of $X^{(n)},\,n\ge 0$. A more detailed discussion is given right after
Proposition ~\ref{lemC1}. In Proposition \ref{lemC1} we study the behavior of
{\em increments} of the local time processes associated to $X^{(n)}$
rather than just the behavior of {\em cumulative} local times which
appear in generalized Ray--Knight theorems. We discover that these
increments scale with a different constant than those of $W_\alpha$
and this violates the {\em additive property} of the limiting local
time processes which is inherent to BMPEs. Since our observations are
in terms of local time processes which are not continuous path
functionals, the proof of this result is rather technical: we have to
approximate local times with occupation times over small intervals and
then take an appropriate subsequence along which the convergence
fails.

\begin{remark}[and an open problem] For the polynomially self-repelling case the
  question of whether $X^{(n)}(t)\Longrightarrow W_\alpha(t)$ for a fixed
  $t>0$ unfortunately remains open.  Provided that one dimensional
  distributions converge, one would also like to study weak
  convergence at the process level and find a viable candidate for a
  weak process limit of $(X^{(n)})_{n\ge 1}$. The last task is
  particularly interesting in view of Theorem~\ref{thmPSR}.
\end{remark}

\subsection{Organization of the paper}

The proof of Theorem~\ref{thmAFC}, except for derivations of several
technical results, is given Section~\ref{two}. In Section~\ref{three}
we prove Theorem~\ref{thmPSR} modulo a key Proposition~\ref{lemC1}
which we state and discuss in Section~\ref{three} but prove only at
the end of Section~\ref{sec:Polya} after a thorough treatment of the
generalized P\'olya urn model of left and right jumps from a single
site. More specifically, Section~\ref{sec:Polya} studies the process
of ``discrepancies'' between the numbers of left and right jumps from
one site and lays a foundation for the technical results of the
previous two sections. Appendix~\ref{app:blp} concerns
  generalized Ray--Knight theorems for SIRWs where we define
  ``branching-like processes'' and adapt as necessary some of the
  results previously obtained in \cite{tGRK}. These results are used
  in Appendix~\ref{app:afc} which contains proofs of
  Proposition~\ref{minmaxtight} and Lemma~\ref{lem:smalllt} from
  Section~\ref{two}. Proofs of two lemmas for the polynomially
  self-repelling case and an auxiliary Lemma~\ref{calc} are given in
  Appendix~\ref{app:psr}.

\subsection{Notation}
Below we gather some of the notation used throughout the paper. 

\begin{itemize}
\item[BESQ$^\delta$] - the square of a Bessel process of generalized dimension
  $\delta\in \R$ (\cite[Definition (1.1), Chapter XI]{ryCMBM} and
\cite[(28)]{gySGBP}). For $\alpha\ge 0$
and $\delta\in\R$ we denote by $Z^{(\alpha,\delta)}$ a BESQ$^\delta$
process divided by $2(2\alpha+1)$. When $\alpha=0$ ({\em
  asymptotically free} case) we shall drop it from the notation and
write simply $Z^{(\delta)}$.
\item[BLP] - branching-like process (Appendix~\ref{app:blp}).
\item[BMPE] - Brownian motion perturbed at its extrema (Definition~\ref{bmpe}).
\item[SIRW] - self-interacting random walk.
\end{itemize}
Weak convergence (denoted by
$\Longrightarrow$) and tightness of stochastic processes in $D([0,T])$
or $D([0,\infty))$ will always be understood with respect to the
standard (i.e., $J_1$) Skorokhod topology.

For a generic stochastic process $(Y_t)_{t\ge 0}$ indexed either by
$t\in\N_0:=\N\cup\{0\}$ or $t\in[0,\infty)$ and $x\in\R$ we define
\begin{equation}
  \label{ent}
  \tau^Y_x:=\inf\{t>0:\,Y_t\ge x\}\quad\text{and}\quad \sigma^Y_x:=\inf\{t>0:\,Y_t\le x\}\quad \text{where } \inf\emptyset :=\infty.
\end{equation}
$P_y^Y$ and $E_y^Y$ will be used to indicate that $Y(0)=y$.

We denote by ${\mathcal L}(x,n)=\sum_{i=0}^n\ind{X_i=x}$ the number of
visits to $x$ by the random walk by time $n$.
The number of upcrossings (resp.\ downcrossings) by  the random walk
of the directed bond $(x,x+1)$, $x\in\Z$, (resp.\ $(x,x-1)$) up to
time $i\in\N$ will be denoted by ${\mathcal E}^i(x)$ (resp.\
${\mathcal D}^i(x)$), i.e.,
\begin{equation}
  \label{upcr}
    {\mathcal E}^i(x)=\sum_{j=0}^{i-1}\ind{X_j=x, X_{j+1}=x+1} ;\quad{\mathcal D}^i(x)=\sum_{j=0}^{i-1} \ind{X_j=x, X_{j+1}=x-1},
\end{equation}
where $\indf{A}$ is the indicator function of set $A$.

Given two probability measures $P_1$ and $P_2$ on the Borel
sigma-field ${\mathcal B}$ of the Skorokhod metric space $(D([0,1],d^\circ)$ (\cite[(12.16)]{bCOPM}), we denote by
  $\text{dist}\,(P_1,P_2)$ the Prokhorov distance between $P_1$ and
  $P_2$(\cite[p.\,72]{bCOPM}):
\[\text{dist}\,(P_1,P_2)=\inf\{\epsilon>0: \forall A\in{\mathcal B},\,
  P_1(A)\le P_2(A^\epsilon)+\epsilon\ \text{and } P_2(A)\le
  P_1(A^\epsilon)+\epsilon\},\] where
$A^\epsilon=\{\omega\in D([0,1]):\  \exists\,\tilde{\omega}\in A,\ d^{\circ}(\omega,\tilde{\omega})<\epsilon\}$. Recall that if $\text{dist}\,(P_N,P)\to 0$ then $P_N\Longrightarrow P$ as $N\to\infty$. In particular, for every $\epsilon>0$ and every continuous bounded functional $F$ on $D([0,1])$ there is an $\eta=\eta(\epsilon,F,P)>0$ such that 
\begin{equation}
  \label{eta}
  \text{dist}\,(P_N,P)<\eta\ \ \text{implies }\ \ \left|\int F\,dP_N-\int F\,dP\right|<\epsilon.
\end{equation}

\section{Asymptotically free case: proof of Theorem~\ref{thmAFC}}\label{two}

For the proof of Theorem \ref{thmAFC} we will decompose the random
walk as $X_n = M_n + \Gamma_n$, where
\begin{equation}\label{Gndef}
  \Gamma_0=0,\quad\Gamma_n = \sum_{i=0}^{n-1} E\left[ X_{i+1}-X_i \mid \mathcal{F}^X_i \right], \quad \mathcal{F}^X_i = \sigma(X_j, \, j\leq i),\quad  n\in\N. 
\end{equation}
Note that with this choice of $\Gamma_n$, the process  $M_n = X_n - \Gamma_n$, $n\in\N_0$, is a martingale with respect to the filtration $\mathcal{F}^X_n$, $n\in\N_0$. The key idea in the proof of Theorem \ref{thmAFC} is that the martingale term in this decomposition will converge to a Brownian motion (Lemma \ref{martpart}) while $\Gamma_n$, which accounts for the accumulated drift experienced by the walk up to time $n$, will be approximated by 
a linear combination of the running minimum and maximum of the walk. 

A similar strategy as the one we employ here has  been used previously  for other self-interacting random walks (see \cite{dCLTERW,dkSLRERW,kpERWPCS,HLSH18}). Since the proofs of a few of the technical results that we need are quite similar to those in the existing literature, we will state  these technical results here and give their proofs in Appendix \ref{app:afc}. 

\medskip

\noindent(i) \textbf{Process level tightness of extrema.}
We will denote the running minimum and running maximum of the random walk $X$ by 
\[
 I_n^X = \min_{k\leq n} X_k \quad\text{and}\quad S_n^X = \max_{k\leq n} X_k. 
\]
It follows easily from the generalized Ray--Knight theorems of T\'oth
in \cite{tGRK} that $S_n^X/\sqrt{n}$ and $I_n^X/\sqrt{n}$ converge in
distribution.  However, for our proof of Theorem \ref{thmAFC} we will
need the following slightly stronger statement - process level
tightness of the running minimum and maximum.
\begin{proposition}\label{minmaxtight}
 Let $\mathcal{S}^X_n(t) = \frac{S^X_{\fl{nt}}}{\sqrt{n}}$ and $\mathcal{I}^X_n(t) = \frac{I^X_{\fl{nt}}}{\sqrt{n}}$. Both $\{\mathcal{S}^X_n\}_{n\geq 0}$ and $\{\mathcal{I}^X_n\}_{n\geq 0}$ are tight in $D([0,\infty))$.
\end{proposition}
The proof of Proposition \ref{minmaxtight} is similar to the one for
excited random walks in \cite[Corollary 3.3]{kpERWPCS} (albeit
somewhat simpler) and is given in Appendix \ref{app:afc}.

\medskip

\noindent(ii) \textbf{Control of the number of rarely visited sites.}
Since the random walk $X$ is recurrent, and
since we expect a limiting distribution with diffusive scaling, it is
natural that most sites in the range of the
walk up to time $n$ should have been visited on the order of
$\sqrt{n}$ times. For our proof of Theorem \ref{thmAFC} we will need
to show that there are not too many sites in the range which have been
visited much  very few times. In
particular, we will need the following lemma.
\begin{lemma}\label{lem:smalllt}
Let $w(\cdot)$ be as in \eqref{w} with $\alpha=0$,
and let $\gamma_+ = \gamma \vee 0$ where $\gamma$ is defined in \eqref{v-u}.
Then for any $M>0$ and any $b > \frac{\gamma_+}{2}$ we have 
\[
\lim_{n\to\infty}  P\left( \sup_{k\leq nt} \sum_{x \in [I_{k-1}^X, S_{k-1}^X]} \ind{\mathcal{L}(x,k-1) \leq M} \geq 4 n^{b} \right) = 0, \qquad \forall M>0.  
\]
\end{lemma}
The proof of Lemma \ref{lem:smalllt} is an adaptation of proofs in
\cite{kmLLCRW,kpERWPCS,kmpCRW} and is given in Appendix \ref{app:afc}.

\medskip

A key point in our approximation of the
accumulated drift term by a linear combination of $S^X_n$ and $I^X_n$
(Lemma \ref{driftpart} below) is that for each site $x\in \Z$, there
is a random but finite accumulated drift from all visits to $x$:
\begin{equation}\label{dxdef}
 \delta_x 
 = \sum_{i=0}^\infty E\left[ X_{i+1}-X_i \mid \mathcal{F}^X_i \right] \ind{X_i=x}, \qquad x \in \Z.  
\end{equation}
It is not a priori obvious that the sum in \eqref{dxdef}
converges. However, the following lemmas show that
under the proper assumptions on the weight function $w$ the
series not only converges absolutely,
a.s., but also has finite moments of
all orders. Moreover, $E[\delta_x]$ can be explicitly
calculated.
\begin{lemma}\label{lem:bdxfin}
 Let the weight function $w$ be as in \eqref{w} with $\alpha=0$, $p \in (1/2,1]$, and $\kappa>0$. 
If 
\[
 \bar\delta_x 
 = \sum_{i=0}^\infty \left| E\left[ X_{i+1}-X_i \mid \mathcal{F}^X_i \right] \right| \ind{X_i=x},
\]
then $E[(\bar\delta_x)^M] < \infty$ for all $M>0$ and $x\in \Z$. 
In particular, this implies that the sum in \eqref{dxdef} converges, $P$-a.s., for all $x\in \Z$, and $E[|\delta_x|^M] < \infty$ for all $x\in\Z$ and $M>0$.
\end{lemma}

\begin{lemma}\label{lem:Edx}
 Let the weight function $w$ be as in \eqref{w} with $\alpha=0$, $p \in (1/2,1]$, and $\kappa>0$.
Then 
\[
 E[\delta_x] = \sgn(x)\,\gamma,
\]
where the constant $\gamma$ is as defined in \eqref{v-u} and
$\sgn\,(0):=0$.
\end{lemma}

The proofs of Lemmas \ref{lem:bdxfin} and \ref{lem:Edx} will be given
in Section \ref{sec:drift} where we analyze the sequence of left/right
steps from a fixed site using a generalized P\'olya urn.

\begin{remark}\label{rem:dxind}
  A crucial observation that will be used in the proof of Theorem
  \ref{thmAFC} below, is that the sequences of random variables
  $(\bar\delta_x)_{x\in \Z}$ and $(\delta_x)_{x\in\Z}$ are both
  sequences of independent random variables. This is due to the fact
  that any step of the random walk depends only on the behavior of the
  walk at previous visits to the current location so that the sequence
  of left/right steps from each site can be generated by
  independent realizations of generalized P\'olya urns (see Section
  \ref{sec:Polya}).  Moreover, since for any site $x<0$ (or for any
  $x>0$ respectively) the process of generating the sequence of
  left/right steps from $x$ is the same, it follows that the
  sequences $(\delta_x)_{x\geq 1}$ and $(\delta_x)_{x\leq -1}$ are
  both respectively i.i.d.
\end{remark}

\begin{remark}\label{rem:sym}
  Another fact that we will use is that there is a
  natural symmetry in the behavior of the walk to the right of the
  origin and to the left of the origin. In particular, we will repeatedly
  use that $S_n^X \overset{\text{Law}}{=} - I_n^X$ for all $n\geq 0$,
  $\delta_x \overset{\text{Law}}{=} - \delta_{-x}$  and
  $\bar\delta_x \overset{\text{Law}}{=} \bar\delta_1$ for all
  $x \neq 0$.
\end{remark}

Having collected the main tools we are now ready to move on to the proof of Theorem \ref{thmAFC}.

\subsection{Proof of Theorem~\ref{thmAFC}}
Recall that $M_n=X_n-\Gamma_n$, $n\in\N_0$, is a martingale. The first step in the proof of Theorem \ref{thmAFC} is to show that this martingale converges under diffusive scaling to a standard Brownian motion. 

\begin{lemma}\label{martpart}
  The process $\left(\frac{M_{\fl{nt}}}{\sqrt{n}}\right)_{t\geq 0}$
  converges in distribution to a standard Brownian motion.
\end{lemma}
\begin{proof}
 Since $M_n$ is a martingale with bounded increments, it is enough (see \cite[Theorem 18.2]{bCOPM}) to show that 
$\lim_{n\to\infty} \frac{1}{n} \sum_{i=0}^{n-1} E\left[(M_{i+1} - M_i)^2 \mid \mathcal{F}^X_i \right] = 1$, in probability. 
We have 
\begin{align*}
 E\left[(M_{i+1} - M_i)^2 \mid \mathcal{F}^X_i \right] &= E[(X_{i+1}-X_i - E[X_{i+1}-X_i \mid \mathcal{F}^X_i ] )^2 \mid \mathcal{F}^X_i ] \\
&= 1 - E[X_{i+1}-X_i \mid \mathcal{F}^X_i ]^2. 
\end{align*}
Therefore, it is enough to prove that 
$\lim_{n\to\infty} \frac{1}{n} \sum_{i=0}^{n-1} E[X_{i+1}-X_i \mid \mathcal{F}^X_i ]^2 = 0$, 
in probability. 
The estimate 
\[
 \sum_{i=0}^{n-1} E[X_{i+1}-X_i \mid \mathcal{F}^X_i ]^2
\leq \sum_{x \in [I_n^X, S_n^X]} \sum_{i=0}^\infty \left| E[X_{i+1}-X_i \mid \mathcal{F}^X_i ] \right| \ind{X_i = x}
= \sum_{x \in [I_n^X, S_n^X]} \bar\delta_{x} 
\]
implies that 
\begin{align*}
& P\left( \frac{1}{n} \sum_{i=0}^{n-1} E[X_{i+1}-X_i \mid \mathcal{F}^X_i ]^2 \geq \epsilon \right) \\
&\qquad \leq P( I_n^X \leq - n^{3/4} ) + P( S_n^X \geq n^{3/4} ) + P\left( \sum_{|x| \leq n^{3/4}} \bar\delta_{x} \geq \epsilon n \right) \\
&\qquad \leq 2P( S_n^X \geq n^{3/4} ) + \frac{E[\bar\delta_{0}] + 2n^{3/4} E[\bar\delta_{1}] }{\epsilon n}, 
\end{align*}
where in the last inequality we used symmetry considerations noted in
Remark \ref{rem:sym}.  By Proposition \ref{minmaxtight} and Lemma
\ref{lem:bdxfin}, the right side vanishes as $n\to\infty$. This
completes the proof.
\end{proof}

\begin{lemma}\label{driftpart}
Let $w$ be as in \eqref{w} with $\alpha = 0$ and $p \in (1/2,1]$. Then, 
\[
\lim_{n\to\infty} P\left( \sup_{k\leq nt} \left| \Gamma_{k} - \gamma( S^X_{k} + I^X_{k}) \right| > \epsilon \sqrt{n} \right) = 0, \quad \forall \epsilon, t>0. 
\]
\end{lemma}

\begin{proof}
For any $x\in \Z$ and $m\geq 1$ let $\delta_{x,m}$ be total drift accumulated in the first $m$ visits to $x$:
\begin{equation}\label{delxm}
 \delta_{x,m} 
= \sum_{i=0}^\infty E\left[ X_{i+1}-X_i \mid \mathcal{F}^X_i \right] \ind{X_i=x,\, \mathcal{L}(x,i)\leq m}, \qquad \forall x\in \Z, \, m\geq 1. 
\end{equation}
With this notation we have that
$\Gamma_k =  \sum_{x \in [I_{k-1}^X, S_{k-1}^X]} \delta_{x,\mathcal{L}(x,k-1)}$. 
Also, it follows from Lemma \ref{lem:Edx} that   
$\gamma(S_{k-1}^X + I_{k-1}^X) = \sum_{x \in [I_{k-1}^X, S_{k-1}^X]} E[\delta_x]$. 
Combining these two facts we see that
\begin{align}
&P\left( \sup_{k\leq nt} \left| \Gamma_{k} - \gamma( S^X_{k-1} + I^X_{k-1}) \right| > \epsilon \sqrt{n} \right) \nonumber \\
&\quad \leq P\Bigg( \sup_{k\leq nt} \Big| \sum_{x \in [I_{k-1}^X, S_{k-1}^X]} \left( \delta_{x,\mathcal{L}(x,k-1)} - \delta_x \right) \Big| \geq \frac{\epsilon\sqrt{n}}{2} \Bigg) \label{dxsums1} \\
&\quad\qquad  + P\Bigg( \sup_{k\leq nt} \Big| \sum_{x \in [I_{k-1}^X, S_{k-1}^X]} \left( \delta_x - E[\delta_x] \right)  \Big| \geq \frac{\epsilon\sqrt{n}}{2} \Bigg). \label{dxsums2}
\end{align}
For the probability in \eqref{dxsums2},  
it follows from the symmetry considerations in Remark \ref{rem:sym} that
for any fixed $K>0$,
\begin{align*}
 &P\Bigg( \sup_{k\leq nt} \Big| \sum_{x \in [I_{k-1}^X, S_{k-1}^X]} \left( \delta_x - E[\delta_x] \right)  \Big| \geq \frac{\epsilon\sqrt{n}}{2} \Bigg)\\
 &\leq 2 P\left(S_{\fl{nt}}^X \geq K\sqrt{n} \right) +  P\Bigg( |\delta_0| \geq \frac{\epsilon \sqrt{n}}{6} \Bigg) 
+ 2 P\Bigg( \max_{k\leq K\sqrt{n}} \Big| \sum_{x=1}^k \left( \delta_x - E[\delta_x] \right) \Big| \geq \frac{\epsilon\sqrt{n}}{6} \Bigg).
\end{align*}
For any fixed $\epsilon, K>0$, the last
two probabilities vanish as $n\to\infty$ by Lemma \ref{lem:bdxfin},
Remark \ref{rem:dxind}, and the strong law of large numbers.  Since
Proposition \ref{minmaxtight} implies that the first probability can
be made arbitrarily small (uniformly in $n$) by taking $K$
sufficiently large, we conclude that the
probability in \eqref{dxsums2} goes to $0$ as
$n\to\infty$ for any $\epsilon>0$.

It remains to estimate the probability in \eqref{dxsums1}.  To this
end, we will fix a parameter
$b \in (\frac{\gamma_+}{2}, \frac{1}{2})$. 
Then, since $|\delta_{x,m} - \delta_x| \leq \bar{\delta}_x$, we get that for any $M>0$
\begin{align}
&P\Bigg( \sup_{k\leq nt} \Big| \sum_{x \in [I_{k-1}^X, S_{k-1}^X]} \left( \delta_{x,\mathcal{L}(x,k-1)} - \delta_x \right) \Big| \geq \frac{\epsilon\sqrt{n}}{2} \Bigg) \nonumber \\
&\leq P\Bigg( \sup_{k\leq nt}  \sum_{x \in [I_{k-1}^X, S_{k-1}^X]} \bar\delta_x \ind{ \mathcal{L}(x,k-1) \leq M}  \geq \frac{\epsilon\sqrt{n}}{4} \Bigg) \nonumber \\
&\quad + P\Bigg( \sup_{k\leq nt} \sum_{x \in [I_{k-1}^X, S_{k-1}^X]} \left| \delta_{x,\mathcal{L}(x,k-1)} - \delta_x \right| \ind{ \mathcal{L}(x,k-1) > M} \geq \frac{\epsilon\sqrt{n}}{4} \Bigg) \nonumber \\
&\leq P\left( \max_{|x|\leq nt } \bar\delta_x \geq \frac{\epsilon}{16}n^{\frac{1}{2}-b}  \right) + P\Bigg( \sup_{k\leq nt} \sum_{x \in [I_{k-1}^X, S_{k-1}^X]} \ind{ \mathcal{L}(x,k-1) \leq M}  \geq 4 n^{b} \Bigg) \label{fv} \\
&\quad + P\Bigg( \sup_{k\leq nt} \sum_{x \in [I_{k-1}^X, S_{k-1}^X]} \left| \delta_{x,\mathcal{L}(x,k-1)} - \delta_x \right| \ind{ \mathcal{L}(x,k-1) > M} \geq \frac{\epsilon\sqrt{n}}{4} \Bigg). \label{mv}
\end{align}
Since $b<\frac{1}{2}$, it follows from Lemma \ref{lem:bdxfin} and Remark \ref{rem:dxind} that the first probability in \eqref{fv} vanishes as $n\to\infty$, while, since $b > \frac{\gamma_+}{2}$, the second probability in \eqref{fv} vanishes by Lemma \ref{lem:smalllt}.
For the probability in \eqref{mv}, we have for any $K>0$ that 
\begin{align}
&  P\Bigg( \sup_{k\leq nt} \sum_{x \in [I_{k-1}^X, S_{k-1}^X]} \left| \delta_{x,\mathcal{L}(x,k-1)} - \delta_x \right| \ind{ \mathcal{L}(x,k-1) > M} \geq \frac{\epsilon\sqrt{n}}{4} \Bigg) \nonumber \\
&\leq 2 P(S_{\fl{nt}-1}^X \geq K\sqrt{n} ) + P\Bigg( \sum_{|x|\leq K\sqrt{n}}  \left( \sup_{m>M} |\delta_{x,m} - \delta_x| \right) \geq \frac{\epsilon\sqrt{n}}{4} \Bigg) \nonumber \\
&\leq 2 P(S_{\fl{nt}-1}^X \geq K\sqrt{n} ) + \frac{4}{\epsilon\sqrt{n}} \sum_{|x|\leq K\sqrt{n}} E\left[ \sup_{m> M} |\delta_{x,m} - \delta_x| \right] \nonumber \\
&\leq 2 P(S_{\fl{nt}-1}^X \geq K\sqrt{n} ) + \frac{4}{\epsilon\sqrt{n}}  E[\bar\delta_0] + \frac{8K}{\epsilon} E\left[ \sup_{m> M} |\delta_{1,m} - \delta_1| \right] \label{llt}
\end{align}
where we have used the symmetry considerations from Remark
\ref{rem:sym} in the first and last inequalities.  Since
$|\delta_{1,m}-\delta_1|\leq \bar\delta_1$, it follows from Lemma
\ref{lem:bdxfin} and the dominated convergence theorem that
$E\left[ \sup_{m>M} |\delta_{1,m} - \delta_1| \right]\to 0$ as
$M\to\infty$.  Using this together with Proposition \ref{minmaxtight}
we can choose $K,M>0$ so that the first and third terms in \eqref{llt}
are arbitrarily small (uniformly in $n$). Thus, we can conclude that
the probability in \eqref{mv} vanishes as $n\to\infty$.
\end{proof}

\begin{proof}[Proof of Theorem \ref{thmAFC}]
First, we claim that the sequence of processes $\left(\frac{X_{\fl{nt}}}{\sqrt{n}}\right)_{t\geq 0}$ is tight in $D([0,\infty))$ and that any subsequential limit is concentrated on continuous paths.
This follows from the decomposition $X_n = M_n + \Gamma_n$, the tightness of the martingale term $M_n$ in Lemma \ref{martpart}, the approximation of the accumulated drift term $\Gamma_n$ by $\gamma(I_n^X + S_n^X)$ from Lemma \ref{driftpart}, and the tightness of the running extrema from Lemma \ref{minmaxtight}. 
The details of this argument are almost identical to the proof of \cite[Lemma 4.4]{kpERWPCS} and are therefore omitted.

By Proposition \ref{minmaxtight} and Lemmas \ref{martpart} and \ref{driftpart} we can then conclude that the process triple $\frac{1}{\sqrt{n}}\left( X_{\fl{nt}}, M_{\fl{nt}}, \Gamma_{\fl{nt}} \right)_{t\geq 0}$ is a tight sequence in $D([0,\infty))^3$ and that any subsequential limit $(Y_1(t), Y_2(t), Y_3(t))_{t\ge 0}$ is a continuous process such that $Y_2$ is a standard Brownian motion, $Y_3(t) = \gamma( \sup_{s\leq t} Y_1(s) + \inf_{s\leq t} Y_1(s) )$ for all $t\ge 0$, $P$-a.s., and $Y_1(t) = Y_2(t) + Y_3(t)$. That is, $Y_1$ is a $(\gamma,\gamma)$-BMPE as claimed. 
\end{proof}

\section{Polynomially self-repelling case}\label{three}

In this section we prove Theorem~\ref{thmPSR} modulo a version of a
generalized Ray--Knight theorem, Proposition~\ref{lemC1}. This
proposition deals with increments of numbers of upcrossings and is a
key fact behind the argument below. Its proof is based, in turn, on a
detailed analysis of a generalized P\'olya urn model which we carry
out later, in Section ~\ref{sec:Polya}.

  Recall the definition of the process $W_\alpha$ in \eqref{wlim} and
  let $I_\alpha(t):=\inf_{0\le s\le t}W_\alpha(s)$ and
  $S_\alpha(t):=\sup_{0\le s\le t}W_\alpha(s)$. Then the triple
  $(I_\alpha(t),W_\alpha(t),S_\alpha(t))$ is a strong Markov
  process\footnote{relative to any filtration with respect to which
    $(B(t))_{t\ge 0}$, that appears in the definition $W_\alpha$, is a
    Brownian motion, see \cite[Section 2.3]{cpyBetaPBM}.} which
  satisfies the equation
\[W_\alpha(t)=B_\alpha(t)+\frac12 S_\alpha(t)+\frac12
  I_\alpha(t),\quad W_\alpha(0)=0,\ t\ge 0,\] where
$B_\alpha(t)=\sqrt{2\alpha+1}\,B(t)$, $t\ge 0$, and $(B(t))_{t\ge 0}$ is a
standard Brownian motion. Recall also that $Z^{(\alpha,\delta)}$ is
$(2(2\alpha+1))^{-1}$ times a BESQ$^\delta$. The process $Z^{(\alpha,\delta)}$ has a non-random initial point and solves
\begin{equation}
  \label{Za0}
  dZ^{(\alpha,\delta)}(x)=\frac{\delta}{2(2\alpha+1)}\,dx+\frac{1}{\sqrt{2\alpha+1}}\,\sqrt{2Z^{(\alpha,\delta)}(x)}\,dB(x),
\end{equation}
$0\le x\le \inf\{y>0:\ Z^{(\alpha,\delta)}(y)=0\}$. Throughout this
section we also assume that $Z^{(\alpha,\delta)}$ is always absorbed
upon hitting $0$. Let
\begin{equation}\label{half}
  L^{W_\alpha}_t(x):=\lim_{\epsilon\downarrow 0}\frac{1}{2\epsilon}\int_0^t\indf{[x,x+\epsilon]}(W_\alpha(s))\,ds,\quad x\in\R,\ t\ge 0,
\end{equation}
be a half of the local time\footnote{The choice to work with a half of
  the local time corresponds to the fact that for the walk we consider
  processes of edge local times.} of $W_\alpha$ at $x$ by time $t$ and set
${\mathcal T}_\ell^{L^{W_\alpha}}:=\inf\{t\ge 0:
L^{W_\alpha}_t(0)>\ell\}$, $\ell\in[0,\infty)$. We shall write
$L^{W_\alpha}_{{\mathcal T}_\ell}$ instead of more cumbersome
$L^{W_\alpha}_{{\mathcal T}_\ell^{L^{W_\alpha}}}$.
\begin{proposition}\label{settingthingsup}
  For every $M>0$,
  $\left(L^{W_\alpha}_{{\mathcal T}_M}(x)\right)_{x\ge
    0}\overset{\text{Law}}{=}\left(Z^{\alpha,1}(x)\right)_{x\ge 0}$ with
  $Z^{(\alpha,1)}(0)=M$.  As $M\to\infty$,
\begin{equation}
  \label{LM}
  \left({\mathcal Z}^{W_\alpha}_M(x)\right)_{x\ge 0}:=\left(L^{W_\alpha}_{{\mathcal T}_{M+1}}(x)-L^{W_\alpha}_{{\mathcal T}_M}(x)\right)_{x\ge 0}\Longrightarrow \left(Z^{(\alpha,0)}(x)\right)_{x\ge 0},
  \ \ Z^{(\alpha,0)}(0)=1.
\end{equation}
\end{proposition}

\begin{proof}
  The first statement follows from \cite[Theorem 3.1]{cpyBetaPBM}. To
  show \eqref{LM}, we note that $L^{W_\alpha}_{{\mathcal T}_M}(x_0)>0$ for
  some $x_0>0$ implies that $L^{W_\alpha}_{{\mathcal T}_M}(x)>0$ on
  $[0,x_0]$ and also that $S_\alpha({\mathcal T}_M^{W_\alpha})>x_0$. Thus,
  on the event $\{L^{W_\alpha}_{{\mathcal T}_M}(x_0)>0\}$ the
  process
  $(L^{W_\alpha}_{{\mathcal T}_{M+1}}(x)-L^{W_\alpha}_{{\mathcal
        T}_M}(x))_{0\le x\le x_0}$ has the same distribution
  as
  $(L^{B_\alpha}_{{\mathcal T}_1}(x))_{0\le x\le
    x_0}$.\footnote{$L^{B_\alpha}_{{\mathcal T}_1}$ is defined analogously
    to $L^{W_\alpha}_{{\mathcal T}_1}$.} By scaling properties of
  Brownian motion and the second Ray--Knight theorem for the standard
  Brownian motion, the process
  $L^{B_\alpha}_{{\mathcal T}_1}(x),\ x\in[0,x_0]$, has the same
  distribution as $Z^{(\alpha,0)}(x),\ x\in[0,x_0]$. Since for every
  $x_0>0$ the probability of 
  $\{L^{W_\alpha}_{{\mathcal T}_M}(x_0)>0\}$ goes to 1 as
  $M\to\infty$, \eqref{LM} follows.
\end{proof}

The key observation we use to prove  Theorem \ref{thmPSR} is that the discrete analog of the left hand side
of \eqref{LM} constructed with $X^{(n)}$ in place of $W_\alpha$
is close in distribution (Proposition~\ref{lemC1}
below) to $Z^{(0,0)}$, which is a half of BESQ$^0$ and, therefore,
is different from $Z^{(\alpha,0)}$. Before continuing our discussion
let us state this result rigorously.

Recall \eqref{upcr} and for $\ell\ge 0$ define
${\mathcal T}^{\mathcal E}_{\ell}:=\inf\{k\ge 0:\,{\mathcal E}^k(0)>\ell \}$. Just
as before, we shall write ${\mathcal E}^{{\mathcal T}_{\ell }}(x)$ instead of
${\mathcal E}^{{\mathcal T}_{\ell }^{\mathcal E}}(x)$.

\begin{proposition}\label{lemC1}
For every fixed $M>0$, 
as $N\to\infty$,
  $\left(N^{-1}{\mathcal E}^{{\mathcal T}_{NM}}(\fl{Nx}) \right)_{ x\ge
    0}\Longrightarrow \left(Z^{(\alpha,1)} (x)\right)_{s\ge 0}$
  with $Z^{(\alpha,1)}(0)=M$. 
  Moreover, for every $\eta>0$ and $c\in[0,1/2)$ there exists $M_0>0 $ such that
  for every $M \geq M_0 $  there  is an $N(M)$ so that for all $N \ge N(M)$,
  \begin{equation}
    \label{EM}
    \text{dist}\,(P^{Z^{M,c}_N},P_{1-2c}^{Z^{(0,0)}})<\eta,
  \end{equation}
  where $P^{Z^{M,c}_N}$ denotes the law of the process
\[Z^{M,c}_N(x):=\frac{1}{N}\left({\mathcal E}^{{\mathcal T}_{N(M+1-c)}}(\fl{Nx})- {\mathcal E}^{{\mathcal T}_{N(M+c)}}(\fl{Nx})\right),\ \ x\in[0,1],\] and $P_{1-2c}^{Z^{(0,0)}}$ denotes the law of $(Z^{(0,0)}(x))_{x\in[0,1]}$ with $Z^{(0,0)}(0)=1-2c$.
\end{proposition}

\begin{remark}
  We note that the first claim immediately follows from \cite[Theorem
  1B]{tGRK}. Here we offer an informal discussion as to why the
  ``limiting'' process $Z^{(0,0)}$ with $c=0$ in \eqref{EM} is
  different from the limiting process $Z^{(\alpha,0)}$ in
  \eqref{LM}. We have already mentioned that this fact is the key to
  the proof of Theorem~\ref{thmPSR}. The appearance of $Z^{(0,0)}$ in
  \eqref{EM} is not intuitive and is based on a careful analysis of
  processes of left and right jumps of $X$ from a single site.  More
  precisely, if we look at the rescaled difference between the number
  of jumps of $X$ to the right between the $KR$-th and $(K+1)R$-th
  jumps to the left from a single site\footnote{i.e., the process
    $(\D_{\tau^\B_m}-\D_{\tau^\B_n})/\sqrt{2(m-n)}$ with $n=KR$ and
    $m=(K+1)R$, see Proposition~\ref{VarDtmDtn}.} then, as $R$ goes to
  infinity, this rescaled difference has approximately zero mean and
  variance $v(\alpha,K)= 1+O(K^{-1})\to 1$ as $K$ grows large. On the
  other hand, for $K=0$ the variance converges to
  $v(\alpha,0)=(2\alpha+1)^{-1}$ (Proposition~\ref{VarDtmDtn}). This
  is exactly the factor which enters generalized Ray--Knight theorems
  making the limiting processes in the first claims of
  Propositions~\ref{settingthingsup} and \ref{lemC1} to be
  $Z^{(\alpha,1)}$ instead of $Z^{(0,1)}$. The dependence of
  $v(\alpha,K)$ on $K$ reflects the dependence of this rescaled
  difference on the history of the walk prior to the $KR$-th jump to
  the left. As a consequence of this dependence, the ``limiting''
  process in \eqref{EM} is $Z^{(0,0)}$ and not $Z^{(\alpha,0)}$ as in
  \eqref{LM}.  It is these findings that allow us to rule out BMPE as
  a possible weak limit.  The mentioned above
  Proposition~\ref{VarDtmDtn} and the proof of Proposition~\ref{lemC1}
  are given at the end of Subsection~\ref{psrurn}.
\end{remark}

The fact that multiples of the limiting BESQ processes in \eqref{LM}
and \eqref{EM} are different does not immediately imply the statement
of Theorem 1.2, since the difference is expressed in terms of the
local time processes, which are not a.s.\ continuous functionals on
the path space. Nevertheless, replacing local times with averages of
occupation times of small intervals we shall be able to prove
Theorem~\ref{thmPSR}.
  \begin{proof}[Proof of Theorem~\ref{thmPSR}] We work on
    $D([0,\infty))$ with the topology generated by one of the
    equivalent Skorokhod metrics (see, for example,
    \cite[(16.4), (12.16)]{bCOPM}). For $\delta>0$ and
    $\ell\in[0,\infty)$ define
  \[{\mathcal T}_{\delta,\ell}(\omega)=\inf\left\{t\ge 0:\,\frac{1}{2
        \delta}\int_0^t\indf{[0,\delta]}(\omega(s))\,ds>\ell\right\};\
    \ G_{\delta,\ell}(\omega)=\frac12\int_{{\mathcal T}_{\delta,\ell}}^{{\mathcal
        T}_{\delta,\ell+1}}\indf{[0,1]}(\omega(s))\,ds.\] We shall show that, on the one hand,
  \begin{itemize}
  \item [(UB)] for each $\epsilon>0$ there are
    $M_0=M_0(\epsilon),\delta_0=\delta_0(\epsilon)>0$ such that for
    all $M\ge M_0$, $\delta\in(0,\delta_0)$, 
    \begin{equation} \label{ubound}
      E^{W_\alpha}\left[G_{\delta,M}^2 \right]-E\left[\left(\int_0^1 Z^{(\alpha,0)}(x)\,dx\right)^2\right]\overset{Lem.\,\ref{calc}}{=} E^{W_\alpha}\left[G_{\delta,M}^2 \right]-\left(1+\frac{2}{3(2\alpha+1)}\right)\le \epsilon,
    \end{equation}
  \end{itemize}
  while, on the other hand,
  \begin{itemize}
  \item [(LB)] for each $\epsilon>0$ there are $M\ge M_0(\epsilon)$, $0<\delta\le \delta_0(\epsilon)$ and $K_0=K_0(\epsilon)$ such that for all $K\ge K_0$ and all sufficiently large $n$ of the form $n=N^2$
    \begin{equation}
      \label{lbound}
      E^{X^{(n)}}\left[G_{\delta,M}^2\wedge K\right]-E\left[\left(\int_0^1 Z^{(0,0)}(x)\,dx\right)^2\right]\overset{Lem.\,\ref{calc}}{=} E^{X^{(n)}}\left[G_{\delta,M}^2\wedge K\right]-\frac{5}{3}\ge -2\epsilon.
    \end{equation}
\end{itemize}
Once this is done, we take an
$\epsilon\in\left(0,\frac{\alpha}{3(2\alpha+1)}\right)$, choose
$M_0$, $\delta_0$ as in (UB) and then $\delta$, $M$, and $K$ as in
(LB), and get that for all large $n$ of the form $n=N^2$
\[ E^{X^{(n)}}\left[G_{\delta,M}^2\wedge K\right]
\ge E^{W_\alpha}\left[G_{\delta,M}^2 \right]+\epsilon
\ge E^{W_\alpha}\left[G_{\delta,M}^2\wedge K\right]+\epsilon.
\] Since $G_{\delta,M}^2\wedge K$ is a
continuous bounded functional on $D([0,\infty))$ (see Lemma \ref{cont} below), the conclusion of
Theorem~\ref{thmPSR} follows.
\end{proof}

It is left to prove bounds (UB) and (LB). We shall need the following technical lemma. Its proof is given in Appendix~\ref{app:psr}.
    \begin{lemma}
      \label{cont} For all $\delta>0, M\ge 0$, $G_{\delta,M}:\Omega\to \R$
      is a $P^{W_\alpha}$-a.s.\ continuous functional. Moreover, there is $\delta_0\in(0,1/2]$ such that for
      every $p\ge 1$,
      $\sup\limits_{M\ge 0,\,
        \delta\in(0,\delta_0]}E^{W_\alpha}\left[(G_{\delta,M})^p\right]<\infty$.
    \end{lemma}

\begin{proof}[Proof of (UB)]
  Let
  ${\mathcal{T}}'_{\delta,M} = \inf \{ s >{\mathcal{T}}_{\delta,M}: W_\alpha
  (s) =0 \}$ and
  ${\mathcal{T}}'_{\delta,M,s} = \inf \{t >{\mathcal{T}}' _{\delta,M}:
  L^{W_\alpha }_t(0) - L^{W_\alpha }_{{\mathcal{T}}' _{\delta,M}}(0) =
  s\}$.\footnote{Recall \eqref{half}: $L^{W_\alpha }_t(0)$ denotes the
    half of the local time at $0$ by time $t$.}  We fix $\lambda > 0 $
  (eventually we will fix $\lambda\in(0,1) $ small) and introduce two
  ``bad'' events 
\begin{align}
 A_1 &= \{ I_\alpha ({\mathcal{T}}_{\delta,M}) < W_\alpha (s)  < S_\alpha ({ \mathcal{T}}_{\delta,M})\ \forall s\in
[{\mathcal{T}}_{\delta,M},{ \mathcal{T}}'_{\delta,M,2}  \vee {\mathcal{T}}_{\delta,M+1}] \}^c
\label{c1}\\ \shortintertext{and}
A_2 &= \{{\mathcal{T}}'_{\delta,M}- {\mathcal{T}}_{\delta,M}  \geq \lambda \} \cup \{
\vert {\mathcal{T}}_{\delta,M+1} - {\mathcal{T}}'_{\delta,M,1} \vert  \geq \lambda  \}. \label{c2}
\end{align}
We shall show below in Lemma~\ref{c12} that for each $\eta > 0 $ there
exist $M_0, \delta_0>0$ such that $P(A_1 \cup A_2) < \eta$ for all
$M \geq M_0, \delta \in(0, \delta_0]$.  
On $(A_1 \cup A_2)^c$ we introduce the process
$ \tilde{B}_\alpha(s):= B_\alpha ({\mathcal{T}}'_{\delta,M}+s) - B_\alpha
({\mathcal{T}}'_{\delta,M}) $ for $s \in[0,\nu - {\mathcal{T}}'_{\delta,M}]$
where
$\nu = \inf \{ s>{\mathcal{T}}_{\delta,M}: W_\alpha (s)
=I_\alpha({{\mathcal{T}}_{\delta,M}} ) \mbox{ or }
S_\alpha({\mathcal{T}}_{\delta,M})\}$.
Then on $(A_1 \cup A_2)^c$ 
\begin{equation*}
\int_{{\mathcal{T}}_{\delta,M}}^{{\mathcal{T}}_{\delta,M+1}}\indf{[0,1]}(W_{\alpha} (s))\,ds - 2 \lambda\le  \int_{{\mathcal{T}}'_{\delta,M}}^{{\mathcal{T}}'_{\delta,M,1}}\indf{[0,1]}(W_{\alpha} (s)) ds = \int_0^{{\mathcal T}^{L^{\tilde{B}_\alpha}}_1} \indf{[0,1]}(\tilde{B}_\alpha(s)) ds
 .
\end{equation*}
We note that on $(A_1\cup A_2)^c$ the process $(\tilde{B}_\alpha(s))_{s\ge 0}$ has the same distribution as $(B_\alpha(s))_{s\ge 0}$ and that, by the second Ray--Knight theorem for the standard Brownian motion and scaling,\[\frac12\int_0^{{\mathcal T}^{L^{B_\alpha}}_1} \indf{[0,1]}(B_\alpha(s)) ds\overset{\text{Law}}{=}\int_0^1Z^{(\alpha,0)}(x)\,dx,\ \text{where }Z^{(\alpha,0)}(0)=1.
\] Hence, 
\begin{align*}
E^{W_\alpha}[{G}_{\delta,M}^2] &=
 E^{W_\alpha}\left[{G}_{\delta,M}^2\indf{(A_1\cup A_2)^c}\right]+ E^{W_\alpha}\left[{G}_{\delta,M}^2\indf{A_1\cup A_2}\right] \\
& \leq
E\left[ \bigg( \frac12\int_0^{{\mathcal T}^{L^{B_\alpha}}_1} \indf{[0,1]}(B_\alpha (s)) ds +\lambda \bigg)^2 \right] + 
\left(E^{W_\alpha}\left[{G}_{\delta,M}^4\right]P(A_1\cup A_2)\right)^{1/2}  \\ &
\leq
E\left[ \left( \int_0^1Z^{(\alpha,0)}(x)\,dx\right)^2\right] + C\lambda +\left(E^{W_\alpha}\left[{G}_{\delta,M}^4\right]P(A_1\cup A_2)\right)^{1/2} 
\end{align*}
for a universal constant $C$.  Finally, we choose $\lambda $ to be
small enough so that $ C\lambda \le \epsilon/2$ and then use Lemmas~\ref{cont} and \ref{c12} to conclude that there exist
$M_0, \delta _0$ such that the last term in the above display
formula is less than $\epsilon/2$ whenever $M \geq M_0 $ and
$\delta \in(0,\delta _0]$. This completes the proof of (UB).
\end{proof}

\begin{lemma}\label{c12}
  Fix $\lambda>0$ and define $A_1$ and $A_2$ by \eqref{c1} and
  \eqref{c2}. For each $\eta > 0 $ there exist $M_0, \delta_0>0$
  depending on $\eta$ and $\lambda$ such that $P(A_1 \cup A_2) < \eta$
  for all $M \geq M_0, \delta \in(0, \delta_0]$.
\end{lemma}

\begin{proof}
We first give a bound for $P(A_1)$.  If
    $[-\sqrt{M},\sqrt{M}] \subset [I_\alpha(\mathcal{T}_{\delta,M}),
    S_\alpha(\mathcal{T}_{\delta,M})]$, then
    $W_\alpha( \mathcal{T}_{\delta,M} + \cdot)$ has the same law as the
    scaled Brownian motion $B_\alpha(\cdot)$ started at some point
    $z \in [0,\delta]$ until exiting $[-\sqrt{M},\sqrt{M}]$.  Also
    note that on the event $A_1$, at the time when the process
    $W_\alpha( \mathcal{T}_{\delta,M} + \cdot)$ exits
    $[I_\alpha(\mathcal{T}_{\delta,M}), S_\alpha(\mathcal{T}_{\delta,M}) ]$,
    there must be a point in $[0,\delta]$ such that a half of the  local time at this point does not exceed 2.  Therefore, 
  \begin{align*}
   P(A_1) &\leq P\left( \mathcal{T}_{\delta,M} < \tau^{W_\alpha}_{\sqrt{M}} \vee \sigma^{W_\alpha}_{-\sqrt{M}} \right)
 + \sup_{z \in [0,\delta]} P_z\left( \inf_{x \in [0,\delta]} L^{B_\alpha}_{\tau^{B_\alpha}_{\sqrt{M}} \wedge \sigma_{-\sqrt{M}}^{B_\alpha}} (x) \leq 2 \right) 
 \\
 &\leq 2 P\left( \frac{1}{\delta} \int_0^\delta L^{W_\alpha}_{\tau^{W_\alpha}_{\sqrt{M}}} (x) \, dx > M \right) + \sup_{z \in [-\delta,\delta]} 2  P_z\left( \inf_{x \in [-\delta,\delta]} L^{B_\alpha}_{\tau^{B_\alpha}_{\sqrt{M}}} (x) \leq 2 \right).
  \end{align*}
    Using the Ray--Knight theorems for
  BMPE and Brownian motion,
respectively, one can show that by choosing $\delta_0$ sufficiently
small and $M_0$ sufficiently large we can ensure that
$P(A_1)\le \eta/3$ for all $M\ge M_0$ and $\delta\in(0,\delta_0]$.

  Our remaining task is to
  bound $P(A_2\cap A_1^c)$. We shall start with the set
  $A_2':=\{{\mathcal T}_{\delta,M}'-{\mathcal T}_{\delta,M}\ge \lambda\}\cap
  A_1^c$. On $A_1^c$, we are simply estimating the probability that
  the process $B_\alpha$ which started at some point in $[0,\delta]$
  does not hit $0$ by time $\lambda$. Taking the worst case scenario
  $B_\alpha(0)=\delta$ and using the reflection principle for
  $B_\alpha$ we get that for all $\delta\in(0,\delta_0]$
  \[P(A_2')\le 1-2P(B_\alpha(\lambda)\ge \delta \mid B_\alpha(0)=0)\le
    \frac{2}{\sqrt{2\pi}}\int_0^{\frac{\delta}{\sqrt{\lambda (2\alpha+1)}}}e^{-x^2/2}dx\le
    \frac{\delta_0}{\sqrt{\lambda}}<\frac{\eta}{3},\] if we choose
  $\delta_0$ sufficiently small.

  We turn to the set $A_2'':=\{|{\mathcal T}_{\delta,M+1}-{\mathcal T}_{\delta,M,1}'|\ge \lambda\}\cap A_1^c$. 
Letting
  \begin{align*}
  & \{|{\mathcal T}_{\delta,M+1}-{\mathcal T}_{\delta,M,1}'|\ge \lambda\}
   = I\cup II\cup III := \\
   &\qquad 
    \{{\mathcal T}_{\delta,M+1}<{\mathcal T}_{\delta,M,1-\delta^{1/3}}'\}
   \cup\{{\mathcal T}_{\delta,M+1}>{\mathcal T}_{\delta,M,1+\delta^{1/3}}'\}
    \cup\{{\mathcal T}_{\delta,M,1+\delta^{1/3}}'-{\mathcal T}_{\delta,M,1-\delta^{1/3}}'\ge \lambda \},
  \end{align*}
we will bound $P(I\cap A_1^c)$,  $P(II\cap A_1^c)$, and  $P(III\cap A_1^c)$ separately.
  For $P(I \cap A_1^c)$,
taking into account that $\frac{1}{2\delta}\int_{{\mathcal T}_{\delta,M}}^{{\mathcal T}_{\delta,M+1}}\indf{[0,\delta]}(W_\alpha(s))\,ds=1$ we get that
\begin{align*}
  P\left(I\cap A_1^c\right)
  &\le P\left(I \cap A_1^c \cap \left\{\frac{1}{2\delta}\int_{{\mathcal T}_{\delta,M}}^{{\mathcal T}_{\delta,M}'}\indf{[0,\delta]}(W_\alpha(s))\,ds\le \frac{\delta^{1/3}}{2}\right\}\right) \\
  &\qquad + P\left( A_1^c \cap \left\{ \frac{1}{2\delta}\int_{{\mathcal T}_{\delta,M}}^{{\mathcal T}_{\delta,M}'}\indf{[0,\delta]}(W_\alpha(s))\,ds > \frac{\delta^{1/3}}{2} \right\} \right)\\
  &\le P\left( A_1^c \cap  \left\{ \frac{1}{2\delta}\int_{{\mathcal T}_{\delta,M}'}^{{\mathcal T}_{\delta,M,1-\delta^{1/3}}'}\indf{[0,\delta]}(W_\alpha(s))\,ds\ge 1-\frac{\delta^{1/3}}{2}\right\} \right) \\
  &\qquad + P\left( A_1^c \cap  \left\{ {\mathcal T}_{\delta,M}'-{\mathcal T}_{\delta,M}>\delta^{4/3}\right\} \right).
\end{align*}
The last probability is bounded by the probability that a Brownian motion $B_\alpha$ started at a point in $[0,\delta]$ doesn't hit the origin by time $\delta^{4/3}$, and thus repeating the argument giving the bound for $P(A_2')$ we get that the last probability is
less than $\delta_0^{1/3}$. 
For the next to last term, since the process $W_\alpha$ is a Brownian motion on the time interval in the integral, it follows from
the second Ray--Knight theorem for
$B_\alpha$
that this term is bounded above by
\begin{align}
&P_{1-\delta^{1/3}}\left(\frac{1}{\delta}\int_0^\delta
    Z^{(\alpha,0)}(x)\,dx>1-\frac{\delta^{1/3}}{2}\right) \label{RKbd} \\
 &\qquad \le
  P_{1-\delta^{1/3}}\left(\left|\frac{1}{\delta}\int_0^\delta
      Z^{(\alpha,0)}(x)\,dx-\left(1-\delta^{1/3}\right)\right|>\frac{\delta^{1/3}}{2}\right)\le 3 \delta^{1/3}, \nonumber
\end{align}
where in
 the last line we used Chebyshev's inequality and Lemma~\ref{calc}. Hence,
$P\left(I\cap A_1^c\right)\le 4 \delta_0^{1/3}$. The
estimate of $P\left(II\cap A_1^c\right)$ is very similar,
since
\begin{align*}
 P\left(II\cap A_1^c\right) 
& \leq P\left( A_1^c \cap \left\{ \frac{1}{2\delta}\int_{{\mathcal T}_{\delta,M}'}^{{\mathcal T}_{\delta,M,1+\delta^{1/3}}'}\indf{[0,\delta]}(W_\alpha(s))\,ds \leq  1 \right\} \right) \\
&\leq P_{1+\delta^{1/3}} \left( \frac{1}{\delta} \int_0^\delta Z^{(\alpha,0)}(x) \, dx \leq 1 \right),
\end{align*}
and then using an argument similar to the bound in \eqref{RKbd} we get
that this is at most $\frac{2}{3} \delta_0^{1/3}(1+\delta_0^{1/3})$.
Finally, note that
$ P\left(III\cap A_1^c\right) \le P\left( L^{B_\alpha}_\lambda(0) \leq
  2\delta^{1/3} \right) \to 0$ as $\delta\to 0$.
We conclude that choosing $\delta_0$ sufficiently small we can ensure that $P(A_2'')<\eta/3$.
\end{proof}

\begin{proof}[Proof of (LB)]
  Given $\epsilon>0$, we pick and fix $c>0$ so small and $K_0$
  so large that
\begin{equation}\label{z00}
E_{1-2c} \left[  \left(\int_0^1{ Z}^{(0,0)}(x)\,dx \right)^2 \wedge K_0 \right] \geq \frac53 - \frac{\epsilon}{2}.
\end{equation}
To justify this, we first take $c$ sufficiently small to ensure that
$E_{1-2c} \left[ \left( \int_0^1 Z^{(0,0)}(x)\,dx \right)^2 \right]
\geq \frac53 - \frac{\epsilon}{4}$, which can be done via scaling and
monotone convergence, and then find $K_0$ by applying the monotone
convergence theorem.

We then choose $\eta>0$ small so that for this $K_0$ the
implication \eqref{eta} holds for the functional
$F(\omega)=\left(\int_0^1\omega(s)\,ds\right)^2\wedge K_0$, $\omega\in D([0,1])$.
Next, for $c$ chosen as in \eqref{z00} we take $M\ge M_0$, where $M_0$ is as in (UB),
and such that (see Proposition \ref{lemC1})
$\text{dist}\, (P^{Z_N^{M,c}}, P_{1-2c}^{Z^{(0,0)}}) < \eta$ for all
$N\ge N(M)$.
It remains to pick $\delta\in(0,\delta_0)$, where $\delta_0$ is as in
(UB). Let
\[
 J_n=\left\{{ n\mathcal{T}}_{\delta,M}(X^{(n)}) <
{\mathcal{T}}^{\mathcal{E}}_{(M+c)\sqrt{n}} \right\}\cap \left\{ {\mathcal{T}}^{\mathcal{E}}_{(M+1-c)\sqrt{n}} < n{\mathcal{T}}_{\delta,M+1}(X^{(n)}) \right\}.
\]
By Lemma \ref{notyet}
below we can choose $\delta$ so small that for all $n=N^2$  large
$P^{X^{(n)}}(J_n)\ge 1-\epsilon/(2K_0)$.  
Finally, note that if $n=N^2$ then the definitions of the functional $G_{\delta,M}$ and the event $J_n$ imply that
\begin{equation}
  \label{GJ}
  G_{\delta,M}(X^{(n)}) \ge 
  \indf{J_{n}} \left( \frac{1}{2N^2} \sum_{k=\mathcal{T}^{\mathcal{E}}_{(M+c)N}}^{\mathcal{T}^{\mathcal{E}}_{(M+1-c)N}-1} \indf{[0,N]}(X_k) \right)
  \geq  
  \indf{J_{n}}\int_0^1 Z^{M,c}_N(x)\, dx, 
\end{equation}
Putting everything together we get that for our choice of $c,K_0,\eta,M,\delta$, all $K\ge K_0$ and all sufficiently large $n$ of the form $n=N^2$
\begin{align*}
  E^{X^{(n)}} \left[G_{\delta,M}^2\wedge K\right]-\frac{5}{3}\overset{\eqref{z00}}{\ge} &E^{X^{(n)}}\left[G_{\delta,M}^2\wedge K_0\right]-E_{1-2c} \left[  \left(\int_0^1Z^{(0,0)}(x)\,dx \right)^2 \wedge K_0 \right]-\frac{\epsilon}{2}\\ \overset{\eqref{eta}}{\ge} &E^{X^{(n)}} \left[G_{\delta,M}^2\wedge K_0\right]-E^{X^{(n)}} \left[  \left(\int_0^1Z_N^{M,c} (x)\,dx \right)^2 \wedge K_0 \right]-\frac{3\epsilon}{2}\\ \overset{\eqref{GJ}}{\ge} &-E^{X^{(n)}} \left[\indf{J_n^c}  \left(\left(\int_0^1Z_N^{M,c} (x)\,dx \right)^2 \wedge K_0 \right)\right]-\frac{3\epsilon}{2}\ge -2\epsilon.
\end{align*}
This completes the proof of (LB).
\end{proof}

\begin{lemma} \label{notyet} Given $\epsilon,c,M >0 $ there exists
  $\delta_0 >0$ such that for all $\delta\in(0,\delta_0)$ and all
  $n=N^2$ large, 
\[
P^{X^{(n)}}\left( {n\mathcal{T}}_{\delta,M}(X^{(n)}) <
{\mathcal{T}}^{\mathcal{E}}_{(M+c)\sqrt{n}}, \, {\mathcal{T}}^{\mathcal{E}}_{(M+1-c)\sqrt{n}} < n
{\mathcal{T}}_{\delta,M+1}(X^{(n)}) \right)\ge 1-\epsilon.
\]
\end{lemma}
\begin{proof}
First of all, note that for all $M',\delta>0$ and integers $N$  
\begin{align*}
 \sum_{i=0}^{\mathcal{T}^{\mathcal{E}}_{N M'}-1} \indf{[0,N\delta]}(X_i)
 &= \sum_{x=0}^{\fl{N\delta}} \left( \mathcal{D}^{\mathcal{T}^{\mathcal{E}}_{NM'}}(x) + \mathcal{E}^{\mathcal{T}^{\mathcal{E}}_{NM'}}(x) \right)  \\
 &= 2 \left( \sum_{x=0}^{\fl{N\delta}-1} \mathcal{E}^{\mathcal{T}^{\mathcal{E}}_{NM'}}(x) \right)  + \mathcal{E}^{\mathcal{T}^{\mathcal{E}}_{NM'}}(-1) + \mathcal{E}^{\mathcal{T}^{\mathcal{E}}_{NM'}}(\fl{N\delta}) - 1, 
\end{align*}
where the second equality is due to the fact that at time
  $\mathcal{T}^{\mathcal{E}}_{NM'}$ the walk has just completed a jump from 0
  to 1 and, thus, 
$\mathcal{D}^{\mathcal{T}^{\mathcal{E}}_{NM'}}(x) =
\mathcal{E}^{\mathcal{T}^{\mathcal{E}}_{NM'}}(x-1) - \indf{\{1\}}(x)$.  This
together with the bound
\[\bigg| \frac{1}{\delta N^2} \sum_{x=0}^{\fl{N\delta}-1} \left({\mathcal
  E}^{{\mathcal T}_{NM'}}(x) -NM'\right) \bigg| \leq \sup_{0\leq x < \fl{N\delta}}
\bigg|\frac{ {\mathcal E}^{{{\mathcal T}_{NM'} }}(k)}{N} - M^ \prime \bigg|\] gives
that for all large $N$
\begin{align*}
 \left\{ { n\mathcal{T}}_{\delta,M}(X^{(n)}) \geq
{\mathcal{T}}^{\mathcal{E}}_{(M+c)\sqrt{n}}  \right\} 
 &\subseteq
 \left\{ \frac{1}{2\delta N^2} \sum_{i=0}^{ \mathcal{T}^{\mathcal{E}}_{(M+c)N}-1} \indf{[0,\delta N]}(X_i) \leq M \right\} 
 \\  &\subseteq
 \left\{ \left| \frac{1}{\delta N^2} \sum_{x=0}^{\fl{N\delta}-1} \left( \mathcal{E}^{\mathcal{T}^{\mathcal{E}}_{(M+c)N}} (x) - (N(M+c) \right) \right| \geq \frac{c}{2} \right\} \\
 &\subseteq \left\{ \sup_{0\leq x <\fl{N\delta} } \left| \frac{\mathcal{E}^{\mathcal{T}^{\mathcal{E}}_{(M+c)N}} (x)}{N} - (M+c)  \right| \geq \frac{c}{2}  \right\}.
\end{align*}
It follows from Proposition \ref{lemC1} that for $\delta_0$ chosen sufficiently small the probability of the last event is less than $\epsilon /2 $ for all $N$ large.  
The event 
$\left\{ {\mathcal{T}}^{\mathcal{E}}_{(M+1-c)\sqrt{n}} \geq n {\mathcal{T}}_{\delta,M+1}(X^{(n)}) \right\}$ can be treated similarly, and we omit the details.
\end{proof}

\section{Results about generalized P\'olya urn model}\label{sec:Polya}

In this section we shall recall relevant facts about generalized
  P\'olya's urns and derive several results needed for
  proofs of Lemma~\ref{lem:bdxfin}, Lemma~\ref{lem:Edx}, and
  Proposition~\ref{lemC1} deferred from the previous two sections, thus, completing the proofs of our main results.

  The sequence of left/right steps from a fixed site can be thought of
  as being generated by a generalized P\'olya urn process.  The urn
  process is slightly different depending on whether the site is to
  the left of the origin, the right of the origin, or at the origin
  and thus, following \cite{tGRK}, we first describe the family of
  generalized P\'olya urn processes, and then later we will specialize
  to the urn processes which correspond to generating steps of the
  random walk at a fixed site.  Given sequences of positive numbers
  $\{b(i)\}_{i\geq 0}$ and $\{r(i)\}_{i\geq 0}$, the generalized
  P\'olya urn process $\{(\mathfrak{B}_n, \mathfrak{R}_n)\}_{n\geq 0}$
  is a Markov chain on $\N_0^2$ started at
  $(\mathfrak{B}_0, \mathfrak{R}_0) = (0,0)$ with transition
  probabilities given by
\begin{align*}
 P\left( (\mathfrak{B}_{n+1}, \mathfrak{R}_{n+1}) = (i+1,j) \mid (\mathfrak{B}_n, \mathfrak{R}_n) = (i,j) \right) &= \frac{b(i)}{b(i)+r(j)},\ \text{and}\\
P\left( (\mathfrak{B}_{n+1}, \mathfrak{R}_{n+1}) = (i,j+1) \mid (\mathfrak{B}_n, \mathfrak{R}_n) = (i,j) \right) &= \frac{r(j)}{b(i)+r(j)},\ \ i,j\ge 0.
\end{align*}
If we consider this as being generated by drawing red/blue balls from an urn, then $\mathfrak{B}_n$ and $\mathfrak{R}_n$ will be the numbers of blue and red balls, respectively, drawn from the urn up to time $n$.

For $* \in \{-,+,0\}$ we will let $\{(\mathfrak{B}^*_n, \mathfrak{R}^*_n)\}_{n\geq 0}$ be the generalized P\'olya urn process corresponding to the sequences $\{b^*(i)\}_{i\geq 0}$ and $\{r^*(i)\}_{i\geq 0}$ 
\[
\begin{array}{l}
 b^-(i) = w(2i) \\ 
 r^-(i) = w(2i+1)
\end{array},
\qquad 
\begin{array}{l}
 b^+(i) = w(2i+1) \\ 
 r^+(i) = w(2i)
\end{array},
\quad\text{and}\quad  
\begin{array}{l}
 b^0(i) = w(2i) \\ 
 r^0(i) = w(2i)
\end{array},
\quad \text{for } i\geq 0. 
 \]
With these choices of parameters, at any site $x$ to the left of the origin the sequence of left/right steps on successive visits to $x$ has the same distribution as the sequence of blue/red draws from the generalized P\'olya urn process $\{(\mathfrak{B}^-_n, \mathfrak{R}^-_n)\}_{n\geq 0}$. Similarly, the process $\{(\mathfrak{B}^+_n, \mathfrak{R}^+_n)\}_{n\geq 0}$ corresponds to left/right steps at sites to the right of the origin and $\{(\mathfrak{B}^0_n, \mathfrak{R}^0_n)\}_{n\geq 0}$ corresponds to left/right steps at the origin.

When considering any of the urn models described above, we will let $\tau_k^\mathfrak{B^*}$ (or $\tau_k^\mathfrak{R^*}$) denote the number of trials until a blue (or red) ball is selected for the $k$-th time. More explicitly, letting $\tau_0^\mathfrak{B^*} = \tau_0^{\mathfrak{R^*}}=0$ we have for $k\geq 1$ that 
\begin{equation}\label{taub}
\tau_k^\mathfrak{B^*} = \inf\{ n> \tau_{k-1}^\mathfrak{B^*}:  \mathfrak{B}^*_n = \mathfrak{B}^*_{n-1} + 1 \} 
\  \text{ and } \ 
\tau_k^\mathfrak{R^*} = \inf\{ n> \tau_{k-1}^{\mathfrak{R}^*}:  \mathfrak{R}^*_n = \mathfrak{R}^*_{n-1} + 1 \}. 
\end{equation}
We will also let $\D^*_n$ be signed difference in the number of red balls and blue balls drawn in the first $n$ steps of the urn process. That is, 
\begin{equation}\label{discr}
 \D^*_n = \mathfrak{R}^*_n - \mathfrak{B}^*_n. 
\end{equation}

\noindent\textbf{Rubin's construction.} It will be helpful at times to use an
equivalent construction of a generalized P\'olya urn process (due to
Rubin, see \cite{Dav90}) using exponential random variables.  Suppose that
$B_1,B_2,\ldots,R_1,R_2,\ldots$ are independent random variables with
$B_i \sim \text{Exp}(b(i-1))$ and $R_i \sim \text{Exp}(r(i-1))$, for
$i\geq 1$.  We then make a red mark on $(0,\infty)$ at
$\sum_{i=1}^k R_i$ for every $k\geq 1$ and similarly a blue mark on
$(0,\infty)$ at $\sum_{i=1}^k B_i$ for every $k\geq 1$.  The urn
process can then be constructed by reading off the sequence of red and
blue marks in order: every red mark corresponds to drawing a red ball  ($\mathfrak{R}$ increases by one) and every blue mark
corresponds to drawing a blue ball ($\mathfrak{B}$
increases by one).

\begin{lemma} \label{Lemma1} Let $w$ be as in \eqref{w} with either
  (1) $\alpha = 0$ or (2) $w(n) = (n+1)^{-\alpha}$ for some
  $\alpha> 0$. There exist constants $C,c>0$ such that for any
  $*\in\{-,+,0\}$
    \[P(|\D^*_{\tau_{n}^{\B^*}}| \ge m )\le C e^{\frac{-c m^2}{m\vee n}} \qquad \forall n,m\in \N.\]
\end{lemma}

\begin{proof}
  Without loss of generality we take $*=+$ and drop it from the
  notation.  We will only give an upper bound for the right tail
  probabilities $P( \mathfrak{D}_{\tau_n^\mathfrak{B}} \geq m)$ since
  the left tail probabilities
  $P( \mathfrak{D}_{\tau_n^\mathfrak{B}} \leq -m)$ can be handled
  similarly.  It follows from Rubin's construction that
\begin{equation}\label{Dtn-Rubin}
 P(\D_{\tau_n^{\mathfrak{B}}} \ge m ) = P\left( \sum_{i=1}^{n+m} R_i < \sum_{i=1}^n B_i \right) = P\left(  \sum_{i=1}^n B_i - \sum_{i=1}^{n+m} R_i > 0 \right),
\end{equation}
where the exponential random variables in Rubin's construction here have distribution $R_i\sim \text{Exp}(w(2i-2))$ and $B_i \sim \text{Exp}(w(2i-1))$ for all $i\geq 1$.
We will control the probability above slightly differently in the cases $\alpha = 0$ and $\alpha > 0$, respectively.

\smallskip

\noindent\textbf{Case I: $\alpha = 0$. }
Recalling the definitions of $U_1$ and $V_1$ in \eqref{UVr}, we have that 
$E\left[\sum_{i=1}^n B_i\right] - E\left[ \sum_{i=1}^{n+m} R_i \right]
= V_1(n) - U_1(n+m)$.  Therefore, we can write
\begin{equation}\label{Dtn-Rubin2}
 P(\D_{\tau_n^{\mathfrak{B}}} \ge m ) = P\left(  \sum_{i=1}^n (B_i-E[B_i]) - \sum_{i=1}^{n+m} (R_i-E[R_i]) > U_1(n+m)-V_1(n) \right).
\end{equation}
The assumption that $\alpha = 0$ implies that the $w(i)$ are uniformly
bounded above, and thus $U_1(n+m)-U_1(n) \geq \delta m$ for some
$\delta$.  Since $\eqref{v-u}$ implies that the difference $U_1(n) - V_1(n)$ is also bounded
for all $n$, we can conclude that there exist $m_0\in\N$ such
that $U_1(n+m)-V_1(n) \geq \delta m + U_1(n)-V_1(n) \geq \delta m/2$,
for all $m\geq m_0$.  Applying this to \eqref{Dtn-Rubin2} we have
\begin{equation*}\label{Dtn-cen}
 P(\D_{\tau_n^{\mathfrak{B}}} \ge m ) \leq P\left(  \sum_{i=1}^n (B_i - E[B_i]) + \sum_{i=1}^{n+m} (E[R_i]-R_i) > \frac{\delta m}{2} \right),
 \quad \forall n\geq 1, \,  m\geq m_0.  
\end{equation*}
Since $B_i$ and $R_i$ are exponential random variables whose
parameters are uniformly bounded away from 0 and $\infty$, it follows
that there exist constants $g,t_0>0$ so that we can bound the moment
generating functions of the centered random variables by
$\sup_i \left( E[e^{t(B_i - E[B_i])}] \vee E[e^{t(E[R_i]-R_i)}]
\right) \leq e^{g t^2/2}$ for all $|t|\leq t_0$.  Therefore, by
  \cite[Theorem III.15]{pSOIRV}, there exists $c>0$ such that for all
$n\geq 1$ and $m\geq m_0$,
$P(\D_{\tau_n^{\mathfrak{B}}} \ge m ) \leq \exp \{  \frac{- c m^2}{n\vee m} \}$.
By choosing a constant $C>0$ sufficiently large we have that $P(\D_{\tau_n^{\mathfrak{B}}} \ge m ) \leq C \exp\{\frac{-cm^2}{n\vee m} \}$ for all $n,m\geq 1$. 

\smallskip

\noindent\textbf{Case II: $w(n) = (n+1)^{-\alpha}$ for some $\alpha > 0$. }
In this case, note that the exponential random variables in
\eqref{Dtn-Rubin} have means $E[R_i] = (2i-1)^\alpha$ and
$E[B_i] = (2i)^\alpha$, $i\geq 1$.  Therefore, if we let
$(R_j')_{j\geq 1}$ be a sequence of independent exponential random
variables which is also independent from  $(B_i)_{i\geq 1}$ and such
that $R_i' = R_i$ for $i\leq n$ and
$R_i' \sim \text{Exp}(1/(2n+1)^\alpha)$ for $i > n$, we have from
\eqref{Dtn-Rubin} that
\begin{align*}
& P(\D_{\tau_n^{\mathfrak{B}}} \ge m ) \leq P\left(  \sum_{i=1}^n B_i - \sum_{i=1}^{n+m} R_i' > 0 \right)\\
 &= P\left(  \sum_{i=1}^n (B_i - E[B_i]) - \sum_{i=1}^{n+m} (R_i' - E[R_i']) > \sum_{i=1}^n \left( (2i-1)^\alpha - (2i)^\alpha \right) + m (2n+1)^\alpha \right).
\end{align*}
Since $\sum_{i=1}^n \left( (2i-1)^\alpha - (2i)^\alpha \right) \sim -2^{\alpha-1} n^{\alpha}$ as $n\to\infty$, it follows that there exists $m_1\in\N$ such that for all $m\geq m_1$ and $n\geq 1$ we have
\begin{align*}
 P(\D_{\tau_n^{\mathfrak{B}}} \ge m )
 &\leq P\left(  \sum_{i=1}^n (B_i - E[B_i]) -\sum_{i=1}^{n+m} (R_i' - E[R_i']) > 2^{\alpha-1} m n^\alpha \right) \\
 &=  P\left(  \sum_{i=1}^n \frac{B_i - E[B_i]}{n^\alpha} + \sum_{i=1}^{n+m} \frac{E[R_i']-R_i'}{n^\alpha} > 2^{\alpha-1} m \right).
\end{align*}
Since we can again obtain uniform bounds on the moment generating functions of the random variables in the sum above of the form 
\[
 \sup_i \left( E\Big[e^{t\frac{B_i-E[B_i]}{n^\alpha}} \Big] \vee E\Big[e^{t\frac{E[R'_i]-R'_i}{n^\alpha}} \Big] \right) \leq e^{\frac{1}{2}g t^2}, \quad \forall t \leq |t_0|, 
\]
for some $g,t_0>0$, 
we can finish the proof just as in the case $\alpha = 0$ above by using the large deviation bounds in \cite{pSOIRV}.  
\end{proof}

\begin{remark}\label{rem:tautail}
Since $\mathfrak{R}_n + \mathfrak{B}_n = n$, it follows that $\D_n = n - 2\mathfrak{B}_n$. Replacing $n$ with $\tau_n^{\mathfrak{B}}$ we get that 
$\D_{\tau_n^\mathfrak{B}} = \tau_n^\mathfrak{B} - 2n$. 
Thus, Lemma \ref{Lemma1} gives concentration bounds on $\tau_n^{\mathfrak{B}}$ as well. Moreover, since $|D_{n+1}-D_n|=1$, we also have that
\begin{align}
  \label{lip}
  |D_{\tau_n^\mathfrak{B}}-D_{2n}|&\le |\tau_n^{\mathfrak{B}}-2n|=|D_{\tau_n^{\mathfrak{B}}}|, \shortintertext{and, thus,} \label{stn} |D_{2n}|\le |D_{2n}-D_{\tau_n^\mathfrak{B}}|&+|D_{\tau_n^\mathfrak{B}}|\le |\tau_n^{\mathfrak{B}}-2n|+|D_{\tau_n^\mathfrak{B}}|=2|D_{\tau_n^\mathfrak{B}}|.
  \end{align}
\end{remark}

\subsection{Asymptotically free case: accumulated drift at a single site}\label{sec:drift}

\begin{proof}[Proof of Lemma \ref{lem:bdxfin}]
  Since the sum in $\bar\delta_x$ depends only on the behavior of the
  walk on successive visits to the site $x$, we can analyze
  $\bar\delta_x$ using one of the generalized P\'olya urn
  processes. We shall only treat the case $x>0$ as the proofs in the
  other cases are similar.  Thus, we will only be using the urn
  process $(\mathfrak{B}^+_n,\mathfrak{R}_n^+)$.  To simplify the
  notation, we will omit the superscript $+$ throughout the proof.

  We re-write $\bar\delta_x$ as follows: 
\begin{align}\label{bdx-urn}
  \bar\delta_x  
 &= \sum_{i=0}^\infty \left| \frac{w(r_x^i)-w(\ell_x^i)}{w(r_x^i)+w(\ell_x^i)} \right| \ind{X_i=x}
  \overset{\text{Law}}{=}
\sum_{n=0}^{\infty} \left| \frac{w(2\mathfrak{R}_n)-w(2\mathfrak{B}_n+1)}{w(2\mathfrak{R}_n)+w(2\mathfrak{B}_n+1)} \right|\\ & =\sum_{i=0}^\infty \left| \frac{1}{w(2\mathfrak{R}_n)} - \frac{1}{w(2\mathfrak{B}_n+1)} \right|\left(\frac{1}{w(2\mathfrak{R}_n)} + \frac{1}{w(2\mathfrak{B}_n+1)}\right)^{-1}.\nonumber
\end{align}
Since $\alpha = 0$,
$\left(\frac{1}{w(2\mathfrak{R}_n)} +
  \frac{1}{w(2\mathfrak{B}_n+1)}\right)^{-1} \leq \frac12\sup_i
w(i)<\infty$.  Hence, it is enough to prove that
\[
 E\left[ \left( \sum_{n=0}^\infty \left| \frac{1}{w(2\mathfrak{R}_n)} - \frac{1}{w(2\mathfrak{B}_n+1)} \right| \right)^M \right] < \infty, \quad \forall M>0.  
\]
To do this, we will show that the sum inside the expectation has
tails that decay faster than any polynomial.  Letting $C_0:=\sup_i (1/w(i))\in[1,\infty)$ we get
\begin{align}
& P\left( \sum_{n=0}^\infty \left| \frac{1}{w(2\mathfrak{R}_n)} - \frac{1}{w(2\mathfrak{B}_n+1)} \right| \geq 4C_0 m \right) \nonumber \\
&\qquad \leq P\left( \tau_m^{\mathfrak{B}} \geq 3m \right) + P\left( \sum_{n\geq \tau_m^{\mathfrak{B}}}  \left| \frac{1}{w(2\mathfrak{R}_n)} - \frac{1}{w(2\mathfrak{B}_n+1)} \right| \geq C_0 m \right). \label{adtail}
\end{align}
For the first term in \eqref{adtail}, it follows from Remark
\ref{rem:tautail} and Lemma \ref{Lemma1} that
$P\left( \tau_m^{\mathfrak{B}} \geq 3m \right) = P\left(
  \mathfrak{D}_{\tau_m^\mathfrak{B}} \geq m \right) \leq C e^{-cm}$.
Thus, it remains to show that the second term in \eqref{adtail}
decreases faster than any polynomial in $m$.  To this end, we first
rewrite
\begin{align*}
 \sum_{n\geq \tau_m^{\mathfrak{B}}}  \left| \frac{1}{w(2\mathfrak{R}_n)} - \frac{1}{w(2\mathfrak{B}_n+1)} \right|
&= \sum_{i=m}^\infty \sum_{n=\tau^{\mathfrak{B}}_{i}}^{\tau^{\mathfrak{B}}_{i+1}-1} \left| \frac{1}{w(2\mathfrak{R}_n)} - \frac{1}{w(2\mathfrak{B}_n+1)} \right| \\
&= \sum_{i=m}^\infty \sum_{n=\tau^{\mathfrak{B}}_{i}}^{\tau^{\mathfrak{B}}_{i+1}-1} \left| \frac{1}{w(2n-2i)} - \frac{1}{w(2i+1)} \right|, 
\end{align*}
where the last equality follows from the fact that $\mathfrak{R}_n = n-\mathfrak{B}_n$ for all $n$ and $\mathfrak{B}_n = i$ for every $n \in [\tau^{\mathfrak{B}}_{i}, \tau^{\mathfrak{B}}_{i+1})$. 
Next, we fix a parameter $0<\kappa' < \min\{p-\frac{1}{2}, \kappa \}$. 
If $|\mathfrak{D}_{\tau_i^\mathfrak{B}}| = |\tau_i^{\mathfrak{B}} - 2i| \leq (\log i)\sqrt{i}$ and $\tau_{i+1}^\mathfrak{B} - \tau_i^\mathfrak{B} \leq i^{\kappa'}$ for all $i\geq m$ then the last sum above does not exceed
\begin{align*}
& \sum_{i=m}^\infty i^{\kappa'} \max_{|n-2i|\leq 2(\log i)\sqrt{i} }  \left| \frac{1}{w(2n-2i)} - \frac{1}{w(2i+1)} \right| \\
&\leq \sum_{i=m}^\infty i^{\kappa'} \max_{|n-2i|\leq 2(\log i)\sqrt{i} }  \left| 2^p B \left( \frac{1}{(2n-2i)^p} - \frac{1}{(2i+1)^p} \right) + 
O\left( \frac{1}{i^{\kappa+1}} \right)
\right| \\
&\leq C \sum_{i=m}^\infty i^{\kappa'} \left( \frac{\log i}{i^{p+\frac{1}{2}}} + \frac{1}{i^{1+\kappa}} \right) . 
\end{align*}
Note that by our choice of $\kappa'$, the sum in the last line is finite and thus can be made arbitrarily small by taking $m$ sufficiently large. Thus, we conclude that for all sufficiently large $m$ 
\begin{align*}
& P\left( \sum_{n\geq \tau_m^{\mathfrak{B}}}  \left| \frac{1}{w(2\mathfrak{R}_n)} - \frac{1}{w(2\mathfrak{B}_n+1)} \right| \geq C_0 m \right) \\
&\qquad \leq P\left( |\mathfrak{D}_{\tau_i^\mathfrak{B}}| > (\log i)\sqrt{i} \text{ or } \tau_{i+1}^\mathfrak{B} - \tau_i^\mathfrak{B} > i^{\kappa'}, \text{ for some } i\geq m \right) \\
&\qquad \leq \sum_{i=m}^\infty \left\{ P\left( |\mathfrak{D}_{\tau_i^\mathfrak{B}}| > (\log i)\sqrt{i} \right) + P\left( \tau_{i+1}^\mathfrak{B} - \tau_i^\mathfrak{B} > i^{\kappa'} \right) \right\}. 
\end{align*}
The first probability in the last line is bounded by
$C e^{-c(\log i)^2}$ by Lemma \ref{Lemma1}. For the last one, the
assumption $\alpha = 0$ implies that the probability of the next draw
in the urn process being a blue ball is uniformly bounded below by
some $q>0$ so that $ \tau_{i+1}^\mathfrak{B} - \tau_i^{\mathfrak{B}} $
is stochastically dominated by a Geo($q$) random variable. In
particular, this implies that
$P\left( \tau_{i+1}^\mathfrak{B} - \tau_i^{\mathfrak{B}} > i^{\kappa'}
\right) \leq C e^{-c i^{\kappa'}}$ for some $C,c>0$. Thus, we have
\[
 P\left( \sum_{n\geq \tau_m^{\mathfrak{B}}}  \left| \frac{1}{w(2\mathfrak{R}_n)} - \frac{1}{w(2\mathfrak{B}_n+1)} \right| \geq C_0 m \right) \leq C \sum_{i=m}^\infty \left( e^{-c(\log i)^2} + e^{-c i^{\kappa'}} \right). 
\]
This completes the proof of the lemma. 
\end{proof}

\begin{proof}[Proof of Lemma \ref{lem:Edx}]
  Lemma \ref{lem:bdxfin} implies that $E[\delta_x]$ is finite for all
  $x \in \Z$. From this and the symmetry considerations in Remark
  \ref{rem:sym} we conclude that $E[\delta_0] = 0$ and
  $E[\delta_x] = - E[\delta_{-x}]$ for $x \neq 0$.  Therefore, we need
  only to prove that $E[\delta_x] = \gamma$ for all $x>0$. Using the
  connection with the generalized P\'olya urn processes as in
  \eqref{bdx-urn},\footnote{Since $x>0$, we only deal with the P\'olya
    urn process $(\mathfrak{B}^+_j, \mathfrak{R}^+_j)$, and thus we
    will omit the superscripts $+$.} we see that it is enough to show
  that
\begin{equation}\label{Edx-sc}
 E\left[ \sum_{j=0}^{\infty}  \frac{ w(2\mathfrak{R}_j)-w(2\mathfrak{B}_j+1)}{w(2\mathfrak{R}_j)+w(2\mathfrak{B}_j+1)}  \right] 
= \lim_{n\to\infty}E\left[ \D_{\tau_n^\mathfrak{B}} \right] 
= \gamma. 
\end{equation}

To justify the first equality in \eqref{Edx-sc}, set $\mathcal{F}_i^{\mathfrak{B},\mathfrak{R}} := \sigma\left( (\mathfrak{B}_j, \mathfrak{R}_j), \, j\leq i \right)$ and  note that
\begin{equation}\label{Ddrift}
 E\left[ \D_{i+1}-\D_i \mid \mathcal{F}_i^{\mathfrak{B},\mathfrak{R}} \right] = \frac{ w(2\mathfrak{R}_i) - w(2\mathfrak{B}_i+1)}{ w(2\mathfrak{R}_i) + w(2\mathfrak{B}_i+1) }. 
\end{equation}
Then, by  Lemma \ref{lem:bdxfin} and the dominated convergence theorem we get
\begin{align}
 E\left[ \sum_{i=0}^{\infty} \frac{ w(2\mathfrak{R}_i) - w(2\mathfrak{B}_i+1)}{ w(2\mathfrak{R}_i) + w(2\mathfrak{B}_i+1) } \right]  
&= \lim_{n\to\infty} E\left[  \sum_{i=0}^{\tau_n^\mathfrak{B}-1} E\left[ \D_{i+1}-\D_i \mid \mathcal{F}_i^{\mathfrak{B},\mathfrak{R}} \right] \right] \nonumber \\
&= \lim_{n\to\infty} E\left[  \sum_{i=0}^{\infty} E\left[ ( \D_{i+1}-\D_i ) \ind{i < \tau_n^\mathfrak{B} } \mid \mathcal{F}_i^{\mathfrak{B},\mathfrak{R}} \right]  \right], \label{Etd1}
\end{align}
where in the last equality we used that $\{ i < \tau_n^\mathfrak{B} \} \in \mathcal{F}_i^{\mathfrak{B},\mathfrak{R}}$ for any $i,n \geq 0$. 
For the expectations in the last line, note that for any fixed $n$,
\begin{align}
 E\left[  \sum_{i=0}^{\infty} E\left[ ( \D_{i+1}-\D_i ) \ind{i < \tau_n^\mathfrak{B} } \mid \mathcal{F}_i^{\mathfrak{B},\mathfrak{R}} \right]  \right]
&= \sum_{i=0}^{\infty} E\left[   E\left[ ( \D_{i+1}-\D_i ) \ind{i < \tau_n^\mathfrak{B} } \mid \mathcal{F}_i^{\mathfrak{B},\mathfrak{R}} \right]  \right] \nonumber \\
&= \sum_{i=0}^{\infty} E\left[ (\D_{i+1}-\D_i) \ind{i < \tau_n^\mathfrak{B}} \right] \nonumber \\
&= E\left[ \sum_{i=0}^{\infty} (\D_{i+1}-\D_i) \ind{i < \tau_n^\mathfrak{B}} \right] \label{Etd2}
= E\left[ \D_{\tau_n^\mathfrak{B}} \right],  
\end{align}
where the interchange of the expectation and sum in the first and third equality is justified by the dominated convergence theorem, since $|\D_{i+1}-\D_i|=1$ and $E[\tau_n^\mathfrak{B}] < \infty$ for all $n$ (by Remark~\ref{rem:tautail}). 
Combining \eqref{Etd1} and \eqref{Etd2} proves the first equality in \eqref{Edx-sc}. 

For the second equality in \eqref{Edx-sc}, recall the definitions of $U_1(n)$ and $V_1(n)$ in \eqref{UVr}, and note that it follows from \cite[Lemma 1]{tGRK} that for all $n\geq 1$,
\[
 E\left[ U_1(\mathfrak{R}_{\tau_n^\mathfrak{B}} ) \right] 
= E\left[ U_1(\D_{\tau_n^\mathfrak{B}} +n ) \right] 
= E\left[ \sum_{i=0}^{\D_{\tau_n^\mathfrak{B}} +n - 1} \frac{1}{w(2i)} \right]  = \sum_{i=0}^{n-1} \frac{1}{w(2i+1)} = V_1(n).
\]
Subtracting $U_1(n)$ from both sides we get 
\begin{align}
 V_1(n) - U_1(n) 
&= E\left[ \sum_{i=0}^{\D_{\tau_n^\mathfrak{B}} + n - 1} \frac{1}{w(2i)}  - \sum_{i=0}^{n-1} \frac{1}{w(2i)} \right] \label{VUdiff} \\
&= E\left[ \D_{\tau_n^\mathfrak{B}}  + \sgn(\D_{\tau_n^\mathfrak{B}} ) \sum_{i= \min\{ n + \D_{\tau_n^\mathfrak{B}} , n \} }^{\max\{ n + \D_{\tau_n^\mathfrak{B}} , n \} - 1}  \left\{ \frac{1}{w(2i)} - 1 \right\}  \right]. \nonumber
\end{align}
To handle the sum inside the expectation on the right, note that since $\alpha = 0$ all the terms in the sum are uniformly bounded and that if $|\D_{\tau_n^\mathfrak{B}} | < (\log n)\sqrt{n}$ then (for $n$ large enough) all of the terms inside the sum are at most $C n^{-p}$ for some $C>0$. Thus, we can conclude that 
\[
E\left[ \sum_{i= \min\{ n + \D_{\tau_n^\mathfrak{B}} , n \} }^{\max\{ n + \D_{\tau_n^\mathfrak{B}} , n \} - 1}   \left| \frac{1}{w(2i)} - 1 \right|  \right] \leq O\left( \frac{\log n}{n^{p-\frac{1}{2}}} \right) + C E\left[ |\D_{\tau_n^\mathfrak{B}}| \ind{| \D_{\tau_n^\mathfrak{B}} | \geq (\log n)\sqrt{n} } \right], 
\]
and Lemma \ref{Lemma1} implies that the last term on the right vanishes as $n\to \infty$. 
Therefore, if $p > 1/2$ then $V_1(n) - U_1(n) = E[\D_{\tau_n^\mathfrak{B}}] + o(1)$. Taking $n\to\infty$ and using the definition of $\gamma$ in  \eqref{v-u} completes the proof of the second equality in \eqref{Edx-sc}. 
\end{proof}

\subsection{Polynomially self-repelling case}\label{psrurn}
Throughout this subsection we will be dealing only with the P\'olya
urn process $(\mathfrak{B}^+_j,\mathfrak{R}^+_j)$ and so the
superscripts $+$ will be suppressed.  Recall also that throughout this
subsection $w(n) = (n+1)^{-\alpha}$ where  $\alpha>0$.

Our main goal is to prove Proposition \ref{lemC1} using the
generalized P\'olya urns as a tool. Thus, it is important to first
explain the connection between the $\mathcal{E}$-processes in
Proposition \ref{lemC1} and the generalized P\'olya urn processes.  It
is not hard to see that for any fixed $\ell \geq 0$, the process
$\left(\mathcal{E}^{\mathcal{T}_\ell}(j) \right)_{j\geq 0}$ is a time
inhomogeneous Markov chain started from
$\mathcal{E}^{\mathcal{T}_\ell}(0)= \ell+1$ with transition
probabilities related to the discrepancy urn process $\mathfrak{D}$ as
follows.  Since at time $\mathcal{T}_\ell$ the walk has just completed
a step from $0$ to $1$, the last visit to site $1$ before time
$\mathcal{T}_\ell$ resulted in a jump to the left and this was the
$\ell$-th jump to the left from $1$. Thus,
$\mathcal{E}^{\mathcal{T}_\ell}(1)$ is equal to the number of right
steps from $1$ before the $\ell$-th left jump, and since the sequence
of left/right jumps at each site can be generated by a generalized
P\'olya urn, it follows that the increment
$ \mathcal{E}^{\mathcal{T}_\ell}(1) -
\mathcal{E}^{\mathcal{T}_\ell}(0)$ has the same distribution as
$\mathfrak{R}_{\tau_\ell^{\B}} - (\ell+1) = \D_{\tau_\ell^{\B}}-1$.
Similarly, the last visit to any site $x\geq 2$ before time
$\mathcal{T}_\ell$ was a step to the left and the number of such left
steps from $x$ is
$\mathcal{D}^{\mathcal{T}_\ell}(x) =
\mathcal{E}^{\mathcal{T}_\ell}(x-1)$.  By similar reasoning as above,
it follows that conditioned on
$\{\mathcal{E}^{\mathcal{T}_\ell}(j)=n\}$ the increment
$ \mathcal{E}^{\mathcal{T}_\ell}(j+1) -
\mathcal{E}^{\mathcal{T}_\ell}(j)$ has the same distribution as
$\D_{\tau_n^{\B}}$.  That is,
\begin{equation}\label{EdiffDtn}
 P\left( \mathcal{E}^{\mathcal{T}_\ell}(j+1) - \mathcal{E}^{\mathcal{T}_\ell}(j) = k \mid \mathcal{E}^{\mathcal{T}_\ell}(j)=n \right) 
  = P\left( \D_{\tau_n^{\B}} = k \right), \qquad j\geq 1.
\end{equation}

Proposition \ref{lemC1} concerns processes of the form
$\left(\mathcal{E}^{\mathcal{T}_{(M+1-c)N}}(j) -
  \mathcal{E}^{\mathcal{T}_{(M+c)N}}(j) \right)_{j\geq 0}$. By similar
reasoning as above, we see that the conditional distributions of the
increments of this process have the same distributions as
$\D_{\tau_m^{\B}} - \D_{\tau_n^{\B}}$ with $m$ and $n$ determined by
the conditioning. Thus, the majority of this subsection will be
devoted to obtaining good estimates on the mean and variance of
$\D_{\tau_m^{\B}} - \D_{\tau_n^{\B}}$ when $m$ and $n$ are both large.

\subsubsection{Variance estimates}

The estimates on $\Var(\D_{\tau_m^\B} - \D_{\tau_n^\B})$ that we will
need for the proof of Proposition \ref{lemC1} will be given in
Proposition \ref{VarDtmDtn}, but since we will take a somewhat winding
path to prove this, it is helpful to give the reader an outline of
where we are headed.  First of all, instead of analyzing differences
of the discrepancy process $\D$ at random stopping times, we will use
a martingale approximation to obtain estimates on differences of the
discrepancy process at deterministic times. This is accomplished in
Lemma \ref{lemkey-var} and Corollaries
\ref{cor-VarDnlim}--\ref{cor-VDn}.  The next task is then in
justifying the change from the discrepancy process at the stopping
times $\tau_n^\B$ to the deterministic time $2n$. In particular, in
Lemma \ref{DtnD2n-diff} we will obtain estimates on
$E[(\D_{\tau_n^\B} - \D_{2n} )^2 ]$, and combining this with the
previous results we obtain the asymptotics of
$\Var(\D_{\tau_m^\B} - \D_{\tau_n^\B} )$ in Proposition
\ref{VarDtmDtn}.

 Obtaining estimates on the variance of $\D_m-\D_n$ directly is
 not optimal because the increments
 $\D_{i+1}-\D_i$ are too strongly correlated. On the other hand, as
 can be seen from the proof of the next result, for $\delta>0$ fixed
 and $n$ large the process
 $\left(\frac{\D_{\fl{tn}}}{\sqrt{n}} t^\alpha \right)_{t\geq \delta}$
 is approximately a martingale, and this observation underlies the
 following lemma.

 \begin{lemma} \label{lemkey-var} For $\delta\in (0,2)$ fixed, there
   exists a constant $C_\delta$ such that for $n$ sufficiently large
   and $\delta n \leq k \leq m \leq 2n$ we have
\[
\left| \Var\left( \D_m \left( \frac{m}{n} \right)^\alpha - \D_k \left( \frac{k}{n}\right)^\alpha \right) - n \int_{k/n}^{m/n} u^{2\alpha} \, du \right| 
\leq C_\delta \sqrt{n}\log^2 n. 
\]
\end{lemma}
\begin{proof}
  Throughout the proof we will let $C_\delta$ denote a constant which
  depends on $\delta$ but not on $n$ and which can change from line to
  line.  Let
  $ \Delta_i = (\D_{i+1} - \D_{i}) - E\left[ \D_{i+1} - \D_{i} \mid
    \mathcal{F}^{\mathfrak{B,R}}_i \right]$. Note that by
  \eqref{Ddrift} and our choice of $w$,
\begin{align}
 E\left[ \D_{i+1} - \D_{i} \mid \mathcal{F}^{\mathfrak{B,R}}_i \right]
&=\frac{(2\mathfrak{B}_i+2)^\alpha - (2\mathfrak{R}_i+1)^\alpha}{(2\mathfrak{B}_i+2)^\alpha + (2\mathfrak{R}_i+1)^\alpha}
  \nonumber \\
&= \frac{(i-\D_i+2)^\alpha -
  (i+\D_i+1)^\alpha }{(i-\D_i+2)^\alpha + (i+\D_i+1)^\alpha }
= \frac{\left(1-\frac{\D_i-2}{i}\right)^\alpha - \left(1+\frac{\D_i+1}{i}\right)^\alpha
}{\left(1-\frac{\D_i-2}{i}\right)^\alpha + \left(1+\frac{\D_i+ 1}{i}\right)^\alpha } \nonumber\\ 
&=: - \tfrac{\alpha\, \D_i}{i} + \epsilon^n_i, 
\qquad \text{where } |\epsilon^n_i|\leq C\,\tfrac{\D_i^2 + i}{i^2}. 
 \label{EDform}
\end{align}
Therefore, we have $\D_{i+1}= \D_i +  \Delta_i - \frac{\alpha\, \D_i}{i}+\epsilon^n_i$ and
\begin{align*}
 \D_{i+1}\left(\tfrac{i+1}{n}\right)^\alpha - \D_i \left( \tfrac{i}{n} \right)^\alpha 
 &= \left( \D_i +  \Delta_i - \tfrac{\alpha \D_i}{i} + \epsilon^n_i \right)   \left( \tfrac{i+1}{n}\right)^\alpha - \D_{i} \left( \tfrac{i}{n} \right)^\alpha\\
 &= \Delta_i \left( \tfrac{i+1}{n} \right)^{\alpha} 
+ \D_i \left\{ \left(\tfrac{i+1}{n}\right)^\alpha\left( 1 -\tfrac{\alpha}{i} \right) -  \left( \tfrac{i}{n} \right)^\alpha \right\} 
+ \epsilon^n_i \left(\tfrac{i+1}{n}\right)^\alpha. 
\end{align*}
Summing over $i \in [k,m)$ gives 
\begin{align}
& \D_m \left( \tfrac{m}{n} \right)^\alpha - \D_k \left( \tfrac{k}{n}\right)^\alpha \nonumber \\
&\qquad = \sum_{i=k}^{m-1} \left\{ \Delta_i \left( \tfrac{i+1}{n} \right)^{\alpha} 
+ \D_i \left( \left(\tfrac{i+1}{n}\right)^\alpha\left( 1 -\tfrac{\alpha}{i} \right) -  \left( \tfrac{i}{n} \right)^\alpha \right) 
+ \epsilon^n_i \left(\tfrac{i+1}{n}\right)^\alpha \right\}. \label{mgd}
\end{align}
To get asymptotics on the variance of this sum, we will first get good bounds on the variances of the sums of the three terms inside the braces separately. That is, we will show that 
\begin{align}
 &\max_{\delta n \leq k \leq m \leq 2n} \left| \Var\left( \sum_{i=k}^{m-1} \Delta_i \left( \frac{i+1}{n} \right)^{\alpha} \right) - n \int_{k/n}^{m/n} u^{2\alpha} \, du \right|
 \leq C_\delta n^{1/5}, \label{vart1}\\
 &\max_{\delta n \leq k \leq m \leq 2n} \Var\left( \sum_{i=k}^{m-1} \D_i \left( \left(\frac{i+1}{n}\right)^\alpha\left( 1 -\frac{\alpha}{i} \right) -  \left( \frac{i}{n} \right)^\alpha \right)  \right) \leq C_\delta (\log^2 n)/n, \label{vart2} \\
 &\max_{\delta n \leq k \leq m \leq 2n} \Var\left( \sum_{i=k}^{m-1} \epsilon_i^n \left( \frac{i+1}{n} \right)^\alpha \right) \leq C_\delta\log^4 n.  \label{vart3}
\end{align}
To see that these bounds are enough to finish the proof, first note
that by expanding the variance of \eqref{mgd} into variance and
covariance terms, bounding the covariance terms by products
  of square roots of the variances, and using the variance bounds in
\eqref{vart2} and \eqref{vart3}, one obtains for $n$ sufficiently
large and $\delta n \leq k \leq m \leq 2n$ that
\begin{multline*}
\left| \Var \left( \D_m \left( \frac{m}{n} \right)^\alpha - \D_k \left( \frac{k}{n}\right)^\alpha \right) - \Var\left( \sum_{i=k}^{m-1} \Delta_i \left( \frac{i+1}{n} \right)^\alpha \right) \right| \\
 \leq C_\delta\log^4 n+ C_\delta\log^2 n \sqrt{ \Var\left( \sum_{i=k}^{m-1} \Delta_i \left( \frac{i+1}{n} \right)^\alpha \right) }.  
\end{multline*}
The proof is then finished by using the variance bounds in \eqref{vart1}. 

It remains now to prove the variance bounds in \eqref{vart1}--\eqref{vart3}. 

\noindent\textbf{Proof of \eqref{vart1}:} 
Since the terms $\Delta_i \left( \frac{i+1}{n} \right)^{\alpha}$ form the increments of a martingale we have that 
\begin{align*}
\Var\left( \sum_{i=k}^{m-1} \Delta_i \left( \frac{i+1}{n} \right)^{\alpha} \right) 
&= \sum_{i=k}^{m-1} \left( \frac{i+1}{n} \right)^{2\alpha} E\left[ \Delta_i^2 \right] \\
&= \sum_{i=k}^{m-1} \left( \frac{i+1}{n} \right)^{2\alpha} \left( 1 - E\left[ E\left[\D_{i+1}-\D_i \mid \mathcal{F}_i^{\B,\mathfrak{R}} \right]^2 \right] \right). 
\end{align*}
Note that an integral approximation shows that
$\sum_{i=k}^{m-1} \left( \frac{i+1}{n} \right)^{2\alpha} = n
\int_{k/n}^{m/n} u^{2\alpha} \, du + O(1)$, where the $O(1)$ error
term is uniform over all $\delta n \leq k \leq m \leq 2n$. To finish
the proof of \eqref{vart1} we need to get bounds on
$E\left[ E\left[\D_{i+1}-\D_i \mid \mathcal{F}_i^{\B,\mathfrak{R}}
  \right]^2 \right]$ that are uniform over $i \in [\delta n, 2n]$.
Recalling \eqref{EDform}, we conclude that there exists a constant
$C_\delta$ such that if $i \in [\delta n, 2n]$ and
$|\D_i| \leq n^{3/5}$ then
$E\left[ \D_{i+1}-\D_i \mid \mathcal{F}_i^{\B,\mathfrak{R}} \right]
\leq C_\delta n^{-2/5}$. Then, since $|\D_{i+1}-\D_i| = 1$, we get
\[
\max_{i \in [\delta n, 2n]} E\left[ E\left[\D_{i+1}-\D_i \mid \mathcal{F}_i^{\B,\mathfrak{R}} \right]^2 \right]
\leq \max_{i \in [\delta n, 2n]} P\left( |\D_i| > n^{3/5} \right)+ C^2_\delta n^{-4/5}.
\]
An application of \eqref{stn} and Lemma \ref{Lemma1} completes the proof of
  \eqref{vart1}.

\noindent\textbf{Proof of \eqref{vart2}:} 
Since it is easy to see that $\max_{\fl{\delta n}\leq i < 2n} \left| \left(\frac{i+1}{n}\right)^\alpha\left( 1 -\frac{\alpha}{i} \right) -  \left( \frac{i}{n} \right)^\alpha \right| \leq \frac{C_\delta}{n^2}$, we have
\begin{align*}
&\max_{\delta n \leq k < m \leq 2n} \Var\left( \sum_{i=k}^{m-1} \D_i \left( \left(\frac{i+1}{n}\right)^\alpha\left( 1 -\frac{\alpha}{i} \right) -  \left( \frac{i}{n} \right)^\alpha \right)  \right)  \\
&\qquad\leq \max_{\delta n \leq k < m \leq 2n} E\left[ \left( \sum_{i=k}^{m-1} \D_i \left( \left(\frac{i+1}{n}\right)^\alpha\left( 1 -\frac{\alpha}{i} \right) -  \left( \frac{i}{n} \right)^\alpha \right) \right)^2 \right] \\ &\qquad \leq C_\delta\,n^{-2} E\left[ \max_{i \in [\delta n, 2n]} \D_i^2 \right],   
\end{align*}
and \eqref{vart2} then follows from this, \eqref{stn}, and the tail
bounds in Lemma \ref{Lemma1}.

\noindent\textbf{Proof of \eqref{vart3}:} 
Recalling that $|\epsilon_i^n| \leq C (\D_i^2 + i)/i^2$ for all $i$ we have that 
\begin{align}
\max_{\delta n \leq k < m \leq 2n}  \Var\left( \sum_{i=k}^{m-1} \epsilon^n_i \left(\frac{i+1}{n}\right)^\alpha \right)
&\leq \max_{\delta n \leq k < m \leq 2n} E\left[ \left(  \sum_{i=k}^{m-1} |\epsilon^n_i| \left(\frac{i+1}{n}\right)^\alpha \right)^2 \right]  \nonumber \\
&\leq C_\delta \left( n^{-2}E\left[ \max_{\delta n \leq i < 2n} \D_i^4 \right] + 1 \right),   \nonumber
\end{align}
and \eqref{vart3} then follows from this, \eqref{stn}, and the tail bounds in Lemma \ref{Lemma1}. 
\end{proof}

Naively taking $k=0$ and $m=n$ in Lemma \ref{lemkey-var} suggests the approximation $\Var(\D_n) \approx n \int_0^1 u^{2\alpha} \, du = \frac{n}{2\alpha+1}$. While this does not follow directly from Lemma \ref{lemkey-var} due to the requirement that $k\geq \delta n$, the following corollary shows that this approximation is, in fact, valid. 
\begin{corollary}\label{cor-VarDnlim}
$\displaystyle\lim\limits_{n\to\infty} \frac{\Var(\D_n)}{n} = \dfrac{1}{2\alpha+1}$. 
\end{corollary}
\begin{proof}
First of all, for any fixed $\delta>0$ note that 
\begin{align*}
\Var(\D_n) &= \Var\left( \D_n - \D_{\fl{\delta n}} \left(\fl{\delta n}/n \right)^\alpha \right) + \left( \fl{\delta n}/n \right)^{2\alpha} \Var(\D_{\fl{\delta n}} ) \\
&\qquad + 2 \Cov\left( \D_n - \D_{\fl{\delta n}} \left( \fl{\delta n}/n \right)^\alpha, \D_{\fl{\delta n}} \left( \fl{\delta n}/n \right)^\alpha \right). 
\end{align*}
Note that it follows from \eqref{stn} and Lemma \ref{Lemma1} that
there is a constant $C>0$ (not depending on $\delta$ or $n$) such that
$\Var(\D_{\fl{\delta n}}) \leq C \delta n$. By Lemma \ref{lemkey-var},  there
  is a constant $C_\delta>0$ such that for $n$ sufficiently large
\begin{align*}
\left|  \frac{\Var(\D_n)}{n}\right. &\left.- \int_0^1 u^{2\alpha} \, du \right|
\leq \left| \frac{\Var\left( \D_n - \D_{\fl{\delta n}} \left( \fl{\delta n}/n \right)^\alpha \right)}{n} - \int_\delta^1 u^{2\alpha} \, du \right| + \int_0^\delta u^{2\alpha} \, du \\ &+ C \delta^{2\alpha+1}  
+ 2n^{-1} \sqrt{\Var\left( \D_n - \D_{\fl{\delta n}} \left( \fl{\delta n}/n \right)^\alpha \right)}\sqrt{\Var\left(\D_{\fl{\delta n}} \left( \fl{\delta n}/n \right)^\alpha \right)} \\
&\leq C_\delta \,n^{-1/2}\log^2 n + C \delta^{2\alpha+1} + 2 \sqrt{ C + C_\delta n^{-1/2}\log^2 n} \, \sqrt{ 2C \delta^{2\alpha+1} }. 
\end{align*}
Taking $n\to\infty$ and then $\delta \to 0$ then completes the proof. 
\end{proof}

Finally, combining Lemma \ref{lemkey-var} and Corollary \ref{cor-VarDnlim} yields the following bounds on $\Var(\D_m-\D_n)$. 
\begin{corollary}\label{cor-VDn}
There exists a constant $C>0$ such that for all $n$ sufficiently large and $m \in [n,2n]$, 
\begin{equation}\label{VDmDn}
\left| \Var(\D_m - \D_n) - (m-n) \right| \leq   C\sqrt{n}\log^2 n  +  C (m-n)^2/n.
\end{equation}
\end{corollary}
\begin{proof}
Using the decomposition
\[
 \D_m - \D_n = \left( \tfrac{m}{n} \right)^{-\alpha} \left( \D_m \left(\tfrac{m}{n}\right)^\alpha - \D_n \right) + \left( \left( \tfrac{m}{n} \right)^{-\alpha} - 1 \right) \D_n, 
\]
we have that 
\begin{align*}
\Var(\D_m - \D_n)
&= \left( \tfrac{m}{n} \right)^{-2\alpha} \Var\left( \D_m \left( \tfrac{m}{n} \right)^\alpha - \D_n \right) 
+ \left( \left( \tfrac{m}{n} \right)^{-\alpha} - 1 \right)^2 \Var(\D_n) \\
&\qquad + 2 \left(\tfrac{m}{n} \right)^{-\alpha} \left( \left( \tfrac{m}{n} \right)^{-\alpha} - 1 \right) \Cov\left( \D_m\left( \tfrac{m}{n}\right)^\alpha - \D_n , \D_n   \right).
\end{align*}
By Lemma \ref{lemkey-var} and Corollary \ref{cor-VDn} the two variance terms on the right can be approximated by $n \int_1^{m/n} u^{2\alpha} \, du = \frac{n}{2\alpha+1} \left((m/n)^{2\alpha+1}-1\right)$ and $\frac{n}{2\alpha+1}$, respectively. 
With these approximations, the sum of the first two terms is approximated by  
\[
 \left( \tfrac{m}{n} \right)^{-2\alpha} \tfrac{n}{2\alpha+1} \left(\left( \tfrac{m}{n} \right)^{2\alpha+1}-1 \right) 
+ \left( \left( \tfrac{m}{n} \right)^{-\alpha} - 1 \right)^2 \left( \tfrac{n}{2\alpha+1} \right) 
= \tfrac{n}{2\alpha+1} \left( 1 + \tfrac{m}{n} - 2 \left(\tfrac{m}{n}\right)^{-\alpha} \right). 
\]
Therefore, we can conclude for $n$ large enough and $m \in [n,2n]$ that
\begin{align*}
&\left| \Var(\D_m - \D_n) - \tfrac{n}{2\alpha+1} \left( 1 + \tfrac{m}{n} - 2 \left(\tfrac{m}{n}\right)^{-\alpha} \right) \right| \\
&\leq \left( \tfrac{m}{n} \right)^{-2\alpha} \left| \Var\left( \D_m \left( \tfrac{m}{n} \right)^\alpha - \D_n \right) - n \int_{1}^{m/n} u^{2\alpha} \, du \right|  
 + \left( \left( \tfrac{m}{n} \right)^{-\alpha} - 1 \right)^2 \left|\Var(\D_n) - \tfrac{n}{2\alpha+1} \right| \\
&\qquad + 2 \left(\tfrac{m}{n} \right)^{-\alpha} \left( \left( \tfrac{m}{n} \right)^{-\alpha} - 1 \right) \Cov\left( \D_m\left( \tfrac{m}{n}\right)^\alpha - \D_n , \D_n   \right) \\
&\leq C  \sqrt{n}\log^2 n  + Cn \left( \left( \tfrac{m}{n} \right)^{-\alpha} - 1 \right)^2 + C \Cov\left( \D_m\left( \tfrac{m}{n}\right)^\alpha - \D_n , \D_n   \right). 
\end{align*}
For the covariance term in the last line above, recall the decomposition in \eqref{mgd} and note that $\Cov(\Delta_i, \D_n) = 0$ for all $i\geq n$. 
Thus, 
\begin{align*}
&\left|   
\Cov\left( \D_m\left( \frac{m}{n}\right)^\alpha - \D_n , \D_n   \right) 
\right| \\ 
&\leq \sqrt{\Var\left(  \sum_{i=n}^{m-1} \D_i \left( \left( \frac{i+1}{n} \right)^\alpha\left(1-\frac{\alpha}{i} \right) - \left(\frac{i}{n} \right)^\alpha\right)  \right)} \sqrt{\Var(\D_n)} \\
&\hspace{2in} + \sqrt{\Var\left( \sum_{i=n}^{m-1} \epsilon_i^n \left(\frac{i+1}{n}\right)^\alpha   \right)} \sqrt{\Var(\D_n)} 
\leq C  \sqrt{n}\log^2 n, 
\end{align*}
where in the last inequality we used the variance bounds from \eqref{vart2} and \eqref{vart3}. 

Thus far we have shown that for all $m\in [n,2n]$ and $n$ sufficiently large,
\[
\left| \Var(\D_m - \D_n) - \tfrac{n}{2\alpha+1} \left( 1 + \tfrac{m}{n} - 2 \left(\tfrac{m}{n}\right)^{-\alpha} \right) \right|
\leq C  \sqrt{n}\log^2 n  + Cn \left( \left( \tfrac{m}{n} \right)^{-\alpha} - 1 \right)^2.
\]
Finally, we note that there is a $C>0$ such that uniformly over all $m\ge n\ge 1$
\[
\left| \tfrac{1}{2\alpha+1}\,\left(1+\tfrac{m}{n} - 2\left(\tfrac{m}{n}\right)^{-\alpha} \right) - \tfrac{m-n}{n} \right| \leq C \left(\tfrac{m-n}{n} \right)^2
\quad\text{and}\quad
\left( \left(\tfrac{m}{n} \right)^{-\alpha} -1 \right)^2 \leq C \left(\tfrac{m-n}{n} \right)^2.
\]
From this the inequality in \eqref{VDmDn} follows easily. 
\end{proof}

The next step in obtaining bounds on $\Var(\D_{\tau_m^\B} - \D_{\tau_n^\B})$ is to obtain control on the difference $\D_{\tau_n^\B} - \D_{2n}$. For this, the following lemma will be useful. 

 \begin{lemma} \label{elenatype} There exist
   $M_0, c_1\in (0, \infty) $ such that for all $M \ge M_0 $,
 $n\in\N$, and $y\in[0,\sqrt{n}]$
\[
  P\bigg(\sup_{\tau^{\B}_{Mn} \le i \leq \tau^{\B}_{(M+1)n}} |\D_i -
    \D_{\tau^{\B}_{Mn}}|\ge y\sqrt{n}\bigg) \le \frac{1}{c_1}\,e^{-c_1
    y^2}.
\]
\end{lemma}

The proof of this Lemma is given in Appendix~\ref{app:psr}.

\begin{lemma}\label{DtnD2n-diff}
 There exists a constant $C>0$ such that 
 \[
 E\left[ \left( \D_{\tau_n^\B} - \D_{2n} \right)^2 \right] \leq C \sqrt{n}. 
 \]
\end{lemma}
\begin{proof}
We will show that 
\[
 E\left[ \left( \frac{\D_{\tau_n^\B} - \D_{2n}}{n^{1/4}} \right)^2 \right]
 = \int_0^\infty 2y P\left( \left| \D_{\tau_n^\B} - \D_{2n} \right| \geq y n^{1/4} \right)dy
\]
is bounded uniformly in $n$.  By \eqref{lip} and Lemma \ref{Lemma1} we
get
\[
\int_{n^{3/8}}^\infty 2y P\left( \left| \D_{\tau_n^\B} - \D_{2n} \right| \geq y n^{1/4} \right)  dy
\leq \int_{n^{3/8}}^\infty 2y P\left( \left| \D_{\tau_n^\B} \right| \geq y n^{1/4} \right)  dy
\leq C \sqrt{n}\, e^{-c n^{1/4}}. 
\]
Thus, it remains only to get good bounds on $P\left( \left| \D_{\tau_n^\B} - \D_{2n} \right| \geq y n^{1/4} \right)$ for $y \leq n^{3/8}$. 
To this end, note that 
\begin{align*}
&P\left( \left| \D_{\tau_n^\B} - \D_{2n} \right| \geq y n^{1/4} \right) \\
&\qquad \leq P\left( | \tau_n^\B - 2n | \geq y \sqrt{n} \right) 
+ P\left( |\D_i - \D_{\tau_n^\B}| \geq yn^{1/4}, \text{ for some } |i-\tau_n^\B| \leq y\sqrt{n} \right) \\
&\qquad \leq P\left( | \D_{\tau_n^\B} | \geq y \sqrt{n} \right) 
+ P\left( \sup_{\tau_{n-\fl{y\sqrt{n}}}^\B \leq i \leq \tau_{n+\fl{y\sqrt{n}}}^\B } | \D_i - \D_{\tau_{n-\fl{y\sqrt{n}}}^\B} | \geq \frac{1}{2}\,y n^{1/4} \right) \\
&\qquad \leq Ce^{-c y^2} + C e^{-cy}, 
\end{align*}
where the second inequality follows from the fact that $\D_{\tau_n^\B} = \tau_n^\B - 2n$ and the last inequality follows from Lemmas \ref{Lemma1} and \ref{elenatype} for $n$ sufficiently large and $y\leq n^{3/8}$. 
\end{proof}

Finally, we obtain asymptotics for $\Var(\D_{\tau_n^{\B}})$ and a
  bound on $\Var(\D_{\tau_m^\B} - \D_{\tau_n^\B} )-2(m-n)$.
\begin{proposition}\label{VarDtmDtn}
  $\displaystyle\lim\limits_{n\to\infty}
    \frac{\Var(\D_{\tau_n^{\B}})}{2n} = \frac{1}{2\alpha+1}$. Moreover,
    there is a constant $C>0$ such that for $n$ sufficiently large and
    $m \in [n,2n]$,
\[
\left|\Var(\D_{\tau_m^{\B}} - \D_{\tau_n^{\B}}) - 2(m-n) \right|
\leq C n^{3/4} + C\frac{(m-n)^2}{n}.
\]
\end{proposition}

\begin{proof}
Using the decomposition 
\[
 \D_{\tau_m^{\B}}-\D_{\tau_n^{\B}} = (\D_{2m}-\D_{2n}) + (\D_{\tau_m^{\B}} - \D_{2m}) - (\D_{\tau_n^{\B}} -\D_{2n} ),  
\]
we can expand the variance into the variance and covariance terms on the right. From this one obtains that for $m\in [n,2n]$ and $n$ sufficiently large
\begin{align*}
& \left| \Var(\D_{\tau_m^{\B}} - \D_{\tau_n^{\B}}) - 2(m-n) \right| \\
&\quad  \leq \left| \Var(\D_{2m}-\D_{2n}) - 2(m-n) \right| + \Var(\D_{\tau_m^{\B}} - \D_{2m}) + \Var(\D_{\tau_n^{\B}} -\D_{2n} ) \\
&\quad\qquad + 2 \sqrt{\Var( \D_{2m}-\D_{2n})} \left\{ \sqrt{\Var(\D_{\tau_m^{\B}} - \D_{2m})} + \sqrt{\Var(\D_{\tau_n^{\B}} - \D_{2n})}  \right\} \\
&\quad\qquad + 2 \sqrt{\Var(\D_{\tau_m^{\B}} - \D_{2m})}\sqrt{\Var(\D_{\tau_n^{\B}} - \D_{2n})} \\
&\quad \leq C n^{3/4} + C\frac{(m-n)^2}{n}, 
\end{align*}
where the last inequality follows from the variance estimates in Corollary \ref{cor-VDn} and Lemma \ref{DtnD2n-diff}. 
The proof of the limit of $\frac{\Var(\D_{\tau_n^{\B}})}{n}$ follows by a similar argument and using Corollary \ref{cor-VDn} and Lemma \ref{DtnD2n-diff}. 
\end{proof}

\subsubsection{The drift for the urn process}

Our main goal in this subsection is to prove that
$E[\D_{\tau_n^\B}] \to \frac{1}{2(2\alpha+1)}$ (Proposition
\ref{proplimit}). We will need the following lemma which slightly improves on the bound
$E[\D_{\tau_n^\B}] = O(\sqrt{n})$ that follows from Lemma \ref{Lemma1}.

\begin{lemma}\label{EDtn-size}
 $\displaystyle\lim\limits_{n\to\infty} \frac{1}{\sqrt{n}} E[\D_{\tau_n^\B}] = 0$.
\end{lemma} 
\begin{proof}
  First of
  all, since Lemma \ref{DtnD2n-diff} implies that
  $E\left[ \left| \D_{\tau_n^\B}-\D_{2n} \right| \right] \leq C
  n^{1/4}$, it is enough to prove that
  $\lim\limits_{n\to\infty} \frac{1}{\sqrt{n}} E[\D_n] = 0$. To this
  end, for any fixed $\delta>0$ we have from \eqref{stn} and
  Lemma \ref{Lemma1} that
 \begin{equation} \label{EDn-dec}
 \left| E[\D_n] \right| 
 \leq \left| E\left[\D_n - \D_{\fl{\delta n}} \left( \fl{\delta n}/n \right)^\alpha \right] \right| + C \delta^{\alpha+\frac{1}{2}} \sqrt{n}, 
 \end{equation}
 For the first term on the right side, we use the decomposition in
 \eqref{mgd}, the fact that $E[\Delta_i] = 0$ for all $i$, and the
 second moment bounds that were used in the proofs of the variance
 bounds in \eqref{vart2}--\eqref{vart3} to obtain that
\begin{align}\label{EDn-dec2}
&\left| E\left[\D_n - \D_{\fl{\delta n}} \left( \fl{\delta n}/n \right)^\alpha \right] \right|  \\
& \leq \left| E\left[ \sum_{i=\fl{\delta n}}^{n-1} \D_i \left( \left( \frac{i+1}{n} \right)^\alpha \left(1-\frac{\alpha}{i}\right) - \left(\frac{i}{n}\right)^{\alpha} \right) \right]    \right| 
+ \left| E\left[ \sum_{i=\fl{\delta n}}^{n-1} \epsilon_i^n \left(\frac{i+1}{n} \right)^\alpha \right]  \right| \nonumber \\
& \leq C_\delta (\log n)/\sqrt{n} + C_\delta \log^2 n. \nonumber
\end{align}
Combining \eqref{EDn-dec} and \eqref{EDn-dec2} we obtain that 
$\limsup\limits_{n\to\infty} \frac{1}{\sqrt{n}} \left| E[\D_n] \right| \leq C \delta^{\alpha+\frac{1}{2}}$, and taking $\delta \to 0$ we get that $\lim\limits_{n\to\infty} \frac{1}{\sqrt{n}} E[\D_n] = 0$, which completes the proof of the lemma. 
\end{proof}

\begin{proposition} \label{proplimit}
$\displaystyle\lim\limits_{n\to\infty}E[\D_{\tau^{\B}_n}]=\frac{1}{2(2 \alpha +1)}$.
\end{proposition}

\begin{proof}
Using the first equality in \eqref{VUdiff} and the assumption that $w(i) = (i+1)^{-\alpha}$, we get
\begin{align*}
 \sum_{j=0}^{n-1} \left\{ (2j+2)^ \alpha-(2j+1)^\alpha \right\} 
&= E\left[ \sum_{i=0}^{\D_{\tau_n^{\B}} + n-1} (2i+1)^\alpha - \sum_{i=0}^{n-1} (2i+1)^\alpha \right] \\
&= E\left[ \D_{\tau_n^\mathfrak{B}} \right] (2n)^\alpha  + E\left[ \sum_{i=n}^{n + \D_{\tau_n^\mathfrak{B}} - 1} \left\{ (2i+1)^\alpha - (2n)^\alpha \right\}  \right], 
\end{align*}
where in the last equality and throughout the remainder of the proof we use the convention that 
$\sum_{i=n}^{n + m - 1} (\cdot) = 0$ if $m=0$ 
and $\sum_{i=n}^{n + m - 1} (\cdot) = -\sum_{i=n + m}^{n - 1} (\cdot)  $ if $m<0$. 
By integral approximations, the sum on the left above is easily seen to equal $\frac{1}{2}(2n)^\alpha + o(n^{\alpha})$, 
and so to
finish the proof, we need to obtain good asymptotics for the last expectation on the right. 
We will analyze the sum inside this expectation differently depending on whether or not $|\D_{\tau_n^{\B}}| \leq n^{3/5}$. Since the sum inside the expectation is always bounded by $C(n+|\D_{\tau_n^{\B}}|)^{\alpha+1}$ for a fixed constant $C>0$, it follows that
\[
 E\left[ \sum_{i=n}^{n + \D_{\tau_n^\mathfrak{B}} - 1} \left\{ (2i+1)^\alpha - (2n)^\alpha \right\} \ind{|\D_{\tau_n^{\B}}| > n^{3/5}} \right] 
\leq C E\left[ (n+|\D_{\tau_n^{\B}}|)^{\alpha+1} \ind{|\D_{\tau_n^{\B}}| > n^{3/5}} \right],
\]
and this upper bound vanishes as $n\to\infty$ by Lemma \ref{Lemma1}.
To control the sum inside the expectation when $|\D_{\tau_n^{\B}}| \leq n^{3/5}$, 
first note that 
by a Taylor series approximation there is a constant $C>0$ such that
\[
 \max_{i:|i-n|\leq n^{3/5}} \left| (2i+1)^\alpha - (2n)^\alpha  - (2n)^\alpha  \alpha \frac{i-n}{n} \right| \leq C n^{\alpha - \frac{4}{5}}. 
\]
Using this and the fact that $\sum_{i=n}^{n+m-1} (i-n) = \frac{|m||m-1|}{2}$ we have that on the event $\{|\D_{\tau_n^{\B}}| \leq n^{3/5} \}$,
\[
 \left| \sum_{i=n}^{n + \D_{\tau_n^\mathfrak{B}} - 1} \left\{ (2i+1)^\alpha - (2n)^\alpha \right\} -  (2n)^\alpha \alpha  \frac{|\D_{\tau_n^{\B}}||\D_{\tau_n^{\B}}-1|}{2n} \right| 
\leq C n^{\alpha-\frac{4}{5}} | \D_{\tau_n^\mathfrak{B}} | \leq C n^{\alpha-\frac{1}{5}}.
\]
Note that it follows from Lemma \ref{Lemma1}, Proposition \ref{VarDtmDtn} and Lemma \ref{EDtn-size} that 
\[
 \lim_{n\to\infty} E\left[ \frac{|\D_{\tau_n^{\B}}||\D_{\tau_n^{\B}}-1|}{2n} \ind{ |\D_{\tau_n^{\B}}| \leq n^{3/5} } \right] 
= \lim_{n\to\infty} E\left[ \left( \frac{\D_{\tau_n^{\B}}}{\sqrt{2n}} \right)^2 \right] 
= \frac{1}{2\alpha+1}.
\]
Combining the analysis of both the case $| \D_{\tau_n^\mathfrak{B}} | > n^{3/5}$ and $| \D_{\tau_n^\mathfrak{B}} |\leq n^{3/5}$, we obtain that 
\[
  E\left[ \sum_{i=n}^{n + \D_{\tau_n^\mathfrak{B}} - 1} \left\{ (2i+1)^\alpha - (2n)^\alpha \right\}  \right]
= (2n)^\alpha \frac{\alpha}{2\alpha+1} + o(n^{\alpha}). 
\]
Thus far we have shown that 
\[
 \frac{1}{2}(2n)^\alpha + o(n^{\alpha}) = E\left[ \D_{\tau_n^\mathfrak{B}} \right] (2n)^\alpha + (2n)^\alpha \frac{\alpha}{2\alpha+1} + o(n^{\alpha}). 
\]
Dividing both sides by $(2n)^\alpha$ and then taking $n\to \infty$ finishes the proof. 
\end{proof}

\subsubsection{Proof of Proposition \ref{lemC1}}
As noted in the remark following Proposition \ref{lemC1}, we need only prove the claim in \eqref{EM}. 

For purely notational reasons, we give a proof only for $c =0$. We argue
by contradiction and suppose that there exists a $\eta_0 > 0 $ for
which the conclusion fails.  This implies the existence of a sequence
of integers $M_r\to\infty$ as $r\to\infty$
such that for each $r$ there exists arbitrarily large $N$ such that
$\text{dist}\,(P^{Z^{M_r,0}_N},P^{Z^{(0,0)}}) \geq \eta_0$.  We note
that the space $(D([0,1]),d^\circ)$ (see \cite[(12.16)]{bCOPM})
is a complete separable metric space.  So by \cite[p.\,72]{bCOPM}, the
Prokhorov metric on the set of probability measures on $D([0,1])$
gives the topology of convergence in distribution.  Thus for any
choice of $N_r,\ r\geq 1$, such that
\begin{equation}\label{distlb}
\text{dist}\,(P^{Z^{M_r,0}_{N_r}},P^{Z^{(0,0)}}) \geq \eta_0,
\end{equation}
the processes $\{Z^{M_r,0}_{N_r}\}_{r \geq 1}$ cannot converge to $Z^{(0,0)}$ in law.
The desired contradiction will be arrived at by showing that 
we can find a sequence $N_r\to \infty$ satisfying \eqref{distlb} but for which 
 $\{Z^{M_r,0}_{N_r}\}_{r \geq 1}$ does converge to $Z^{(0,0)}$ in law.

To choose the sequence $N_r$, we first note that the first conclusion in Proposition \ref{lemC1} implies that 
\begin{align*}
 &\lim_{N\to\infty} P\left( \max_{j \leq N} \left| \mathcal{E}^{\mathcal{T}_{M_r N}}(j) - M_r N \right| \geq \frac{M_r N}{2} \right) \\
 &\qquad = P\left( \sup_{s\leq 1} \left| Z^{(\alpha,1)}(s) - M_r \right| \geq \frac{M_r}{2} \, \Bigl| \, Z^{(\alpha,1)}(0) = M_r \right) \\
 &\qquad = P\left( \sup_{s\leq \frac{1}{M_r}} \left| Z^{(\alpha,1)}(s) - 1 \right| \geq \frac{1}{2} \, \Bigl| \, Z^{(\alpha,1)}(0) = 1 \right) , 
\end{align*}
where the last equality follows from standard rescaling of Bessel squared processes. Since clearly the right side vanishes as $M_r \to \infty$ we can choose a sequence of integers $N_r \geq M_r^4$ such that 
\begin{equation}\label{Trcond}
 \lim_{r\to\infty} P\left( \max_{j \leq N_r} \left| \mathcal{E}^{\mathcal{T}_{M_r N_r}}(j) - M_r N_r \right| \geq \frac{M_r N_r}{2} \right) = 0. 
\end{equation}

To prove that $Z_{N_r}^{M_r,0}$ converges to $Z^{(0,0)}$ in law, we
will apply \cite[Theorem 4.1, p.\,354]{ekMP}.  As noted in the
paragraph preceding \eqref{EdiffDtn}, the process
$\mathcal{E}^{\mathcal{T}_\ell}$ is a time inhomogeneous Markov chain,
but if we define
$\tilde{\mathcal{E}}^{\mathcal{T}_\ell}(j) =
\mathcal{E}^{\mathcal{T}_\ell}(j) - \ind{0}(j)$ then
$\tilde{\mathcal{E}}^{\mathcal{T}_\ell}$ is a time homogeneous Markov
chain and so we will work with this process instead.  We first
consider (and then modify) the two dimensional process
$(Y^r_1(s), Y^r_2(s))$ defined by
\[
Y^r_1(s) = \frac{\tilde{\mathcal E}^{\mathcal{T}_{(M_r+1)N_r}}(\fl{sN_r})  - \tilde{\mathcal E}^{\mathcal{T}_{M_rN_r}}(\fl{sN_r})   }{N_r}, 
\quad\text{and}\quad
Y^r_2(s) = \frac{ \tilde{\mathcal E}^{\mathcal{T}_{M_rN_r}}(\fl{sN_r})   }{M_rN_r}.
\]
We are really only interested in proving that $Y_1^r$ converges in
distribution to $Z^{(0,0)}$; the second coordinate $Y_2^r$ is included
for convenience of the proof because $Y_1^r$ is not a Markov
  chain but the joint process $(Y_1^r,Y_2^r)$ is. Note that to
    make the proof simpler we (intentionally) overscaled $Y_2^r$  so
that it will converge to a constant.  Define
\begin{equation}
 T_r := \inf \left\{ j \in \N: \left| \tilde{\mathcal E}^{\mathcal{T}_{M_r N_r}}(j) - M_r N_r \right| \geq \frac{M_r N_r}{2} \right\}, 
\end{equation}
so that \eqref{Trcond} becomes
$\lim_{r\to\infty} P(T_r \leq N_r) = 0$.  We then take
$Z^r_i (s):=Y^r_i (s\wedge (T_r/N_r) ) $ so that
  $Z^r_i (s) = Y^r_i (s)$, $\forall s \leq 1$, $i\in\{1,2\}$, with
  probability tending to $1$ as $r\to \infty$.  As in \cite[Theorem
4.1]{ekMP} we choose processes $(B^r_i(s))_{s\in[0,1]} $,
  $i\in\{1,2\}$, so that $ Z^r_i(s) - B^r_i(s) $, $s\in[0,1]$, is a
martingale with respect to
$\mathcal{F}^{Z^r}_s = \sigma( Z^r_1(u), Z^r_2(u),\ u \leq s)$.
Obviously, the processes $Z^r_i $, $i=1,2$, are constant on
  $[(j-1)/N_r,j/N_r)$,
so we can write
\[
 B^r_i(s) = \sum_{j=0}^{\fl{sN_r }-1} b^r_i(j), 
\quad \text{where}\quad b^r_i(j )= E\left[ Z_{i}^r\left(\tfrac{j+1}{N_r} \right) -Z_{i}^r \left(\tfrac{j}{N_r}\right) \mid \mathcal{F}^{Z^r}_{j/N_r} \right]. 
\]
Similarly, we introduce $(A^r_{ik}(s))_{s\in[0,1]}$, $i,k \in\{1,2\}$, so
that $(Z^r_i(s)- B^r_i(s))(Z^r_k(s)- B^r_k(s)) -A^r_{ik}(s) $,
$s\in[0,1]$, is a martingale for all $i,k\in\{1,2\}$.
Again we can write
\[
 A^r_{ik}(s) = \sum_{j=0}^{\fl{sN_r }-1} a^r_{ik}(j),
\quad \text{where}\quad
a^r_{ik}(j) = \Cov\left(  Z^r_i\left(\tfrac{j+1}{N_r}\right), Z^r_k\left(\tfrac{j+1}{N_r}\right) \, \Bigl| \,  \mathcal{F}^{Z^r}_{j/N_r}  \right).
\]

Finishing the proof of Proposition \ref{lemC1} amounts to verifying the conditions of \cite[Theorem 4.1]{ekMP} with processes
$b_1,b_2, a_{12}, a_{22} \equiv 0,$ and $ a_{11}(x) = 2x$. 
That is, letting 
\begin{equation}\label{srldef}
 \sigma^r_\ell = \inf\{ s\geq 0:  Z_1^r(s) \geq \ell \},\qquad r \in \N,\ \ell > 1,  
\end{equation}
it is sufficient to show that for any fixed $\ell > 1$ we have 
\begin{alignat}{2}\label{EK4-3}
 &\lim_{r\to\infty} E\left[ \sup_{s\leq \sigma^r_\ell \wedge 1} \left| Z_i^r(s) - Z_i^r(s-) \right|^2 \right] = 0, \quad &&\text{for }  i \in \{1,2\}, 
\\\label{EK4-4}
 &\lim_{r\to\infty} E\left[ \sup_{s\leq \sigma^r_\ell \wedge 1} \left| B_i^r(s) - B_i^r(s-) \right|^2 \right] = 0, \quad &&\text{for } i \in \{1,2\},  
\\\label{EK4-5}
 &\lim_{r\to\infty} E\left[ \sup_{s\leq \sigma^r_\ell \wedge 1} \left| A_{ik}^r(s) - A_{ik}^r(s-) \right| \right] = 0, \quad &&\text{for } i,k \in \{1,2\},  
\\\label{EK4-6}
 &\lim_{r\to\infty} \sup_{s\leq \sigma^r_\ell \wedge 1} \left| B^r_i(s) \right| = 
 0, \quad \text{in probability}, &&\text{for } i \in \{1,2\},
\\\label{EK4-7a}
 &\lim_{r\to\infty} \sup_{s \leq \sigma^r_\ell \wedge 1} \left| A^r_{i2}(s) \right| = 0, \quad \text{in probability,}\quad &&\text{for }i\in\{1,2\}, \quad\text{and}
\\
\label{EK4-7b} 
  &\lim_{r\to\infty} \sup_{s\leq \sigma^r_\ell \wedge 1}  \left| A^r_{11}(s) - 2 \int_0^s Z^r_1(s) \, ds \right| = 0, \quad \text{in probability}.&&\  
\end{alignat}

We first address 
\eqref{EK4-4} and \eqref{EK4-6}.
Recalling the definitions above, and using \eqref{EdiffDtn} we have that
\begin{align*}
 b^r_2(j) &= \frac{1}{M_r N_r} E\left[ \tilde{\mathcal E}^{\mathcal{T}_{M_rN_r}}\left( (j+1)\wedge T_r \right) - \tilde{\mathcal E}^{\mathcal{T}_{M_rN_r}}\left( j \wedge T_r \right) \mid  \tilde{\mathcal E}^{\mathcal{T}_{M_rN_r}}(x), \, x\leq j \right] \\
 &= \frac{1}{M_rN_r}\,E\left[\D_{\tau_n^\B}\right]\ind{j<T_r}\qquad \text{where }\ n=\tilde{\mathcal E}^{\mathcal{T}_{M_rN_r}}(j)
  \\
 &\leq \frac{1}{M_r N_r}\max_{n \in [\frac{M_rN_r}{2}, \frac{3M_rN_r}{2} ]} E[\D_{\tau_n^\B}]. 
\end{align*}
It follows from Proposition \ref{proplimit} that for $r$ sufficiently large we have $|b_2^r(j)| \leq (M_r N_r)^{-1}$ for all $j$. 
Therefore, we can conclude that 
\[
\sup_{s\leq 1} (B_2^r(s)-B_2^r(s-))^2 
\ = \ \sup_{j\leq N_r} (b^r_2(j))^2 \leq (N_rM_r)^{-2} \rightarrow 0,
\]
and also 
\[
\sup_{s\leq 1} |B_2^r(s)| \leq  
N_r \sup_{j\leq N_r} |b^r_2(j)| \leq M_r^{-1} \rightarrow 0.
\]

A similar computation yields 
\begin{equation*}
 b^r_1(j) 
= \frac{1}{N_r}E\left[(\D_{\tau_m^{\B}}-\D_{\tau_n^{\B}})\right]\ind{j<T_r},\quad \text{where }\  n=\tilde{\mathcal E}^{\mathcal{T}_{M_rN_r}}(j)\ \text{ and } \  m=\tilde{\mathcal E}^{\mathcal{T}_{(M_r+1)N_r}}(j).
\end{equation*}
Since Proposition \ref{proplimit} implies that
$E\left[ \D_{\tau_m^{\B}} -\D_{\tau_n^{\B}} \right] \to 0$ as both
$m$ and $n$ go to $\infty$ and since 
$\tilde{\mathcal E}^{\mathcal{T}_{(M_r+1)N_r}}(j)\ \geq \tilde{\mathcal E}^{\mathcal{T}_{M_rN_r}}(j)
\geq \frac{M_rN_r}{2}$ on the event $\{j<T_r\}$, it follows that
\[
 \sup_{s\leq 1} (B_1^r(s)-B_1^r(s-))^2 \to 0
\quad\text{and}\quad
\sup_{s\leq 1} |B_1^r(s)| \to 0,  \quad \text{as }r \rightarrow \infty.
\]

We now consider 
\eqref{EK4-5},  \eqref{EK4-7a}, and \eqref{EK4-7b}.
For this, note that 
\begin{align*}
 a_{22}^r(j) 
&=  \frac{1}{(M_r N_r)^2}\Var\left( \tilde{\mathcal E}^{\mathcal{T}_{M_rN_r}}((j+1) \wedge T_r )  \, \Bigl| \,  \tilde{\mathcal E}^{\mathcal{T}_{M_rN_r}}(x), \, x \leq j \right) \\
&= \frac{1}{(M_r N_r)^2} \Var( \D_{\tau_m^{\B}} )\ind{j<T_r}, \quad  \text{where } \ m=\tilde{\mathcal E}^{\mathcal{T}_{M_rN_r}}(j),\\\shortintertext{and similarly}
a_{11}^r(j) 
&= \frac{1}{N_r^2} \Var\left( \D_{\tau_m^{\B}} - \D_{\tau_n^{\B}} \right)\ind{j<T_r}, \quad \text{where } \ n=\tilde{\mathcal E}^{\mathcal{T}_{M_rN_r}}(j)  \text{ and } \  m=\tilde{\mathcal E}^{\mathcal{T}_{(M_r+1)N_r}}(j). 
\end{align*}
We remark that in the analysis of the terms involving $a_{11}^r(j)$ we will use the fact that if $j < \sigma^r_\ell N_r$ then $\tilde{\mathcal E}^{\mathcal{T}_{(M_r+1)N_r}}(j) - \tilde{\mathcal E}^{\mathcal{T}_{M_rN_r}}(j) \leq \ell N_r$. 
Then to prove \eqref{EK4-5}, note that Proposition \ref{VarDtmDtn} implies that for $r$ sufficiently large 
\begin{align*}
 \sup_{s\leq \sigma^r_\ell \wedge 1} |A_{11}^r(s) - A_{11}^r(s-)| &= \max_{j< (\sigma^r_\ell \wedge 1)N_r} |a_{11}^r(j)| \\&\qquad \leq 
\max_{\substack{|n-M_rN_r| \leq \frac{M_rN_r}{2} \\ |m-n| \leq \ell N_r}} \frac{1}{N_r^2}\,\Var(\D_{\tau_m^\B}-\D_{\tau_n^\B})
\leq \frac{C}{N_r}, \\
 \sup_{s\leq 1} |A_{22}^r(s) - A_{22}^r(s-)| &= \max_{j<N_r} |a_{22}^r(j)| 
\leq \max_{|n-M_rN_r|\leq \frac{M_rN_r}{2}} \frac{\Var(\D_{\tau_n^\B})}{(M_rN_r)^2}  
\leq \frac{C}{M_r N_r}, \\
 \text{and }\quad 
 \sup_{s\leq 1} |A_{12}^r(s) - A_{12}^r(s-)| &= \max_{j<N_r} |a_{12}^r(j)| \leq  \max_{j<N_r} \sqrt{a_{11}^r(j)}\sqrt{a_{22}^r(j)} \leq \frac{C}{\sqrt{M_r}N_r}.
\end{align*}
(Note that to obtain the last inequality in the second line above we
are also using that $N_r \geq M_r^4$.)  This is enough to prove
\eqref{EK4-5}.  The last two bounds yield \eqref{EK4-7a}, since
$\sup_{s\leq 1} |A_{ij}^r(s)| \leq N_r \max_{j<N_r} |a_{ij}^r(j)|$.

For \eqref{EK4-7b}, note that Proposition \ref{VarDtmDtn}, together with the definition of the stopping times $T_r$ and $\sigma^r_\ell$, implies that for $r$ sufficiently large  
\begin{multline*}
\sup_{j < T_r \wedge (N_r \sigma^r_\ell)} \left| a_{11}^r(j) - \frac{2}{N_r} Z^r_1(j/N_r) \right|
\leq \frac{C}{N_r^2} \left( \left( \frac{3M_r N_r}{2} \right)^{3/4} + \frac{(2\ell N_r)^2}{\frac{M_r N_r}{2}}  \right)\\
\leq \frac{C M_r^{3/4}}{N_r^{5/4}} + \frac{C}{M_r N_r}. 
\end{multline*}
Thus, on the event $\{T_r \geq N_r\}$, we have that 
\begin{align*}
& \sup_{s\leq \sigma^r_\ell \wedge 1} \left| A_{11}^r(s) - \int_0^s 2 Z_1^r(s) \, ds \right|  \\
 &\leq \sup_{s\leq \sigma^r_\ell \wedge 1} \left\{ \sum_{j=0}^{\fl{N_r s}-1} \left| a_{11}^r(j) - \frac{2}{N_r} Z_1^r(j/N_r) \right| + \frac{Z_1^r(\fl{N_r s}/N_r)}{N_r} \right\} 
 \leq \frac{C M_r^{3/4}}{N_r^{1/4}} + \frac{C}{M_r} + \frac{\ell}{N_r}. 
\end{align*}
Since $M_r \to \infty$ as $r\to\infty$ and $N_r \geq M_r^4$, this bound vanishes as $r\to\infty$, and since $P(T_r \leq N_r) \to 0$ as $r\to\infty$, this completes the proof of \eqref{EK4-7a}. 

Finally, we will prove \eqref{EK4-3} for the case $i=1$ as a similar (but simpler) proof works for $i=2$ as well. First of all, note that
\begin{align*}
& E\left[ \sup_{s\leq \sigma^r_\ell \wedge 1} \left| Z_1^r(s) - Z_i^r(s-) \right|^2 \right] = E\left[ \max_{j< \sigma^r_\ell N_r \wedge N_r \wedge T_r } \left| Z_1^r(\tfrac{j+1}{N_r}) - Z_i^r(\tfrac{j}{N_r}) \right|^2  \right] \\
&\quad \leq \frac{M_r^2}{N_r} + \int_{M_r/\sqrt{N_r}}^\infty 2x P\Biggl( \max_{j<\sigma^r_\ell N_r \wedge N_r \wedge T_r} \left| Z_1^r(\tfrac{j+1}{N_r}) - Z_i^r(\tfrac{j}{N_r}) \right| > x \Biggr) \, dx. \\
&\quad \leq \frac{M_r^2}{N_r} + \sum_{j=0}^{N_r-1} \int_{M_r/\sqrt{N_r}}^\infty 2x P\Biggl( \left| Z_1^r(\tfrac{j+1}{N_r}) - Z_i^r(\tfrac{j}{N_r}) \right| > x, \, j < (\sigma_\ell^r N_r \wedge T_r) \Biggr) \, dx. 
\end{align*}
To bound the probabilities in the integral on the right, first recall that if $\tilde{\mathcal E}^{\mathcal{T}_{M_rN_r}}(j)=n$ and $\tilde{\mathcal E}^{\mathcal{T}_{(M_r+1)N_r}}(j)=m$ then $Z_1^r(\tfrac{j+1}{N_r}) - Z_i^r(\tfrac{j}{N_r})$ has the same distribution as $\frac{\D_{\tau_m^\B}-\D_{\tau_n^\B}}{N_r}$, and moreover if $j < ( \sigma_\ell^r N_r \wedge T_r)$ then these values of $n$ and $m$ must be such that $|n-M_rN_r|\leq \frac{M_r N_r}{2}$ and $n\leq m \leq n+ \ell N_r$. 
Therefore, 
conditioning on  $\tilde{\mathcal E}^{\mathcal{T}_{M_rN_r}}(j)$ and $\tilde{\mathcal E}^{\mathcal{T}_{(M_r+1)N_r}}(j)$ gives 
\begin{align*}
&P\Biggl( \left| Z_1^r(\tfrac{j+1}{N_r}) - Z_i^r(\tfrac{j}{N_r}) \right| > x, \, j < (\sigma_\ell^r N_r \wedge T_r) \Biggr) \\
&= \sum_{m,n} P\left( | \D_{\tau_m^\B} - \D_{\tau_n^\B} | > x N_r \right) 
 P\left( \tilde{\mathcal E}^{\mathcal{T}_{M_rN_r}}(j) = n, \,  \tilde{\mathcal E}^{\mathcal{T}_{(M_r+1)N_r}}(j) = m, \, j < (\sigma_\ell^r N_r \wedge T_r) \right) \\
&\leq 2 \sup_{\frac{M_r N_r}{2} \leq  n \leq  2 M_r N_r } P\left( | \D_{\tau_n^\B} | \geq \frac{x N_r}{2} \right)\leq C \exp\left\{ -c \left( \tfrac{x^2 N_r}{x\vee M_r} \right) \right\},
\end{align*}
where the last inequality holds by Lemma \ref{Lemma1}.
Thus, we can conclude that 
\[
 E\left[ \sup_{s\leq 1} \left| Z_1^r(s) - Z_i^r(s-) \right|^2 \right] 
  \leq \frac{M_r^2}{N_r} + N_r \int_{M_r/\sqrt{N_r}}^\infty 2x C \exp\left\{ -c \left( \tfrac{x^2 N_r}{x\vee M_r} \right) \right\} \, dx. 
\]
Since $M_r\to \infty$ as $r\to\infty$ and $N_r \geq M_r^4$, this upper bound vanishes as $r\to\infty$, 
thus completing the proof of \eqref{EK4-3} for $i=1$. 
 
\qed

\appendix

\section{Branching-like processes and generalized Ray-Knight theorems}\label{app:blp}

T\'oth's analysis of the type of SIRWs considered in this paper has been done primarily through the study of the directed edge local times of the walk stopped at certain stopping times. 
In this appendix we will first recall the definition of these processes and their connection with the directed edge local times of the random walk and then prove a few results which we shall need in Appendix~\ref{app:afc}.

We begin by defining  two homogeneous Markov chains on $\N_0$. 
The transition probabilities will be given in terms of the generalized P\'olya urn models from section \ref{sec:Polya}. 
\begin{itemize}
 \item[{\tiny $\bullet$}] $\zeta: = (\zeta_k)_{k\geq 0}$ is a Markov chain with transition probabilities given by 
 \begin{equation}\label{bblp}
  P(\zeta_{k+1} = j \mid \zeta_k = i) = P\left( \mathfrak{R}^{-}_{\tau_{i+1}^{\mathfrak{B}^-}} = j \right), \qquad \forall i,j\geq 0. 
 \end{equation}
 \item[{\tiny $\bullet$}] $\tilde\zeta: = (\tilde\zeta_k)_{k\geq 0}$ is a Markov chain with transition probabilities given by 
 \begin{equation}\label{fblp}
  P(\tilde\zeta_{k+1} = j \mid \tilde\zeta_k = i) = P\left( \mathfrak{R}^{+}_{\tau_{i}^{\mathfrak{B}^+}} = j \right), \qquad \forall i,j\geq 0. 
 \end{equation}
 \end{itemize}
 
\begin{remark}
We will refer to $\zeta$ and $\tilde\zeta$ as BLPs due to the fact that in the case where the weight function $w(\cdot) \equiv 1$ we have that $\tilde\zeta$ is a Galton-Watson branching process with Geo(1/2) offspring distribution and $\zeta$ is a Galton-Watson branching process with one immigrant before reproduction and Geo(1/2) offspring distribution. 
\end{remark} 

\begin{remark}
 Note that in \cite{tGRK} these Markov chains are defined slightly differently so that 
 \begin{align*}
    &P(\zeta_{k+1} = j \mid \zeta_k = i) = P\left( \mathfrak{B}^{+}_{\tau_{i+1}^{\mathfrak{R}^+}} = j \right), \qquad \forall i,j\geq 0. 
\\\shortintertext{
and }
  &P(\tilde\zeta_{k+1} = j \mid \tilde\zeta_k = i) = P\left( \mathfrak{B}^{-}_{\tau_{i}^{\mathfrak{R}^-}} = j \right), \qquad \forall i,j\geq 0. 
 \end{align*}

 However, it is easy to see that the difference in these definitions is simply an interchange of the labels ``red'' and ``blue'' and changing the corresponding parameters of the urn model accordingly. Thus, the Markov chains $\zeta$ and $\tilde\zeta$ defined as in \eqref{bblp} and \eqref{fblp} are equivalent to those in \cite{tGRK}. 
\end{remark}

The BLPs $\zeta$ and $\tilde\zeta$ are related to the $\mathcal{E}$ and $\mathcal{D}$ processes of local times of directed edges as defined in \eqref{upcr} when the random walk is stopped at certain special stopping times. There are various choices of the stopping times that can be used, but we will discuss this connection here only for the stopping times that will be needed for our purposes. 

For any $z \in \Z$ and $m\in \N$, let $\tau_{z,m} = \min\{n\geq 0: \, \mathcal{L}(z,n) = m \}$ be the $m$-th time the random walk reaches $z$. 
Following similar reasoning as in the paragraph above \eqref{EdiffDtn}, one can see that the process $(\mathcal{E}^{\tau_{z,m}}(x) )_{x\geq z}$ is a time inhomogeneous Markov chain. Indeed, using the fact that 
\[
\mathcal{D}^{\tau_{z,m}}(x) = \mathcal{E}^{\tau_{z,m}}(x-1)+1 \text{ if } z<x\leq 0, \quad \text{and}\quad \mathcal{D}^{\tau_{z,m}}(x) = \mathcal{E}^{\tau_{z,m}}(x-1) \text{ if } x>z \vee 0, 
\]
and the fact that the sequence of left/right steps from $x$ can be generated by the P\'olya urn process 
$(\mathfrak{B}_n^-, \mathfrak{R}_n^-)$ if $x<0$ or $(\mathfrak{B}_n^+, \mathfrak{R}_n^+)$ if $x>0$, it follows that 
\begin{align}\label{zetaE}
P\left( \mathcal{E}^{\tau_{z,m}}(x) = j \mid \mathcal{E}^{\tau_{z,m}}(x-1)=i \right) &= 
P\left( \zeta_1 = j \mid \zeta_0 = i \right)
,\ \  \text{if } z<x<0, \ \ \text{ and}
\\
\label{tzetaE}
P\left( \mathcal{E}^{\tau_{z,m}}(x) = j \mid \mathcal{E}^{\tau_{z,m}}(x-1)=i \right)
  &= P\left( \tilde\zeta_1 = j \mid \tilde\zeta_0 = i \right), \ \  \text{if } x > z\vee 0. 
\end{align}
Thus, $(\mathcal{E}^{\tau_{z,m}}(x))_{z\leq x<0}$ has the distribution of the BLP $\zeta$ with a random initial condition given by the distribution of $\mathcal{E}^{\tau_{z,m}}(z)$, 
and $(\mathcal{E}^{\tau_{z,m}}(x))_{x\geq z\vee 0}$ has the distribution of the BLP $\tilde\zeta$ with a random initial condition given by the distribution of $\mathcal{E}^{\tau_{z,m}}(z\vee 0)$.

The results in this section will only be needed for the proof of our
results in the asymptotically free case ($\alpha = 0$), and so we will
restrict our discussion to this case.  The following result due to
T\'oth \cite{tGRK} shows that when $\alpha=0$ the BLPs $\zeta$ and
$\tilde\zeta$ have scaling limits which are multiples of squared
Bessel processes of dimension $2-2\gamma$ and $2\gamma$, respectively.

\begin{proposition}[\text{\cite{tGRK}}]\label{prop:grkaf}
 Assume that $w$ is as in \eqref{w} with $\alpha = 0$. 
 \begin{enumerate}[(1)]
  \item For $n\geq 1$ let $\zeta^{(n)} = (\zeta^{(n)}_k)_{k\geq 0}$ have the distribution of the BLP $\zeta$ with initial condition $\zeta^{(n)}_0 = \fl{yn}$ for some $y \geq 0$, and let $\mathcal{Z}_n(t) = \frac{\zeta^{(n)}_{\fl{nt}} }{n}$ for $n\geq 1$ and $t\geq 0$. 
  Then, on the space $D([0,\infty))$ we have
 \begin{equation}\label{bblp-da}
  \mathcal{Z}_n(\cdot) \Longrightarrow Z^{(2-2\gamma)}(\cdot). 
 \end{equation}
 \item For $n\geq 1$ let $\tilde\zeta^{(n)} = (\tilde\zeta^{(n)}_k)_{k\geq 0}$ have the distribution of the BLP $\tilde\zeta$ with initial condition $\tilde\zeta^{(n)}_0 =  \fl{yn}$ for some $y > 0$, and let $\mathcal{\tilde{Z}}_n(t) = \frac{\tilde\zeta^{(n)}_{\fl{nt}} }{n}$ for $n\geq 1$ and $t\geq 0$. 
  Then, on the space $D([0,\infty)) \times [0,\infty)$ we have 
  \begin{equation}\label{fblp-da}
   \left( \mathcal{\tilde{Z}}_n(\cdot), \sigma_0^{\mathcal{\tilde{Z}}_n} \right)  \Longrightarrow \left( Z^{(2\gamma)}(\cdot \wedge \sigma_0^{Z^{(2\gamma)}}), \sigma_0^{Z^{(2\gamma)}} \right). 
  \end{equation}
 \end{enumerate}
\end{proposition}

\begin{remark}
  While in \cite{tGRK} it was assumed that the weight function $w$ was
  as in \eqref{w} with $p=\kappa = 1$, a careful reading of the proofs
  of the diffusion approximations in \cite[Section 5]{tGRK} shows that
  the argument goes through for $p \in (0,1]$ and $\kappa > 0$ with
  very few modifications.  In particular, the only thing that changes
  in \cite[Lemma 2A]{tGRK} is the error terms. For instance, if
  following T\'oth's notation, we let
  $F(x) = E[V_1(\zeta_1)\,|\, V_1(\zeta_0) = x]-x$, then the
  asymptotics $F(x) = 1-\gamma + O(x^{-1})$ in \cite[Lemma 2A]{tGRK}
  can be replaced (using the same proof) with
  $F(x) = 1-\gamma + O(x^{-(p\wedge\kappa)} )$.  Similar modifications
  can be made to the other error terms in \cite[Lemma 2A]{tGRK}, and
  from this point on the proof of the diffusion approximations for the
  BLPs goes through without any changes.
\end{remark}

In addition to the diffusion approximation for the BLPs, we will also need a few results which give information on the distributions of hitting probabilities of the BLP $\zeta$.

\begin{lemma}\label{lem:sztail}
Let $w$ be as in \eqref{w} with $\alpha=0$. 
If $\gamma \geq 0$, then for any $\epsilon>0$ there exists a constant $C_\epsilon>0$ such that  
$P^\zeta_0( \sigma_0^\zeta > n ) \geq C_\epsilon n^{-\gamma-\epsilon}$
for all $n$ large enough. 
\end{lemma}

\begin{remark}
  We suspect that the true tail asymptotics are of the form
  $P_0^\zeta( \sigma_0^\zeta > n) \sim C n^{-\gamma}$ when $\gamma>0$
  and $P_0^\zeta( \sigma_0^\zeta > n) \sim C/\log n$ when $\gamma=0$,
but proving such precise asymptotics would require significant extra
work and the asymptotics in Lemma \ref{lem:sztail} are sufficient for
our purposes in the remainder of the paper.
\end{remark}

\begin{proof}[Proof of Lemma~\ref{lem:sztail}]
 First of all, note monotonicity of the Markov process $\zeta$ with respect to its initial condition and the Strong Markov property together imply that
\[
 P^\zeta_0( \sigma_0^\zeta > n ) \geq
 P^\zeta_0( \sigma_0^\zeta > \tau_n^\zeta) P^\zeta_n( \sigma_0^\zeta > n  )
\]
For the second probability on the right, the diffusion approximation implies that  
\[
\liminf_{n\to\infty} P^\zeta_n( \sigma_0^\zeta > n ) \geq \lim_{n\to\infty} P^\zeta_n( \sigma_{n/2}^\zeta > n ) = P( \sigma_{1/2}^{Z^{(2-2\gamma)}} > 1 \mid Z_0^{(2-2\gamma)} = 1 ) > 0,
\]
and thus
it is enough to show that 
$P^\zeta_0( \sigma_0^\zeta > \tau_n^\zeta ) \geq C_\epsilon n^{-\gamma-\epsilon}$. 
To this end, first note that the diffusion approximation implies that for any $x \in (0,1)$, 
\begin{align*}
\liminf_{k\to\infty} P^\zeta_{2^{k-1}} ( \tau^\zeta_{2^k} < \sigma_0^\zeta ) 
&\geq \lim_{k\to\infty} P^\zeta_{2^{k-1}} ( \tau^\zeta_{2^k} < \sigma_{x 2^{k-1}}^\zeta ) \\
&= P(\tau^{Z^{(2-2\gamma)}}_2 < \sigma_x^{Z^{(2-2\gamma)}} \mid Z^{(2-2\gamma)}(0) = 1) =
\begin{cases}
 \frac{1-x^\gamma}{2^\gamma-x^\gamma} & \text{if } \gamma > 0\\[1mm]
 \frac{\log(x)}{\log(x/2)} & \text{if } \gamma = 0. 
\end{cases}
\end{align*}
(Note that the last probability can be computed using martingale
properties of $(Z_t^{(2-2\gamma)})^\gamma$ when $\gamma>0$ and
 of $\log(Z_t^{(2-2\gamma)})$  when
$\gamma=0$.)  Taking $x\to 0$ we can conclude that
$\liminf_{k\to\infty} P^\zeta_{2^{k-1}} ( \tau^\zeta_{2^k} <
\sigma_0^\zeta ) \geq 2^{-\gamma}$, and thus for any $\epsilon>0$
there exists a $k_0 = k_0(\epsilon)$ so that
\[
 P^\zeta_{2^{k-1}}( \tau^\zeta_{2^k} < \sigma_0^\zeta ) \geq 2^{-\gamma-\epsilon}, \quad \forall k\geq k_0. 
\]
Then, for $n\geq 2^{k_0}$ we have 
\begin{align*}
 P^\zeta_0( \sigma_0^\zeta > \tau_n^\zeta )
 &\geq P^\zeta_0( \sigma_0^\zeta > \tau_{2^{\lceil \log_2 n \rceil }}^\zeta ) = P^\zeta_0( \sigma_0^\zeta > \tau_{2^{k_0-1}}^\zeta  )
 \prod_{k=k_0}^{\lceil \log_2 n \rceil} P^\zeta_0( \tau^\zeta_{2^k} < \sigma_0^\zeta \mid  \tau^\zeta_{2^{k-1} } < \sigma_0^\zeta ) \\
 &\geq P^\zeta_0( \sigma_0^\zeta > \tau_{2^{k_0-1}}^\zeta )
 \prod_{k=k_0}^{\lceil \log_2 n \rceil} P^\zeta_{2^{k-1}} ( \tau^\zeta_{2^k} < \sigma_0^\zeta )\\ &\geq P^\zeta_0( \sigma_0^\zeta > \tau_{2^{k_0-1}}^\zeta ) \left( 2^{-\gamma-\epsilon}\right)^{\lceil \log_2 n \rceil} 
 \geq P^\zeta_0( \zeta_1 \geq 2^{k_0-1} ) 2^{-\gamma-\epsilon} n^{-\gamma-\epsilon}.
\end{align*}
\end{proof}
In contrast to the previous lemma, since $Z^{(2-2\gamma)}$ is
transient when $\gamma<0$ it is natural to expect that with positive
probability the process $\zeta$ never goes too far below where it
starts. The following lemma makes precise the sort of statement that
we will need later.

\begin{lemma}\label{lem:noreturn}
 Assume $\gamma < 0$. Then, there exists a constant $c\geq 0$ so that $P_{4n}^\zeta\left( \sigma^\zeta_n = \infty \right) \geq c$ for all $n\geq 1$. 
\end{lemma}

\begin{proof}
If $k = \lceil \log_2 n \rceil + 1$, then  
$P_{4n}^\zeta\left( \sigma^\zeta_n = \infty \right) \geq P^\zeta_{2^{k}} \left( \sigma^\zeta_{2^{k-1}} = \infty \right)$. 
Therefore, it is enough to show that
\begin{equation}\label{2ktrans}
 P^\zeta_{2^{k}} \left( \sigma^\zeta_{2^{k-1}} = \infty \right) \geq c > 0, \quad \forall k\geq 0. 
\end{equation}
For notational convenience we will let $n_k = 2^k$ and $m_k = 2^k - \fl{2^{2k/3}}$ for $k\geq 0$. 
We claim that to prove \eqref{2ktrans} it is enough to show that there exists $q>1/2$ and $k_1\in\N$ such that 
\begin{equation}\label{bias}
 P^\zeta_{m_k} \left( \tau^\zeta_{n_{k+1}} < \sigma^\zeta_{n_{k-1}} \right) \geq q, \quad \forall k\geq k_1, 
\end{equation}
and 
\begin{equation}\label{undershoot}
P^\zeta_{m_k} \left( \zeta_{\sigma^\zeta_{n_{k-1}} \wedge \tau^\zeta_{n_{k+1}} } < m_{k-1} \right) \leq  
C n_k e^{-c n_k^{1/3}} , \quad \forall k\geq k_1.
\end{equation}

To see that \eqref{bias} and \eqref{undershoot} imply \eqref{2ktrans},
first note that it's enough to prove \eqref{2ktrans} only for
$k\geq k_1$ and secondly that we may assume without loss of generality
that $k_1$ is large enough so that $n_{k-1} < m_k$ for all
$k\geq k_1$.  Now consider a Markov process $\{L_i\}_{i\geq 0}$ on
$\{\del\}\cup\{k_1,k_1+1,k_1+2,\ldots\}$ with $\del$ being an
absorbing state and where at states $k\geq k_1$ the process jumps to
the right with probability $q$, to the left (if $k>k_1$) with
probability $1-q-Cn_k e^{-c n_k^{1/3}}$, and otherwise jumps to the
absorbing state $\del$.  Using the strong Markov property, together
with the fact that $\zeta$ is monotone with respect to its starting
point, one can then use \eqref{bias}-\eqref{undershoot} and a coupling
argument to conclude that
$P^\zeta_{2^k}( \sigma_{2^{k-1}}^\zeta = \infty) \geq P^L_k(
\sigma_{k-1}^L \wedge \sigma_\del^L = \infty) \geq P^L_{k_1}(
\sigma_\del^L = \infty)$, where
$\sigma_\del^L = \inf\{i\geq 0: L_i = \del \}$.  Finally, basic simple
random walk computations can show that
$ P^L_{k_1}( \sigma_\del^L = \infty) > 0$.  Thus, it remains only to
show \eqref{bias} and \eqref{undershoot}.

For \eqref{bias}, for any $\epsilon>0$ we have that for $k$ sufficiently large (depending on $\epsilon$) we have $\fl{(2+\epsilon)m_k} \geq n_{k+1}$ and $\fl{(1/2+\epsilon)m_k} \geq n_{k-1}$. 
Therefore, 
\begin{align*}
 \liminf_{k\to\infty}  P^\zeta_{m_k} \left( \tau^\zeta_{n_{k+1}} < \sigma^\zeta_{n_{k-1}} \right)
&\geq \lim_{k\to\infty}  P^\zeta_{m_k} \left( \tau^\zeta_{\fl{(2+\epsilon)m_k}} < \sigma^\zeta_{\fl{(1/2+\epsilon)m_k}} \right) \\
&= \frac{1 - \left(\frac{1}{2}+\epsilon\right)^\gamma}{\left(2+\epsilon\right)^\gamma - \left(\frac{1}{2}+\epsilon\right)^\gamma},
\end{align*}
where the last equality follows from the diffusion approximation in \eqref{bblp-da} and an explicit hitting probability computation for the limiting diffusion using the fact that $(Z_t^{(2-2\gamma)})^\gamma$ is a martingale. 
Taking $\epsilon\to 0$ we get that $\liminf_{k\to\infty}  P^\zeta_{m_k} \left( \tau^\zeta_{n_{k+1}} < \sigma^\zeta_{n_{k-1}} \right)\geq \frac{1-(1/2)^\gamma}{2^\gamma-(1/2)^\gamma} > \frac{1}{2}$, where the last inequality follows from the assumption that $\gamma<0$. This is enough to prove \eqref{bias}. 

For \eqref{undershoot}, first note that 
\begin{align}
& P^\zeta_{m_k} \left( \zeta_{\sigma^\zeta_{n_{k-1}} \wedge \tau^\zeta_{n_{k+1}} } <  m_{k-1} \right) \nonumber \\
&\qquad = \sum_{i\geq 1} 
\sum_{z \in (n_{k-1},n_{k+1})} 
P^\zeta_{m_k} \left( \sigma^\zeta_{n_{k-1}} \wedge \tau^\zeta_{n_{k+1}} = i, \, \zeta_{i-1} = z, \, \zeta_i < m_{k-1} \right) \nonumber \\
&\qquad = \sum_{i\geq 1} \sum_{z \in (n_{k-1},n_{k+1})} P^\zeta_{m_k} \left( \sigma^\zeta_{n_{k-1}} \wedge \tau^\zeta_{n_{k+1}} \geq i, \, \zeta_{i-1} = z \right) P^\zeta_z \left( \zeta_1 < m_{k-1} \right) \nonumber \\
&\qquad \leq P^\zeta_{n_{k-1}+1} \left( \zeta_1 < m_{k-1} \right) E^\zeta_{m_k}\left[ \sigma^\zeta_{n_{k-1}} \wedge \tau^\zeta_{n_{k+1}} \right] \label{undershoot1}
\end{align}
For the first term on the right in \eqref{undershoot1}, note that given $\zeta_0 = n_{k-1}+1$ then $\zeta_1$ has the same distribution as $\mathfrak{R}_{\tau^\mathfrak{B}_{n_{k-1}+2}} = \D_{\tau^\mathfrak{B}_{n_{k-1}+2}} + n_{k-1}+2$. Therefore, it follows from Lemma \ref{Lemma1} that for $k$ sufficiently large there are constants $c,C>0$ such that 
\begin{equation}\label{undershoot2}
 P^\zeta_{n_{k-1}+1} \left( \zeta_1 < m_{k-1} \right)
 = P\left( \D_{\tau^\mathfrak{B}_{n_{k-1}+2}} \leq -\fl{n_{k-1}^{2/3}} - 2 \right) \leq C e^{-c n_{k-1}^{1/3}}. 
\end{equation}
It remains only to prove that the expected value in \eqref{undershoot1} is bounded above by $C n_k$. To this end, it is enough to show that there exists a $c \in (0,1)$ and $k_2\in \N$ such that
\begin{equation}\label{noescape}
 P_{m_k} \left( \sigma^\zeta_{n_{k-1}} \wedge \tau^\zeta_{n_{k+1}} > \ell n_k \right) \leq (1-c)^\ell, \quad \text{for all } k\geq k_2 \text{ and } \ell \geq 1. 
\end{equation}
To see this, note that the Markov property implies that 
\begin{equation}\label{noescape1}
  P_{m_k}^\zeta \left( \sigma^\zeta_{n_{k-1}} \wedge \tau^\zeta_{n_{k+1}} > \ell n_k \right) \leq \left( \max_{z \in (n_{k-1},n_{k+1})} P_z^\zeta\left( \sigma^\zeta_{n_{k-1}} \wedge \tau^\zeta_{n_{k+1}} > n_k \right)   \right)^\ell. 
\end{equation}
If we let $z_k \in (n_{k-1},n_{k+1})$ be one of the maximizers in the right side, then by taking subsequences as $k\to\infty$ so the right side achieves the limsup and then taking a further subsequence so that $z_k/n_k \to y \in [1/2,2]$ we can apply the diffusion approximation of the BLP $\zeta$ by $Z^{(2-2\gamma)}$ to obtain that 
\begin{align*}
& \limsup_{k\to\infty}  \max_{z \in (n_{k-1},n_{k+1})} P_z^\zeta\left( \sigma^\zeta_{n_{k-1}} \wedge \tau^\zeta_{n_{k+1}} > n_k \right) 
 \\
 &\qquad \leq \sup_{y \in [1/2,2]} P_y^{Z^{(2-2\gamma)}}\left( \sigma_{1/2}^{Z^{(2-2\gamma)}} \wedge \tau_2^{Z^{(2-2\gamma)}} > 1 \right) < 1. 
\end{align*}
Together with \eqref{noescape1} this is enough to prove \eqref{noescape}. 
\end{proof}

\section{Proofs of technical results - asymptotically free case}\label{app:afc}

\noindent(i) \textbf{Process level tightness of extrema.}
\begin{proof}[Proof of Proposition \ref{minmaxtight}.]
 We will only give the proof for the tightness of the sequence $\mathcal{I}^X_n$ since the proof is similar for $\mathcal{S}^X_n$. 
It is enough to show that 
\[
 \lim_{\delta \to 0} \limsup_{n\to\infty} P\left( \sup_{\substack{k,\ell \leq nt \\|k-\ell|\leq n\delta}} |I^X_k - I^X_\ell| \geq 2 \epsilon \sqrt{n} \right) = 0, \quad \forall \epsilon>0. 
\]
To this end, note that if the running minimum decreases by at least $2 \epsilon \sqrt{n}$ in less than $n\delta$ steps, then there must be some interval $[-m\fl{\epsilon\sqrt{n}}, -(m-1)\fl{\epsilon\sqrt{n}} ]$ with $m\geq 1$ which the random walk crosses from right to left in less than $n\delta$ steps.
Moreover, note that this interval must have $m \leq \lceil t/\delta \rceil$ since otherwise it will take more than time $nt$ to cross all these intervals. 
Therefore,
\begin{multline*}
 P\left( \sup_{\substack{k,\ell \leq nt \\|k-\ell|\leq n\delta}} |I^X_k - I^X_\ell| \geq 2 \epsilon \sqrt{n} \right)
\leq \sum_{m=1}^{\lceil t/\delta \rceil } P\left( \tau_{-m\fl{\epsilon\sqrt{n}},1}  - \tau_{-(m-1)\fl{\epsilon\sqrt{n}},1}  \leq \delta n \right) \\
 \leq 
\sum_{m=1}^{\lceil t/\delta \rceil } P\left( 2 \sum_{x=-m\fl{\epsilon \sqrt{n}} + 1}^{-(m-1)\fl{\epsilon\sqrt{n}}} \mathcal{E}^{\tau_{-m\fl{\epsilon\sqrt{n}},1} }(x) \leq \delta n \right) =
\lceil t/\delta \rceil  P^\zeta_0 \left( 2 \sum_{i=1}^{\fl{\epsilon\sqrt{n}}} \zeta_i \leq \delta n \right), 
\end{multline*}
where the last equality follows from \eqref{zetaE}. 
 By Proposition \ref{prop:grkaf}(1), this upper bound converges  to 
$\lceil t/\delta \rceil P\left( \int_0^1 Z^{(2-2\gamma)}(t) \, dt < \frac{\delta}{2\epsilon^2} \mid Z^{(2-2\gamma)}(0) = 0 \right)$ as $n\to\infty$. 
To finish we need to prove that this expression vanishes as
$\delta \to 0$, and this is accomplished by the following lemma.
\begin{lemma}
 There exist constants $C,c>0$ (depending only on $\gamma$) such that 
\[
 P\left( \int_0^1 Z^{(2-2\gamma)}(t) \, dt < x \mid Z^{(2-2\gamma)}(0) = 0 \right) \leq C e^{-c x^{-1/2}}, \quad \forall x>0. 
\]
\end{lemma}

\begin{proof}
By the scaling property of BESQ processes we have for any $a>0$ that
\[
 \int_0^1 Z^{(2-2\gamma)}(t)\, dt = a^2 \int_0^{1/a} \frac{Z^{(2-2\gamma)}(at)}{a} \, dt \overset{\text{Law}}{=} a^2 \int_0^{1/a} Z^{(2-2\gamma)}(t) \, dt,  
\]
where the last equality in law is when the Bessel processes on both sides are started at $Z^{(2-2\gamma)}(0) = 0$. 
Applying this with $a = \sqrt{x}$ gives
\begin{align*}
& P\left( \int_0^1 Z^{(2-2\gamma)}(t) \, dt < x \mid Z^{(2-2\gamma)}(0) = 0 \right)
= P\left( \int_0^{x^{-1/2}} Z^{(2-2\gamma)}(t) \, dt < 1 \mid Z^{(2-2\gamma)}(0) = 0 \right) \\
&\hspace{1in} \leq P \left( \int_{i-1}^i Z^{(2-2\gamma)}(t) \, dt < 1, \, \text{for } 1\leq i \leq \fl{x^{-1/2}} \mid Z^{(2-2\gamma)}(0) = 0 \right)\\
&\hspace{1in} \leq P\left( \int_0^1 Z^{(2-2\gamma)}(t) \, dt < 1 \mid Z^{(2-2\gamma)}(0) = 0 \right)^{\fl{x^{-1/2}}}, 
\end{align*}
where the last equality follows from the Markov property and the monotonicity of BESQ processes with respect to their initial condition. 
The proof of the lemma is then finished by noting that $P_0\left( \int_0^1 Z^{(2-2\gamma)}(t) \, dt < 1 \right) < 1$. 
\end{proof}
The proof of Proposition~\ref{minmaxtight} is now complete.
\end{proof}

\noindent(ii) \textbf{Control of the number of rarely visited sites.}
The proof of Lemma \ref{lem:smalllt} below follows a strategy
that is similar to the proofs of analogous statements in
\cite{kmLLCRW,kpERWPCS,kmpCRW} but is somewhat simpler.  Recall that
$\mathcal{L}(x,n-1) = \mathcal{E}^n(x-1) + \mathcal{E}^n(x) +
\ind{X_n<x\leq 0} - \ind{0<x\leq X_n}$ and that the ${\mathcal E}$
processes are related to the BLPs $\zeta$ and $\tilde\zeta$ as in
\eqref{zetaE} and \eqref{tzetaE}. Hence, the control of the number of
sites with small local times in Lemma \ref{lem:smalllt} can be
obtained by first proving similar results for the BLPs $\zeta$ and
$\tilde\zeta$.

Since the process $\tilde\zeta$ is absorbing at zero, we will only
need to bound the number of times it goes below $M$ before hitting
zero for the first time.  To this end, note that if the event
$\left\{ \sum_{i=0}^{\sigma_0^{\tilde\zeta}-1} \ind{\tilde\zeta_i \leq
    M} \geq k \right\}$ occurs, then the first $k-1$ times the process
$\tilde\zeta$ enters $[0,M]$ its next step must necessarily not be to
zero. Using the monotonicity of the process $\zeta$ with respect to
the initial condition, we have that
\begin{equation}\label{tz-slt}
 \sup_{m\geq 1} P_m^{\tilde{\zeta}}\left( \sum_{i=0}^{\sigma_0^{\tilde\zeta}-1} \ind{\tilde\zeta_i \leq M} \geq k \right) 
 \leq 
 \left( 1-P^{\tilde\zeta}_M(\tilde\zeta_0 = 0) \right)^{k-1}, 
\end{equation}
and note that $P^{\tilde\zeta}_M(\tilde\zeta_0 = 0)>0$ so upper bound in \eqref{tz-slt} decays exponentially in $k$ for $M$ fixed.

Since the process $\zeta$ is not absorbing at zero, we will need to
control the time the process $\zeta$ spends below $M$ up to a fixed
time. The following lemma will be sufficient for our purposes.

\begin{lemma}\label{lem:zetasmall}
Let $w(\cdot)$ be as in \eqref{w} with $\alpha=0$.
Then for any $M,K>0$ and $b > \frac{\gamma\vee 0}{2} =: \frac{\gamma_+}{2}$
there are constants $C,c,r>0$ (depending on $K,M$ and $b$) such that
\[
\sup_{m\geq 0} P^{\zeta}_m \left( \sum_{i\leq K \sqrt{n} } \ind{\zeta_i \leq M} > n^b \right)  \leq C e^{-c n^r}, \qquad \forall n\geq 1. 
\]
\end{lemma}

\begin{proof}
Since the process $\zeta$ is monotone with respect to the initial condition, it is enough to prove the upper bound only for the case $\zeta_0 = 0$. 
We begin by fixing some $d \in (\frac{\gamma_+}{2},b)$. 
Now, if the event $\{ \sum_{i\leq K\sqrt{n} } \ind{\zeta_i \leq M} \geq n^b \}$ occurs then either (1) the process $\zeta$ returns to 0 at least $\fl{n^d}$ times in the first $K\sqrt{n}$ steps of the Markov chain or 
(2) in one of the first $\fl{n^d}$ excursions from 0 the process spends at least $n^{b-d}$ steps below $M$. 
Therefore, it follows 
that
\begin{align*}
& P^{\zeta}_0 \left( \sum_{i\leq K\sqrt{n}} \ind{\zeta_i \leq M} \geq n^b \right)  \\
&\qquad \leq P^{\zeta}_0 \left( \sum_{i\leq K\sqrt{n}} \ind{\zeta_i = 0} \geq n^d \right) + n^d P^{\zeta}_0 \left( \sum_{i=0}^{\sigma^\zeta_0 -1} \ind{\zeta_i \leq M} \geq n^{b-d} \right) \\
&\qquad \leq \left( 1- P^\zeta_0( \sigma^\zeta_0 \geq K\sqrt{n} ) \right)^{\fl{n^d}} + n^d  \left( 1 - P^{\zeta}_M(\zeta_1 = 0) \right)^{n^{b-d}-1}, 
\end{align*}
where the second term in the last inequality follows from a similar argument as the one preceding \eqref{tz-slt}.
We will handle the first term in the last line differently depending on whether $\gamma\in [0,1)$ or $\gamma<0$.
If $\gamma \in[0,1)$, then choosing an $\epsilon>0$ small enough so that
$\frac{\gamma+\epsilon}{2} < d$ it follows from Lemma \ref{lem:sztail} that 
$ \left( 1- P^\zeta_0( \sigma^\zeta_0 \geq K\sqrt{n} ) \right)^{\fl{n^d}} \leq C' e^{-c' n^{d-\frac{\gamma+\epsilon}{2}}}$ for some $C',c'>0$ (depending on $K$ and $\epsilon$).
On the other hand, if $\gamma < 0$ then it follows from Lemma \ref{lem:noreturn} that $\left( 1- P^\zeta_0( \sigma^\zeta_0 \geq K\sqrt{n} ) \right)^{\fl{n^d}}
\leq \left( 1- P^\zeta_0( \sigma^\zeta_0 = \infty ) \right)^{\fl{n^d}} \leq (1-c)^{\fl{n^d}}$ for some $c>0$. 
In either case, these bounds on the first term are enough to complete the proof. 
\end{proof}

\begin{proof}[Proof of Lemma \ref{lem:smalllt}]
First of all, note that $\mathcal{L}(x,k-1) \geq \max\{ \mathcal{D}^k(x), \mathcal{E}^k(x)\}$. Therefore, 
\begin{align*}
  &P\left( \sup_{k\leq nt} \sum_{x \in [I_{k-1}^X, S_{k-1}^X]} \ind{\mathcal{L}(x,k-1) \leq M } \geq 4 n^{b} \right) \\
&\leq P\left( \sup_{k\leq nt}  \sum_{x \in [X_{k}, S_{k}^X]} \ind{\mathcal{E}^k(x) \leq M } \geq  2 n^b \right) 
+ P\left( \sup_{k\leq nt}  \sum_{x \in [I_{k}^X, X_{k}]} \ind{\mathcal{D}^k(x) \leq M} \geq  2 n^b \right). 
\end{align*}
We will show only that the first probability in the last line vanishes as $n\to\infty$ since the second probability can be handled by a symmetric argument. 

Recall that 
  $\tau_{z,m} = \min\{ n\geq 0: \mathcal{L}(x,n) = m \} $ denotes the
time that the random walk reaches site $z\in \Z$ for the $m$-th time.
Note that if $\max_{k\leq nt} |X_k| \leq K\sqrt{n}$ then for every
$k\leq nt$ we have $\tau_{z,m} = k$ for some $|z|\leq K\sqrt{n}$ and
$m\leq n/2$.  Since $\max_{k\leq nt} |X_k|$ is tight (see
\cite[Corollary 1A]{tGRK} or Proposition \ref{minmaxtight} above),
given any $\epsilon>0$ we can choose $K$ large enough so that
$P(\max_{k\leq nt} |X_k| > K \sqrt{n}) < \epsilon$ and thus
\begin{align*}
& P\left( \sup_{k\leq nt}  \sum_{x \in [X_{k}, S_{k}^X]} \ind{\mathcal{E}^k(x) \leq M } \geq  2 n^b \right) \\
&\qquad \leq \epsilon + \sum_{\substack{|z|\leq K\sqrt{n}\\1\leq m \leq n/2 \\ k\leq nt }} P\left( \tau_{z,m}=k , \,  \sum_{x \in [z, S_k^X]} \ind{\mathcal{E}^{\tau_{z,m}}(x) \leq M } \geq  2 n^b \right) \\
&\qquad \leq \epsilon + \sum_{\substack{|z|\leq K\sqrt{n}\\1\leq m \leq n/2\\ k\leq nt }} \left\{ P\left( \sum_{x \in [z\wedge 0, -1]} \ind{\mathcal{E}^{\tau_{z,m}}(x) \leq M } \geq   n^b \right) \right. \\
&\hspace{2in} \left.
+ P\left( \sum_{x \geq z\vee 0 } \ind{0 < \mathcal{E}^{\tau_{z,m}}(x) \leq M } \geq   n^b -1 \right) \right\}.
\end{align*}
Since the process $\{ \mathcal{E}^{\tau_{z,m}}(x) \}_{x \in [z\wedge 0, -1]}$ is distributed like the BLP $\zeta$ and the process $\{\mathcal{E}^{\tau_{z,m}}(x) \}_{x \geq z \vee 0}$ is distributed like the BLP $\tilde\zeta$ 
(both with random initial conditions), we have that this last sum is bounded above by 
\begin{align*}
(K+1) t n^{5/2} \left\{ \sup_{m\geq 0} P^\zeta_m\left( \sum_{i\leq K\sqrt{n}} \ind{\zeta_i \leq M } \geq  n^b \right) + \sup_{m\geq 1} P_m^{\tilde\zeta} \left( \sum_{i=0}^{\sigma_0^{\tilde\zeta} -1 } \ind{\tilde\zeta_i \leq M } \geq  n^b  -1 \right) \right\}.
\end{align*}
It then follows from \eqref{tz-slt} and Lemma \ref{lem:zetasmall} that
this last expression vanishes as $n\to\infty$ for any fixed $K,M,t>0$
and $b > \frac{\gamma_+}{2}$.  Since $\epsilon>0$ was arbitrary, this
completes the proof of the lemma.
\end{proof}

\section{Proofs of technical results - polynomially self-repelling case}\label{app:psr}

\begin{proof}[Proof of Lemma~\ref{cont}]
We start with the proof of the $P^{W_\alpha}$-a.s.\ continuity
    of $G_{\delta,M}$. Fix $\delta>0$ and $M\ge 0$. Let ${\mathcal C}$ be
    a subset of $C([0,\infty))$ which consists of functions $\omega$
    such that
\begin{enumerate}[(i)]
\item $\text{meas}\{t\ge 0:\omega(t)\in[0,\delta]\}=\infty$;
\item for $\ell\in \{M,M+1\}$, ${\mathcal T}_{\delta,\ell-}(\omega)={\mathcal T}_{\delta,\ell}(\omega)$;
\item $\text{meas}\{t\ge 0:\omega(t)\in\{0,\delta,1\}\}=0$, 
\end{enumerate}
where $\text{meas}(\cdot)$ denotes Lebesgue measure.
Property (i) guarantees that ${\mathcal T}_{\delta,\ell}(\omega)<\infty$
for all $\ell\ge 0$. Property (ii) ensures that
${\mathcal T}_{\delta,\ell}(\omega)$ is continuous in $\ell$ at
$\ell\in\{M,M+1\}$ (it is right-continuous in $\ell$ by the
definition). Observe that $P^{W^\alpha}({\mathcal C})=1$.

{\em Step 1.} Due to (i)--(iii), for $\ell\in\{M,M+1\}$, the functional
${\mathcal T}_{\delta,\ell}$ is continuous on $D([0,\infty))$ at every
$\omega\in{\mathcal C}$. More precisely, given
$\omega\in{\mathcal C}$ and $\ell\in\{M,M+1\}$, put
$R(\ell,\omega)={\mathcal T}_{\delta,\ell}(\omega)+1$. Then
$\forall \epsilon>0$ there is a $\lambda>0$ (depending on $\epsilon, \delta$, and $\omega$) such that for each
$\omega'\in D([0,\infty))$ satisfying
$\sup_{t\le R(\ell,\omega)}|\omega'(t)-\omega(t)|<\lambda$, we have
$|{\mathcal T}_{\delta,\ell}(\omega)-{\mathcal
  T}_{\delta,\ell}(\omega')|<\epsilon$.

{\em Step 2.} Fix $\epsilon>0$, $\omega\in{\mathcal C}$, and let $R=R(M+1,\omega)$. Then for every $\omega'\in D([0,\infty))$ such that
$\sup_{t\le R}|\omega'(t)-\omega(t)|<\lambda$ where $\lambda$ is sufficiently small we get
\begin{align*}
  \left|G_{\delta,M}(\omega)-G_{\delta,M}(\omega')\right|\le &\int_{{\mathcal T}_{\delta,M}(\omega)}^{{\mathcal T}_{\delta,M+1}(\omega)}\left|\indf{[0,1]}(\omega(t))-\indf{[0,1]}(\omega'(t))\right|dt\\ & \makebox[1.5cm]{\ }+|{\mathcal T}_{\delta,M}(\omega)-{\mathcal T}_{\delta,M}(\omega')|+|{\mathcal T}_{\delta,M+1}(\omega)-{\mathcal T}_{\delta,M+1}(\omega')|\\ \overset{\textit{Step 1}}{\le}&\int_{{\mathcal T}_{\delta,M}(\omega)}^{{\mathcal T}_{\delta,M+1}(\omega)}\left|\indf{[0,1]}(\omega(t))-\indf{[0,1]}(\omega'(t))\right|dt+2\epsilon.
\end{align*}
We are left to estimate the first term in the right hand side. It is equal to
\begin{align*} 
&\int_{{\mathcal T}_{\delta,M}(\omega)}^{{\mathcal T}_{\delta,M+1}(\omega)}\left|\indf{[0,1]}(\omega(t))-\indf{[0,1]}(\omega'(t))\right|\indf{[-\lambda,\lambda]\cup[1-\lambda,1+\lambda]}(\omega(t))\,dt \\ 
&\le \int_{{\mathcal T}_{\delta,M}(\omega)}^{{\mathcal T}_{\delta,M+1}(\omega)}\indf{[-\lambda,\lambda]\cup[1-\lambda,1+\lambda]}(\omega(t))\,dt\le \int_0^R\indf{[-\lambda,\lambda]\cup[1-\lambda,1+\lambda]}(\omega(t))\,dt\overset{(iii)} <\epsilon, 
\end{align*}
for sufficiently small $\lambda$. This finishes the proof of the $P^{W^\alpha}$-a.s.\ continuity of $G_{\delta,M}$.

\smallskip

We recall that $\tau^{W_\alpha}_1=\inf\{t> 0: W^\alpha(t)=1\}$ and
partition $\Omega$ into three events
\[A_1=\{{\mathcal T}_{\delta,M}\ge \tau^{W_\alpha}_1\},\ \
  A_2=\{\tau^{W_\alpha}_1\ge {\mathcal T}_{\delta,M+1}\},\ \
  A_3=\{{\mathcal T}_{\delta,M}<\tau^{W_\alpha}_1<{\mathcal
    T}_{\delta,M+1}\}.\] We note that on $A_1$ we can replace
$W_\alpha$ with the corresponding multiple $B^r_\alpha$ of reflected
Brownian motion, which starts at a random point $x\in[0,\delta]$ and
is reflected at $0$ and $1$, and integrate from $0$ to
${\mathcal T}_{\delta,1}$ (since for $B^r_\alpha$ the functional
$G_{\delta,M}$ depends on $M$ only through the value of
$B^r_\alpha({\mathcal T}_{\delta,M})\in[0,\delta]$). This gives us the
following estimate:
      \begin{equation}
        \label{a1}
        E^{W_\alpha}\left[(G_{\delta,M})^p;A_1\right]\le \max_{x\in[0,\delta]}E_x^{B^r_\alpha}\left[(G_{\delta,0})^p\right]\le \max_{x\in[0,1]}E_x^{B^r_\alpha}\left[({\mathcal T}_{\delta,1})^p\right].
      \end{equation}
      Note that this bound does not depend on $M\ge 0$. On
      $A_2$, we simply integrate from $0$ to $\tau^{W_\alpha}_1$ and
      then use the Ray--Knight theorem for $W_\alpha$:
      \begin{align}
        E^{W_\alpha}\left[(G_{\delta,M})^p;A_2\right]
        &\le E\left[\left(\int_0^{\tau^{W_\alpha}_1}\indf{[0,1]}(W_\alpha(s))\,ds\right)^p\right] \nonumber \\
        &=E\left[\left(\int_0^1Z^{(\alpha,1)}(x)\,dx\right)^p\right]=:C_{\alpha,p},         \label{a2}
      \end{align}
      a constant which does not depend on $M$ or $\delta$. Finally, on
      $A_3$ we have
      \[\int_{{\mathcal T}_{\delta,\ell}}^{{\mathcal
            T}_{\delta,\ell+1}}\indf{[0,1]}(W_\alpha(s))\,ds\le
        \left(\int_0^{\tau^{W_\alpha}_1}+\int_{\tau^{W_\alpha}_1}^{{\mathcal
              T}_{\delta,M+1}}\right)\indf{[0,1]}(W_\alpha(s))\,ds.\]
      Using the elementary inequality $(a+b)^p\le 2^{p-1}(a^p+b^p)$,
      \eqref{a2}, and the estimate similar to \eqref{a1} to treat the last
      integral we get that for all $\delta\in(0,1/2]$ and $M\ge 0$
\begin{equation}
  \label{a3}
  E^{W_\alpha}\left[(G_{\delta,M})^p;A_3\right]\le 2^{p-1}\left(C_{\alpha,p}+\max\limits_{x\in[0,1]}E_x^{B^r_\alpha}\left[({\mathcal T}_{\delta,1})^p\right]\right).
\end{equation}
Collecting \eqref{a1}-\eqref{a3} together we conclude that 
\begin{equation}
  \label{a123}
  \sup\limits_{M\ge 0,\,
        \delta\in(0,\delta_0]}E^{W_\alpha}\left[(G_{\delta,M})^p\right]\le (2^{p-1}+1)\sup\limits_{\delta\in(0,\delta_0]}\left(C_{\alpha,p}+\max\limits_{x\in[0,1]}E_x^{B^r_\alpha}\left[({\mathcal T}_{\delta,1})^p\right]\right).
\end{equation}
We just need to get a bound on
$\max\limits_{x\in[0,1]}E_x^{B^r_\alpha}\left[({\mathcal
    T}_{\delta,1})^p\right]$ uniformly over all small
$\delta>0$. Since $L^{B^r_\alpha}_t(0)\to \infty$ as $t\to\infty$ with
probability 1, there is a $t>0$ such that
\[P_0^{B^r_\alpha}\left(L_t(0)>2\right)\ge \frac34,\ \ \text{which implies} \
  P_0^{B^r_\alpha}\left(\tau^{L(0)}_2\le t\right)\ge \frac34.\] Since
$\mathcal{T}_{\delta,1} \Longrightarrow \tau_1^{L(0)}$ as
$\delta\to 0$, we may choose $\delta_0>0$ so that
$P_0^{B^r_\alpha}\left({\mathcal T}_{\delta,1}<\tau^{L(0)}_2\right)\ge
3/4$ for all $\delta\in(0,\delta_0]$. This will imply that
\[P^{B^r_\alpha}_0\left({\mathcal T}_{\delta,1}\le t\right)\ge
  P^{B^r_\alpha}_0\left({\mathcal T}_{\delta,1}<\tau^{L(0)}_2\le
    t\right)\ge \frac12\ \ \text{for all }\delta\in(0,\delta_0].\] To
move the starting point from $0$ to $x$ we note that by the strong
Markov property,
\[P^{B^r_\alpha}_x\left({\mathcal T}_{\delta,1}\le
    t+1\right)\ge
  P^{B^r_\alpha}_1\left(\tau_0<1\right)P^{B^r_\alpha}_0\left({\mathcal
      T}_{\delta,1}\le
    t\right)\ge\frac12P_1^{B^r_\alpha}(\tau_0<1)=:c>0,
        \] 
for all $x \in [0,1]$, and thus again by the strong Markov property,
\[\max\limits_{x\in[0,1]}P^{B^r_\alpha}_x\left({\mathcal T}_{\delta,1}>
    n(t+1)\right)\le (1-c)^n,\ \ \text{for all }n\in\N.\] This
immediately gives the desired upper bound and completes the proof.
    \end{proof}

\begin{lemma}\label{calc}
  Let $Z^{(\alpha,0)}$ be the stochastic process defined by \eqref{Za0} which starts at $s>0$ and is absorbed upon hitting $0$. Then for all $y,\alpha\in[0,\infty)$, 
\[
 E\left( \int_0^y Z^{(\alpha,0)}(x)\,dx\right)=ys,\quad \Var{\left(\int_0^y Z^{(\alpha,0)}(x)\,dx\right)}=\frac{2y^3s}{3(1+2\alpha)}.
\]
\end{lemma}

\begin{proof}
  The first statement is immediate from \eqref{Za0}. Integration by parts gives
\[
 \int_0^y Z^{(\alpha,0)}(x)\,dx-ys=\frac{1}{\sqrt{2\alpha+1}}\int_0^y(y-x)\sqrt{2Z^{(\alpha,0)}(x)}\,dB(x).
\]
By It\^o's isometry and again \eqref{Za0} we get that
\[
 \Var\left(\int_0^y Z^{(\alpha,0)}(x)\,dx\right) =\frac{2}{2\alpha+1}\int_0^y(y-x)^2E\left[ Z^{(\alpha,0)}(x) \right]dx=\frac{2y^3s}{3(1+2\alpha)}.
\]

\end{proof}

\begin{proof}[Proof of Lemma~\ref{elenatype}]
  Without loss of generality we can assume that $M\in\N$ and $y\ge 1$. Let
  \begin{align*}
  A^+&:=\bigg\{\sup_{\tau^{\B}_{Mn} \le i \leq \tau^{\B}_{(M+1)n}}
    (\D_i - \D_{\tau^{\B}_{Mn}})\ge y\sqrt{n}\bigg\}\\
  \text{and} \qquad A^-&:=\bigg\{\inf_{\tau^{\B}_{Mn} \le i \leq \tau^{\B}_{(M+1)n}}
    (\D_i - \D_{\tau^{\B}_{Mn}})\le -y\sqrt{n}\bigg\},
  \end{align*}
  and note that we
  need to bound $P(A^+ \cup A^- )$. Since
  \[P\left(A^+ \cup A^- \right)\le P\left(\left(A^+ \cup
      A^- \right)\cap\left\{|\D_{\tau^{\B}_{Mn}}|\le
      y\sqrt{Mn}\right\}\right)+P\left(|\D_{\tau^{\B}_{Mn}}|>
      y\sqrt{Mn}\right), \] and since the last probability is bounded as
  necessary by Lemma~\ref{Lemma1}, we only need to estimate the
  first term on the right hand side of the above inequality. We have
\begin{multline*}
P\left(\left(A^+ \cup A^- \right)\cap\left\{|\D_{\tau^{\B}_{Mn}}|\le y\sqrt{Mn}\right\}\right) \\
\le
    P\left(A^+ \, \big|\,|\D_{\tau^{\B}_{Mn}}|\le
    y\sqrt{Mn}\right)+P\left(A^- \,\big|\,|\D_{\tau^{\B}_{Mn}}|\le y\sqrt{Mn}\right).
\end{multline*}
  Due to the self-repelling property of
  the model, the probability of $A^+ $ is a non-increasing function of
  $\D_{\tau^{\B}_{Mn}}$ (recall that the discrepancy is the number of
  red balls drawn minus the number of blue balls drawn, and the more
  red balls drawn - the smaller the probability to draw a red ball in
  the future). Similarly, the probability of $A^- $ is a
  non-decreasing function of $\D_{\tau^{\B}_{Mn}}$. Therefore,
  \begin{multline*}
    P\left(\left(A^+ \cup A^- \right)\cap\left\{|\D_{\tau^{\B}_{Mn}}|\le y\sqrt{Mn}\right\}\right)\le \\ P\left(A^+ \,
    \big|\,\D_{\tau^{\B}_{Mn}}=-\fl{ y\sqrt{Mn}}\right)+P\left(A^- \,\big|\,\D_{\tau^{\B}_{Mn}}=
    \fl{y\sqrt{Mn}}\right).
  \end{multline*}
 The estimates of the last two probabilities
  are very similar, so we shall only bound the first of them.

  We think of $D=(\D_i)_{\tau^{\B}_{Mn}\le i\le \tau^{\B}_{(M+1)n}}$ as
  a nearest neighbor walk which goes up each time a red ball is drawn
  and down if a blue ball is drawn. The walk starts at
  $-\fl{y\sqrt{Mn}}$ at time $\tau^{\B}_{Mn}$.
  We note that every
  $D$-walk path in $A^+ $ must hit
  $a:=-\fl{ y\sqrt{Mn}}+\fl{y\sqrt{n}}$ in no more than
  $2(n-1)+\fl{y\sqrt{n}}$\footnote{At the hitting time of $a$ the
    number of up steps taken is equal to the number of down steps plus
    $\fl{y\sqrt{n}}$ but due to the restriction imposed by $A^+ $ the
    number of down steps taken is at most $n-1$, so the total number
    of steps taken up to this hitting time does not exceed
    $2(n-1)+\fl{y\sqrt{n}}$.} steps, and thus we shall estimate the
  probability of the ``larger event''
  \[A^+_1:=\bigg\{\sup_{\tau^{\B}_{Mn} \le i \leq \tau^{\B}_{Mn}+2n+\fl{y\sqrt{n}}}
    (\D_i - \D_{\tau^{\B}_{Mn}})\ge
    y\sqrt{n}\bigg\}\supset A^+.\]
  To this end, we write
  \begin{align}
  &P(A^+_1 \mid \D_{\tau^{\B}_{Mn}}=-\fl{ y\sqrt{Mn}}) \nonumber \\ 
    &\le  P\left(A^+_1\cap (B^-)^c \mid \D_{\tau^{\B}_{Mn}}=-\fl{y\sqrt{Mn}}\right)
   +P(B^- \mid \D_{\tau^{\B}_{Mn}}=-\fl{
      y\sqrt{Mn}}),\label{plus}
  \end{align}
  where
  \[
B^-:=\bigg\{\inf_{\tau^{\B}_{Mn} \le i \leq \tau^{\B}_{Mn}+2n+\fl{y\sqrt{n}}}
    (\D_i - \D_{\tau^{\B}_{Mn}})\le -\fl{y\sqrt{Mn}}\bigg\}.\]
 We will bound probabilities in \eqref{plus} by coupling the $D$-walk with appropriate simple random walks, and, thus, in the remainder of the proof we will denote a simple random walk which starts at 0 and steps to the right with probability $q \in (0,1)$ on each step by $\mathcal{S}_q = \{S_{q,n}\}_{n\geq 0}$. 
For the second probability in \eqref{plus},
note that as long as $\D_i\le 0$ the probability that the $(i+1)$-th step
is up is at least $1/2$. 
Thus, we can couple the shifted $D$-walk $(\D_i + \fl{y\sqrt{Mn}})_{i\geq 1}$ with a simple symmetric random walk $\mathcal{S}_{1/2}$ so that the shifted $D$-walk is bounded below by  $\mathcal{S}_{1/2}$ until the simple random walk  $\mathcal{S}_{1/2}$ goes above $ \fl{y\sqrt{Mn}}$. 
From this we obtain that 
\begin{align*}
P\left( B^- \mid \D_{\tau^{\B}_{Mn}}=-\fl{y\sqrt{Mn}} \right)
&\leq 
P\left( \max_{k \leq 2n + \fl{y \sqrt{n}}} |S_{1/2,k}| \geq \fl{y \sqrt{Mn}} \right) \\
&\leq 4 P\left( S_{1/2,3n} \geq \fl{y \sqrt{Mn}} \right)
\end{align*}
where the second inequality follows from the reflection principle and
the assumption that $y\leq \sqrt{n}$.  Finally, applying Hoeffding's
inequality to the last probability on the right we obtain that there
exist constants $M_0,c_2 \in (0,\infty)$ such that
$P\left( B^- \mid \D_{\tau^{\B}_{Mn}}=-\fl{y\sqrt{Mn}} \right) \leq
\frac{1}{c_2} e^{-c_2 y^2}$
for all $M\geq M_0$, $n\geq 1$ and $y \leq \sqrt{n}$.

 It is left to bound the first term in \eqref{plus}.
To this end, we first note that given $\D_i$, the probability that the $\D_{i+1}-\D_i$ (i.e., the $D$-walk takes a jump up) is equal to 
\[
 \frac{(2\B_i+2)^\alpha}{(2\B_i+2)^\alpha+(2(\B_i+\D_i )+1)^\alpha} = \frac{1}{1 + \left( 1 + \frac{\D_i - \frac{1}{2}}{\B_i + 1}  \right)^\alpha}. 
\]
Since for every $D$-walk path in the event $A^+_1 \cap (B^-)^c$ we have $\D_i \geq -2\fl{y \sqrt{Mn}}$ and $\B_i \geq Mn$, we have that the probability of jumping up can always be bounded above by 
\[
 \frac{1}{1 + \left( 1 - \frac{2\fl{y \sqrt{Mn}} + \frac{1}{2}}{Mn}  \right)_+^\alpha}
\leq  \frac{1}{2}+\frac{k_0y}{\sqrt{Mn}}=:q_{y}, 
\]
where $k_0>0$ is a constant depending only on $\alpha$ (and not on $M$, $n$ or $y \in [1,\sqrt{n}]$)\footnote{Note that the constant $q_y$ depends on the parameters $\alpha$, $M$ and $n$ as well, but we emphasize only the dependence on $y$ because our final bound will be depending only on $y$ and uniform over $M\geq M_0$ and $n\geq 1$.}.
Therefore, we can easily couple the shifted $D$-walk $(\D_i + \fl{y\sqrt{Mn}})_{i\geq 1}$ with a simple random walk $\mathcal{S}_{q_y}$ so that the shifted $D$-walk is bounded above by  $\mathcal{S}_{q_y}$ until the shifted $D$-walk goes below $-\fl{y\sqrt{Mn}}$. 
Thus, we obtain that 
\begin{align*}
 P\left(A^+_1\cap (B^-)^c \mid \D_{\tau^{\B}_{Mn}}=-\fl{y\sqrt{Mn}}\right)
&\leq P\left( \max_{k\leq 2n + \fl{y \sqrt{n}}} S_{q_y,k} \geq y\sqrt{n} \right) \\
&\leq 2P\left( S_{q_y, 2n + \fl{y \sqrt{n}}} \geq y \sqrt{n} \right),
\end{align*}
where the second inequality follows again from a reflection principle and the fact that $q_y \geq \frac{1}{2}$.
If $M\geq (12 k_0)^2$ and $y\leq \sqrt{n}$ then $E[S_{q_y,2n+\fl{y\sqrt{n}}}] = \frac{2k_0 y}{\sqrt{Mn}}\left(  2n + \fl{y \sqrt{n}} \right) \leq \frac{y\sqrt{n}}{2}$, and we can conclude from Hoeffding's inequality that there is a $c_3>0$ such that
\begin{align*}
& P\left(A^+_1\cap (B^-)^c \mid \D_{\tau^{\B}_{Mn}}=-\fl{y\sqrt{Mn}}\right) \\
&\qquad \leq 2P\left( S_{q_y, 2n + \fl{y \sqrt{n}}} - E[S_{q_y, 2n+\fl{y\sqrt{n}}}] \geq \frac{y\sqrt{n}}{2} \right)
\leq \frac{1}{c_3} e^{-c_3 y^2}, 
\end{align*}
for all $M \geq (12 k_0)^2$, $n\geq 1$ and $y\leq \sqrt{n}$.
\end{proof}

\bibliographystyle{alpha}


\begin{thebibliography}{HLSH18}

\bibitem[Bil99]{bCOPM}
Patrick Billingsley.
\newblock {\em {Convergence of probability measures}}.
\newblock {Wiley Series in Probability and Statistics: Probability and
  Statistics}. John Wiley \& Sons, Inc., New York, second edition, 1999.
\newblock A Wiley-Interscience Publication.

\bibitem[CD99]{cdPUPBM}
L.~Chaumont and R.~A. Doney.
\newblock {Pathwise uniqueness for perturbed versions of {B}rownian motion and
  reflected {B}rownian motion}.
\newblock {\em Probab. Theory Related Fields}, 113(4):519--534, 1999.

\bibitem[CPY98]{cpyBetaPBM}
Philippe Carmona, Fr{\'e}d{\'e}rique Petit, and Marc Yor.
\newblock {Beta variables as times spent in {$[0,\infty[$} by certain perturbed
  {B}rownian motions}.
\newblock {\em J. London Math. Soc. (2)}, 58(1):239--256, 1998.

\bibitem[Dav90]{Dav90}
Burgess Davis.
\newblock Reinforced random walk.
\newblock {\em Probab. Theory Related Fields}, 84(2):203--229, 1990.

\bibitem[Dav96]{Dav96}
Burgess Davis.
\newblock {Weak limits of perturbed random walks and the equation
  {$Y_t=B_t+\alpha\sup\{Y_s\colon\ s\leq t\}+\beta\inf\{Y_s\colon\ s\leq
  t\}$}}.
\newblock {\em Ann. Probab.}, 24(4):2007--2023, 1996.

\bibitem[Dav99]{Dav99}
Burgess Davis.
\newblock Brownian motion and random walk perturbed at extrema.
\newblock {\em Probab. Theory Related Fields}, 113(4):501--518, 1999.

\bibitem[DK12]{dkSLRERW}
Dmitry Dolgopyat and Elena Kosygina.
\newblock {Scaling limits of recurrent excited random walks on integers}.
\newblock {\em Electron. Commun. Probab.}, 17:no. 35, 14, 2012.

\bibitem[Dol11]{dCLTERW}
Dmitry Dolgopyat.
\newblock {Central limit theorem for excited random walk in the recurrent
  regime}.
\newblock {\em ALEA Lat. Am. J. Probab. Math. Stat.}, 8:259--268, 2011.

\bibitem[EK86]{ekMP}
Stewart~N. Ethier and Thomas~G. Kurtz.
\newblock {\em {Markov processes}}.
\newblock {Wiley Series in Probability and Mathematical Statistics: Probability
  and Mathematical Statistics}. John Wiley \& Sons, Inc., New York, 1986.
\newblock Characterization and convergence.

\bibitem[GJY03]{gySGBP}
Anja G{\"o}ing-Jaeschke and Marc Yor.
\newblock {A survey and some generalizations of {B}essel processes}.
\newblock {\em Bernoulli}, 9(2):313--349, 2003.

\bibitem[HLSH18]{HLSH18}
Wilfried Huss, Lionel Levine, and Ecaterina Sava-Huss.
\newblock {Interpolating between random walk and rotor walk}.
\newblock {\em Random Structures Algorithms}, 52(2):263--282, 2018.

\bibitem[KM11]{kmLLCRW}
Elena Kosygina and Thomas Mountford.
\newblock {Limit laws of transient excited random walks on integers}.
\newblock {\em Ann. Inst. Henri Poincar{\'e} Probab. Stat.}, 47(2):575--600,
  2011.

\bibitem[KMP22]{kmpCRW}
Elena Kosygina, Thomas Mountford, and Jonathon Peterson.
\newblock Convergence of random walks with {M}arkovian cookie stacks to
  {B}rownian motion perturbed at extrema.
\newblock {\em Probab. Theory Related Fields}, 182(1-2):189--275, 2022.

\bibitem[KP16]{kpERWPCS}
Elena Kosygina and Jonathon Peterson.
\newblock {Functional limit laws for recurrent excited random walks with
  periodic cookie stacks}.
\newblock {\em Electron. J. Probab.}, 21:Paper No. 70, 24, 2016.

\bibitem[KZ13]{kzERWsurvey}
Elena Kosygina and Martin Zerner.
\newblock {Excited random walks: results, methods, open problems}.
\newblock {\em Bull. Inst. Math. Acad. Sin. (N.S.)}, 8(1):105--157, 2013.

\bibitem[Pet75]{pSOIRV}
V.~V. Petrov.
\newblock {\em {Sums of independent random variables}}.
\newblock Springer-Verlag, New York-Heidelberg, 1975.
\newblock Translated from the Russian by A. A. Brown, Ergebnisse der Mathematik
  und ihrer Grenzgebiete, Band 82.

\bibitem[PW97]{pwPBM}
Mihael Perman and Wendelin Werner.
\newblock {Perturbed {B}rownian motions}.
\newblock {\em Probab. Theory Related Fields}, 108(3):357--383, 1997.

\bibitem[RY99]{ryCMBM}
Daniel Revuz and Marc Yor.
\newblock {\em {Continuous martingales and {B}rownian motion}}, volume 293 of
  {\em {Grundlehren der Mathematischen Wissenschaften [Fundamental Principles
  of Mathematical Sciences]}}.
\newblock Springer-Verlag, Berlin, third edition, 1999.

\bibitem[T{\'o}t94]{tTSAWGBR}
B{\'a}lint T{\'o}th.
\newblock {``{T}rue'' self-avoiding walks with generalized bond repulsion on
  {${\bf Z}$}}.
\newblock {\em J. Statist. Phys.}, 77(1-2):17--33, 1994.

\bibitem[T{\'o}t95]{tTSAW}
B{\'a}lint T{\'o}th.
\newblock {The ``true'' self-avoiding walk with bond repulsion on {$\mathbf
  Z$}: limit theorems}.
\newblock {\em Ann. Probab.}, 23(4):1523--1556, 1995.

\bibitem[T{\'o}t96]{tGRK}
B{\'a}lint T{\'o}th.
\newblock {Generalized {R}ay-{K}night theory and limit theorems for
  self-interacting random walks on {${\bf Z}^1$}}.
\newblock {\em Ann. Probab.}, 24(3):1324--1367, 1996.

\end{thebibliography}

\end{document}